\documentclass[aop, preprint, 12pt]{imsart}

\RequirePackage[OT1]{fontenc}
\RequirePackage{amsthm,amsmath}
\RequirePackage[numbers]{natbib}
\RequirePackage[colorlinks,citecolor=blue,urlcolor=blue]{hyperref}
\usepackage{mathrsfs}
\usepackage{amsfonts}
\usepackage{color}
\usepackage{pifont}
\usepackage{amssymb}
\usepackage[english]{babel}
\usepackage[T1]{fontenc}
\usepackage{graphicx}
\usepackage{subfigure}
\usepackage{latexsym}
\usepackage{amstext}
\usepackage{mathrsfs}
\usepackage{hyperref}
\usepackage{dsfont}
\usepackage{graphics}
\usepackage{enumerate}
\usepackage{paralist}
\usepackage{color}
\usepackage{fullpage}
\newtheorem{prop}{Proposition}[section]

\newtheorem{lem}[prop]{Lemma}

\numberwithin{equation}{section}

\newcommand{\beq}{\begin{eqnarray}}
\newcommand{\beqq}{\begin{eqnarray*}}
\newcommand{\eeq}{\end{eqnarray}}
\newcommand{\eeqq}{\end{eqnarray*}}

\usepackage{mathtools}
\usepackage[english]{babel}
\usepackage[utf8]{inputenc}
\usepackage{amsmath}
\usepackage{graphicx}
\usepackage[colorinlistoftodos]{todonotes}
\usepackage{polynom}
\usepackage{amsfonts}
\usepackage{dsfont}
\usepackage{amsthm}
\usepackage{hyperref}
\usepackage{graphicx}
\newtheorem{theorem}{Theorem}[section]

\newtheorem{lemma}{Lemma}[section]

\newtheorem{corollary}[theorem]{Corollary}

\usepackage{fullpage}
\usepackage[T1]{fontenc}
\usepackage{hyperref}
\usepackage{xcolor}
\definecolor{link-color}{rgb}{0.15,0.4,0.15}
\hypersetup{
  colorlinks,
  linkcolor = {link-color}, citecolor = {link-color},
}

\usepackage{graphicx}
\usepackage{hyperref}
\usepackage{amsmath,amsthm,amssymb}
\usepackage{bbm}
\usepackage{tabularx}
\usepackage{tikz}

\newcommand{\R}{\mathbb{R}}
\newcommand{\N}{\mathbb{N}}
\renewcommand{\P}{\mathbb{P}}
\newcommand{\E}{\mathbb{E}}

    \def\d{{\textnormal d}}

\newenvironment{eqnarr}{\begin{IEEEeqnarray}{rCl}}{\end{IEEEeqnarray}\ignorespacesafterend}

\renewcommand{\eqref}[1]{\hyperref[#1]{(\ref*{#1})}}

\newcommand*{\norm}[1]{\lVert #1 \rVert}

    \def\beq{\begin{eqnarr}}
    \def\eeq{\end{eqnarr}}
    \def\beqq{\begin{eqnarray*}} 
    \def\eeqq{\end{eqnarray*}} 

        \def\d{{\rm d}}

    \def\d{{\textnormal d}}
\newtheorem{remark}{Remark}[section]

\renewcommand{\ln}{\log}

\newcommand*{\pref}[1]{\hyperref[#1]{(\ref*{#1})}}
\newcommand*{\refpref}[2]{\hyperref[#2]{\ref*{#1}(\ref*{#2})}}

\renewcommand{\ln}{\log}

\newcommand{\pdag}{\mathbf{P}^{\dagger}}
\newcommand{\edag}{\mathbf{E}^{\dagger}}

\everymath{\displaystyle}

\makeatletter
\def\namedlabel#1#2{\begingroup
    #2%
    \def\@currentlabel{#2}%
    \phantomsection\label{#1}\endgroup
}
\makeatother

  \newcommand{\D}{{\rm d}}

\usepackage{cancel} 
\usepackage{soul} 

\newcommand{\U}{\texttt U}

\newcommand{\bS}{\texttt S}
\newcommand{\ebS}{\emph{\texttt S}}

\newcommand{\bF}{\texttt F}
\newcommand{\ebF}{\emph{\texttt F}}

\newcommand{\bL}{\texttt L}

\startlocaldefs
\numberwithin{equation}{section}
\theoremstyle{plain}

\endlocaldefs

\begin{document}

\begin{frontmatter}
\title{Stochastic Methods for the Neutron  Transport Equation I: Linear Semigroup asymptotics}

\runtitle{Recurrent extension for ssMp in a Wedge}

\begin{aug}
\author{\fnms{Emma Horton}\thanksref{t2}\ead[label=e5]{emma.horton94@gmail.com}},
 \author{\fnms{Andreas E. Kyprianou}\thanksref{t3}\ead[label=e3]{a.kyprianou@bath.ac.uk}}
\and \author{\fnms{Denis Villemonais}\ead[label=e6]{denisvillemonais@gmail.com}}

\thankstext{t2}{This work was done whilst in receipt of a PhD scholarship  from industrial partner Wood (formerly Amec Foster Wheeler).}

\thankstext{t3}{Supported by EPSRC grant EP/P009220/1}



\address{
E. L. Horton,\\
Institut \'Elie Cartan de Lorraine \\
Universit\'e de Lorraine\\
54506 Vandoeuvre-l\`es-Nancy Cedex, \\
France.\\
\printead{e5}
}

\address{
A.E. Kyprianou,  \\
University of Bath\\
Department of Mathematical Sciences \\
Bath, BA2 7AY\\
 UK.\\
\printead{e3}
}


\address{D. Villemonais\\
Institut \'Elie Cartan de Lorraine\\
Universit\'e de Lorraine \\
54506 
Vandoeuvre-l\`es-Nancy Cedex\\
France.\\
\printead{e6}
}
\end{aug}

\begin{abstract}\hspace{0.1cm}
The Neutron Transport Equation (NTE) describes the flux of neutrons through an inhomogeneous fissile medium. 
In this paper, we reconnect the  NTE to the physical model of the spatial Markov branching process which describes the process of nuclear fission, transport, scattering, and absorption. By reformulating the NTE in its mild form and identifying its solution as an expectation semigroup, we use modern techniques to develop a Perron-Fr\"obenius (PF) type decomposition, showing that growth is dominated by a leading eigenfunction and its associated left and right eigenfunctions. In the spirit of results for spatial branching and fragmentation processes, we use our PF decomposition to show the existence of an intrinsic martingale and associated  spine decomposition. Moreover, we show how  criticality in the PF decomposition dictates the convergence of the intrinsic martingale. The mathematical difficulties in this context come about through unusual piecewise linear motion of particles coupled with an infinite type-space which is taken as neutron velocity. The fundamental nature of our PF decomposition also plays out in  accompanying work \cite{SNTE-II, MCNTE}.
\end{abstract}



\end{frontmatter}

\setcounter{tocdepth}{1}

\section{Introduction}\label{intro}The neutron transport equation (NTE) describes the flux of neutrons across a planar cross-section  in an inhomogeneous fissile medium (measured is number of neutrons per cm$^2$ per second). Neutron flux is described as a function of time, $t$, Euclidian location, $r$, direction, $\Omega$ and neutron energy $E$. It is not uncommon in the physics literature to assume that velocity is a function of both direction and energy, thereby reducing the number of variables by one. This allows us to describe the dependency of flux more simply in terms of time and, what we call, the {\it configuration variables} $ (r, v) \in  D \times V$ where $D\subseteq\mathbb{R}^3$ is a non-empty, open, smooth, bounded and convex domain such that $\partial D$ has zero Lebesgue measure, and 
$V$ is the velocity space, which we take to be the three dimensional annulus  $V = \{\upsilon\in \mathbb{R}^3:{\texttt v}_{\texttt{min}}\leq  |\upsilon|\leq {\texttt v}_{\texttt{max}}\}
$, where $0<{\texttt v}_{\texttt{min}}<{\texttt v}_{\texttt{max}}<\infty$. 

\smallskip
 As a {\it backward equation}, the NTE is written in the form 
\begin{align}
\frac{\partial}{\partial t}\psi_t(r, \upsilon) &=\upsilon\cdot\nabla\psi_t(r, \upsilon)  -\sigma(r, \upsilon)\psi_t(r, \upsilon)\notag\\
&+ \sigma_{\texttt{s}}(r, \upsilon)\int_{V}\psi_t(r, \upsilon') \pi_{\texttt{s}}(r, \upsilon, \upsilon')\d\upsilon' + \sigma_{\texttt{f}}(r, \upsilon) \int_{V}\psi_t(r, \upsilon') \pi_{\texttt{f}}(r, \upsilon, \upsilon')\d\upsilon',
\label{bNTE}
\end{align}
where the different components (or {\it cross-sections} as they are known in the nuclear physics literature) have the following interpretation:
\begin{align*}
\sigma_{\texttt{s}}(r, \upsilon) &: \text{ the rate at which scattering occurs from incoming velocity $\upsilon$,}\\
\sigma_{\texttt{f}}(r, \upsilon) &: \text{  the rate at which fission occurs from incoming velocity $\upsilon$,}\\
\sigma(r, \upsilon) &: \text{ the sum of the rates } \sigma_{\texttt{f}}+ \sigma_{\texttt{s}} \text{ and is known as the total cross section,}\\
\pi_{\texttt{s}}(r, \upsilon, \upsilon')\d\upsilon' &: \text{  the scattering yield at velocity $\upsilon'$ from incoming velocity }  \upsilon, \\
 &\hspace{0.5cm}\text{ satisfying }\textstyle{\int_V}\pi_{\texttt{s}}(r, \upsilon, \upsilon'){\rm d}\upsilon'=1,\text{ and }\\
 \pi_{\texttt{f}}(r, \upsilon, \upsilon')\d\upsilon' &:  \text{  the neutron yield at velocity $\upsilon'$ from fission with incoming velocity }   \upsilon,\\
 &\hspace{0.5cm}\text{ satisfying }{\color{black} \textstyle{\int_V}\pi_{\texttt{f}}(r, \upsilon, \upsilon')\d\upsilon' <\infty.}
 \end{align*}
Some or all of the three  assumptions below will be used from time to time in our main results.
\smallskip
{\bf 
{\color{black} 

(H1): Cross-sections $\sigma_{\texttt{s}}$, $\sigma_{\texttt{f}}$, $\pi_{\texttt{s}}$ and $\pi_{\texttt{f}} $ are uniformly bounded away from   infinity.
\smallskip

 (H2): 
We have 
$
\sigma_{\texttt{s}} \pi_{\texttt{s}}  + 
\sigma_{\texttt{f}} \pi_{\texttt{f}}>0$ on $D\times V\times V$.
\smallskip

(H3): There is an open ball $B$ compactly embedded in $D$ such that $\sigma_{\texttt{f}}\pi_{\texttt{f}} >0$ on $B\times V\times V$.
}
}
\smallskip

It is also usual to insist on the physical boundary conditions
\begin{equation}
\left\{
\begin{array}{ll}
\psi_0(r, \upsilon) = g(r, \upsilon) &\text{ for }r\in D, \upsilon\in{V},
\\
&
\\
\psi_t(r, \upsilon) = 0& \text{ for } t \ge 0 \text{ and } r\in \partial D
\text{ if }\upsilon
\cdot{\bf n}_r>0,
\end{array}
\right.
\label{BC}
\end{equation}

where ${\bf n}_r$ is the outward unit normal at $r \in \partial D$ 
and $g: D \times V \to [0, \infty)$ is a bounded, measurable function on which we will later impose further conditions. Physically, this boundary conditions mean that any neutron starting on the boundary of the reactor with velocity pointing outwards will be `killed'.
\smallskip

Formally speaking, \eqref{bNTE} as stated is ill defined (because of regularity issues associated to the transport operator $\upsilon\cdot\nabla$) and has traditionally otherwise appeared in applied mathematics and physics literature in the form of an {\it abstract Cauchy problem}   on  $L_2(D\times V)$, the space of square integrable functions on $D\times V$. This has formed the principle historical outlook of the analysis of the NTE, appealing to $c_0$-semigroup theory; see for example the classical works of \cite{DS, Leh, LW1, LW2, J, Bell, PR69, D, DL6, M-K, Hbook, LPS, PP}. 
\smallskip 

The connection of the NTE via semigroup theory to an underlying stochastic process has, in contrast,  received a very limited amount  of attention; cf \cite{D, LPS, MT}. Accordingly the stochastic analysis of \eqref{bNTE} has seen very little development in light of recent innovations in the relevant theory of stochastic processes.

\smallskip 

In the current article, we are more interested in exploring how NTE can be interpreted as a {\it mild equation}, describing the mean semigroup evolution of the stochastic process that models the underlying physical process of neutron fission, transport, scatter and absorption. More precisely, we have two main contributions: (i) to develop a new precise statement of the form $\psi_t\sim {\rm e}^{\lambda_*t} c_g\varphi + o({\rm e}^{\lambda_*t})$, where $\lambda_*$ and $\varphi$ are  a leading eigenvalue and eigenfunction associated to the NTE and $c_g$ is a constant that depends on the initial data $g$; (ii) to make the first step in understanding how the growth of the solution to the NTE relative to its lead eigenfunction plays out in terms of the aforementioned physical stochastic process and an associated martingale. 

\smallskip

This paper follows the review article \cite{MultiNTE} which consolidates the existing $c_0$-semigroup approach and how it relates to 
the stochastic representation.
A deeper subsequent analysis in the direction of our second objective is continued in the accompanying paper \cite{SNTE-II}. Further numerical and Monte-Carlo considerations based on our stochastic approach will also appear in forthcoming work \cite{MCNTE}.

\smallskip

In order to consider the probabilistic perspective, we start by defining the underlying stochastic processes which mimics the physics of neutron fission, transport, scattering and absorption.

\section{The physical process and the mild NTE}\label{2}Consider a neutron branching process (NBP), which at time $t\geq0$ is represented by a configuration of particles which are specified via their physical location and velocity in $D\times V$, say $\{(r_i(t), \upsilon_i(t)): i = 1,\dots , N_t\}$, where $N_t$ is the number of particles alive at time $t\ge 0$.
In order to describe the process, we will represent it as a process in the space of finite atomic measures
\begin{equation}
X_t(A) = \sum_{i=1}^{N_t}\delta_{(r_i(t), \upsilon_i(t))}(A), \qquad A\in\mathcal{B}(D\times V), \;t\ge 0,
\label{randommeasure}
\end{equation}
where $\delta$ is the Dirac measure, defined on $\mathcal{B}(D\times V)$, the Borel subsets of $D\times V$.
The evolution of $(X_t, t\geq 0)$ is a stochastic process valued in the space of atomic measures
$
\mathcal{M}(D\times V): = \{\textstyle{\sum_{i = 1}^n}\delta_{(r_i,\upsilon_i)}: n\in \mathbb{N}, (r_i,\upsilon_i)\in D\times V, i = 1,\cdots, n\}
$
which evolves randomly as follows.

\smallskip

A particle positioned at $r$ with velocity $\upsilon$ will continue to move along the trajectory $r + \upsilon t$, until one of the following things happens. 
\begin{enumerate}
\item[(i)] The particle leaves the physical domain $D$, in which case it is instantaneously killed. 
\item[(ii)] Independently of all other neutrons, a scattering event occurs when a neutron comes in close proximity to an atomic nucleus and, accordingly, makes an instantaneous change of velocity. For a neutron in the system with position and velocity $(r,\upsilon)$, if we write $T_{\texttt{s}}$ for the random time that scattering may occur, then independently of any other physical event that may affect the neutron, 
$
\Pr(T_{\texttt{s}}>t) = \exp\{-\textstyle{\int_0^t} \sigma_{\texttt{s}}(r+\upsilon s, \upsilon){\rm d}s \}, $ for $t\geq0.$
\smallskip

When scattering occurs at space-velocity $(r,\upsilon)$, the new velocity is selected in $V$ independently with probability $\pi_{\texttt{s}}(r, \upsilon, \upsilon')\d\upsilon'$. 
\item[(iii)] Independently of all other neutrons, a fission event occurs when a neutron smashes into an atomic nucleus. 
For a neutron in the system  with initial position and velocity $(r,\upsilon)$, if we write $T_{\texttt{f}}$ for the random time that fission may occur, then independently of any other physical event that may affect the neutron, 
$
\Pr(T_{\texttt{f}}>t) = \exp\{-\textstyle{\int_0^t} \sigma_{\texttt{f}}(r+\upsilon s, \upsilon){\rm d}s \},$ for $t\geq 0.
$
\smallskip

When fission occurs,  the smashing of the atomic nucleus produces  lower mass isotopes and releases  a random number of neutrons, say $N\geq 0$, which are ejected from the point of impact with randomly distributed, and possibly correlated, velocities, say $(\upsilon_i: i=1,\cdots, N)$. The outgoing velocities are described by  the atomic random measure 
\begin{equation}
\label{PP}
\mathcal{Z}(A): = \sum_{i= 1}^{N } \delta_{\upsilon_i}(A), \qquad A\in\mathcal{B}(V).
\end{equation} 

When fission occurs at location $r\in\mathbb{R}^d$ from a particle with incoming velocity $\upsilon\in{V}$, 
we denote by ${\mathcal P}_{(r,\upsilon)}$ the law of $\mathcal{Z}$.
The probabilities ${\mathcal P}_{(r,\upsilon)}$ are such that, for  $\upsilon'\in{V}$, for bounded and measurable $g: V\to[0,\infty)$,
\begin{align}
\int_V g(\upsilon')\pi_{\texttt{f}}(r, v, \upsilon')\d\upsilon' &= {\mathcal E}_{(r,\upsilon)}\left[\int_V g(\upsilon')\mathcal{Z}(\d \upsilon')\right]
=: {\mathcal E}_{(r,\upsilon)}[\langle g, \mathcal{Z}\rangle].
\label{Erv}
\end{align}
Note, the possibility that $\Pr(N = 0)>0$, which will be tantamount to neutron capture (that is, where a neutron slams into a nucleus but no fission results and the neutron is absorbed into the nucleus). 
\end{enumerate}
\smallskip

In essence, one may think of the process $X: = (X_t, t\geq 0)$ as a typed spatial Markov branching process, where  the type of each particle is the velocity  $\upsilon\in{V}$  and the underlying Markov motion is nothing more than movement in a straight line at velocity $\upsilon$.  

{\color{black}
\begin{remark}\rm
It is worth noting how the assumptions (H1)-(H3) play out in the construction of the NBP. Whilst they serve as a sufficient conditions, they are not necessary. For example, one could equally assume that e.g. there are two open domains $B_{\texttt s}$ and $B_{\texttt f}$ (which may or may not intersect)  contained in $D$ on which $\sigma_{\texttt{f}}(r, \upsilon) \pi_{\texttt{f}}(r, \upsilon, \upsilon')  >0$, for  $B_{\texttt s}\times V\times V$ and $
\sigma_{\texttt{f}}(r, \upsilon) \pi_{\texttt{f}}(r, \upsilon, \upsilon') >0$, for  $B_{\texttt f}\times V\times V$, respectively. This would ensure that, at least starting from some configurations  $(r,\upsilon)\in D\times V$, the NBP could access regions where scatter or fission  occurs with positive probability. From there, the particle system will thus propagate by allowing further opportunities for scatter or fission. That said, there will also be some initial configurations $(r,\upsilon)\in D\times V$ for which the particles will neither scatter nor undergo fission and head straight for the boundary $\partial D$, whereupon they are killed. 

\smallskip

This example informally alerts us to the notion of  {\it `irreducibility'} of the state space. For contrast, and to highlight the issue further, it is worth comparing the situation to  e.g.  a branching Brownian motion on a smooth, convex, bounded domain of $D\subseteq \mathbb{R}^d$ for which the branching rate is supported only on a subdomain $B$ strictly contained in $D$. In that setting, the Brownian motion of a given particle would always be able to `find' the region $B$ with positive probability, where branching can occur  (thus propagating the stochastic process in a non-trivial way). Through this comparison, we see that  the piecewise linear spatial paths of neutrons in the NBP, although simpler to depict than the path of a Brownian motion, are significantly more irregular. The assumption (H2) may be thus be thought of as a sufficient condition to ensure irreducibility of the state space by enforcing the possibility of either scatter or fission (but not necessarily the possibility of both), whereas assumption (H3) ensures that there is at least one area of the domain where fission occurs. The condition (H1) simply ensures that activity (scatter and fission) cannot happen too fast, and hence the eventuality of explosion in finite time does not appear in our forthcoming calculations.

\end{remark}
}

\begin{remark}\label{nonuniqueNBP}\rm
The NBP is parameterised by the quantities $\sigma_{\texttt s}, \pi_{\texttt s}, \sigma_{\texttt f}$ and the family of measures ${\mathcal P} =({\mathcal P}_{(r,\upsilon)} , r\in D,\upsilon\in V)$ and accordingly we refer to it as a $(\sigma_{\texttt s}, \pi_{\texttt s}, \sigma_{\texttt f}, \mathcal{P})$-NBP. It is associated to the NTE via the relation \eqref{Erv},  and, although a $(\sigma_{\texttt s}, \pi_{\texttt s}, \sigma_{\texttt f}, \mathcal{P})$-NBP is uniquely defined, a NBP specified by $(\sigma_{\texttt s}, \pi_{\texttt s}, \sigma_{\texttt f}, \pi_{\texttt f})$ alone is not.  

\smallskip

What is of importance for the purpose of our analysis, however, is that for the given quadruple $(\sigma_{\texttt s}, \pi_{\texttt s}, \sigma_{\texttt f}, \pi_{\texttt f})$, at least one $(\sigma_{\texttt s}, \pi_{\texttt s}, \sigma_{\texttt f}, \mathcal{P})$-NBP exists such that \eqref{Erv} holds. 
It is relatively easy to construct an example of $\mathcal{P}$ as such.

\smallskip

{\color{black}Indeed, let us suppose (H1) and (H3) hold.} Then, for a given $\pi_\texttt{f}$, define 
\[
n_{\texttt{max}}= \textstyle{\min\{k\geq 1: \sup_{(r, \upsilon) \in D\times V}\int_V \pi_{\texttt f}(r,\upsilon,\upsilon')\d \upsilon'\leq k\} }.
\]

The ensemble $(\upsilon_i, i =1,\cdots, N)$ is such that: (i) $N\in \{0,n_{\texttt{max}}\}$;  (ii) the probability of the event $\{N = n_{\texttt{max}}\}$ under $\mathcal{P}_{(r,\upsilon)}$ is given by $\textstyle{\int_V \pi_{\texttt f}(r,\upsilon,\upsilon'')\d \upsilon''/n_{\texttt{max}}}$; (iii)  on the event  $\{N = n_{\texttt{max}}\}$, each of the $n_{\texttt{max}}$ neutrons are released with the same velocities $\upsilon_1=\dots = \upsilon_{n_{\texttt{max}}}$; (iv) the distribution of this common velocity is given by 
\[
\mathcal{P}_{(r,\upsilon)}(\upsilon_i \in \d\upsilon'| N = n_{\texttt{max}} ) = \frac{\pi_{\texttt f}(r,\upsilon,\upsilon')}{\int_V \pi_{\texttt f}(r,\upsilon,\upsilon'')\d \upsilon''}\d\upsilon',
\]
for $i = 1,\dots, n_{\texttt{max}}$. 

\smallskip 

With the construction (i)-(iv) for $\mathcal{P}_{(r,\upsilon)}$, we have 
for bounded and measurable $g: V\to[0,\infty)$,
\begin{align*}
&\int_V g(\upsilon')\pi_{\texttt{f}}(r, v, \upsilon')\d\upsilon' \\
&= 
0 \times \left(1- \mathcal{P}_{(r,\upsilon)}(N = n_{\texttt{max}})\right) + \mathcal{P}_{(r,\upsilon)}(N = n_{\texttt{max}})n_{\texttt{max}}\int_V g(\upsilon')\mathcal{P}_{(r,\upsilon)}(\upsilon_i \in \d\upsilon'| N = n_{\texttt{max}} )\\
&=\frac{\int_V \pi_{\texttt f}(r,\upsilon,\upsilon'')\d \upsilon''}{n_{\texttt{max}}}
n_{\texttt{max}}\int_V g(\upsilon')
 \frac{\pi_{\texttt f}(r,\upsilon,\upsilon')}{\int_V \pi_{\texttt f}(r,\upsilon,\upsilon'')\d \upsilon''}\d\upsilon'\\
&=\int_V g(\upsilon')\pi_{\texttt{f}}(r, v, \upsilon')\d\upsilon',
\end{align*}
thus matching \eqref{Erv}, as required. 
\smallskip

It is interesting to note that the construction above is precisely what happens in industrial models of nuclear reactor cores (for which only the cross-sections $(\sigma_{\texttt s}, \pi_{\texttt s}, \sigma_{\texttt f}, \pi_{\texttt f})$ are known) when it comes to  Monte-Carlo simulation; see further discussion below as well as \cite{MCNTE}. 
\end{remark}

The maximum number of neutrons that can be emitted during a fission event with positive probability (for example in an  environment where the heaviest nucleus is {\it Uranium-235},  there are at most 143 neutrons that can be released in a fission event, albeit, in reality it is more likely that 2 or 3 are released). {\color{black}We will thus occasionally work with:
 
\smallskip

 {\bf (H4): Fission offspring are bounded in number  by the constant $n_{\texttt{max}}> 1$.}
 }
 \smallskip

In particular this means that 
$
\textstyle{\sup_{ r\in D, \upsilon\in V}\int_V\pi_{\texttt{f}}(r,\upsilon,\upsilon')\d \upsilon'\leq n_\texttt{max}}.
$

\smallskip

Write $\mathbb{P}_{\delta_{(r, \upsilon)}}$ for the the law of $X$ when issued from a single particle  with space-velocity configuration $(r,\upsilon)\in {D}\times V$.
More generally, for $\mu\in\mathcal{M}(D\times V)$, 
we understand 
$
\mathbb{P}_{\mu} : = \mathbb{P}_{\delta_{(r_1,\upsilon_1)}}\otimes\cdots\otimes\mathbb{P}_{\delta_{(r_n,\upsilon_n)}}$
 when $\mu = \textstyle{\sum_{i = 1}^n} \delta_{(r_i,\upsilon_i)}.
$
In other words, the process $X$ when issued from initial configuration $\mu$ , is equivalent to issuing $n$ independent copies of $X$, each with configuration $(r_i, \upsilon_i)$, $i = 1,\cdots, n$.
\smallskip

Like all spatial Markov branching processes, $(X, \mathbb{P})$, where $\mathbb{P} : = (\mathbb{P}_{\mu}, \mu\in\mathcal{M}(D\times V))$, respects the {\it Markov branching property} with respect to the filtration $\mathcal{F}_t: = \sigma((r_i(s),\upsilon_i(s)) : i = 1,\cdots, N_s, s\leq t)$, $t\geq 0$.  That is to say, for all bounded and measurable $g: {D}\times V\to [0,\infty)$ and $\mu\in\mathcal{M}(D\times V)$ written $\mu = \textstyle{\sum_{i = 1}^n\delta_{(r_i,\upsilon_i)}},
$
we have 
$
\textstyle{\mathbb{E}_{\mu}[\prod_{i=1}^{n}g(r_i, \upsilon_i)] = \prod_{i=1}^{n}
u_t[g](r_i, \upsilon_i)},$ for  $t\geq 0, r\in D, \upsilon\in{V},
$
where $\textstyle{u_t[g](r,\upsilon): = \mathbb{E}_{\delta_{(r, \upsilon)}}[\prod_{i=1}^{N_{t}}g(r_i(t), \upsilon_i(t))]}. 
$ 
{\color{black}
{\it In this setting it is also customary to work with the notion that the empty product is valued as unity; see \cite{INW1, INW2, INW3}.}
}

\smallskip

What is of particular interest to us in the context of the NTE is the {\it expectation semigroup} of the neutron branching process. More precisely, and with pre-emptive notation, we are interested in 
\begin{equation}
\psi_t[g](r,\upsilon) : = \mathbb{E}_{\delta_{(r, \upsilon)}}[\langle g, X_t \rangle], \qquad t\geq 0, r\in \bar{D}, \upsilon\in{V},
\label{semigroup}
\end{equation}
for $g\in L^+_\infty(D\times V)$, the space of non-negative uniformly bounded measurable  functions on  $D\times V$. 
Here we have made a slight abuse of notation (see $\langle \cdot,\cdot\rangle$ as it appears in \eqref{Erv}) and written $\langle g, X_t \rangle$ to mean $\textstyle{\int_{D\times V}} g(r,\upsilon)X_t(\d r,\d \upsilon)$.

\smallskip

To see why $(\psi_t, t\geq 0)$ deserves the name of expectation semigroup, it is a straightforward exercise with the help of the  Markov branching property to show that 
\begin{equation}
\psi_{t+s}[g](r,\upsilon) =\psi_t[\psi_s[g]](r,\upsilon)\qquad s,t\geq 0.
\label{branchingsemigroup}
\end{equation}

The connection of the expectation semigroup \eqref{semigroup} with the NTE \eqref{bNTE} was explored in the recent article \cite{MultiNTE} (see also older work in \cite{D, LPS}). In order to present the relevant findings, let us momentarily introduce some notation. 
The deterministic evolution $
\texttt{U}_t[g]( r,\upsilon) = g( r+\upsilon t, \upsilon)\mathbf{1}_{\{t<\kappa^D_{r,\upsilon}\}}, t\geq 0,
$
and $
\kappa_{r,\upsilon}^{D} := \inf\{t>0 : r+\upsilon t\not\in D\},
$
represents the advection semigroup associated with a single neutron travelling at velocity $\upsilon$ from $r$.
The backwards scatter operator is denoted by
\begin{equation}
{\bS}{f}(r, \upsilon) = \sigma_{\texttt{s}}(r, \upsilon)\int_{V}{f}(r, \upsilon') \pi_{\texttt{s}}(r, \upsilon, \upsilon') \d\upsilon'  - \sigma_{\texttt{s}}(r, \upsilon){f}(r, \upsilon) 
\label{S}
\end{equation} 
and the backwards fission operator is given by 
\begin{equation}
{\bF}{f}(r, \upsilon)  =  \sigma_{\texttt{f}}(r, \upsilon) \int_{V}{f}(r, \upsilon') \pi_{\texttt{f}}(r, \upsilon, \upsilon')\d\upsilon' -\sigma_{\texttt{f}}(r, \upsilon)f(r,\upsilon),
\label{F}
\end{equation}
for $f\in L^+_\infty(D\times V)$,
such that both $\bS$ and $\bF$  are defined on $D\times V$ and zero otherwise.

\begin{lemma}[\cite{MultiNTE}]\label{NBPrep}Under (H1) and (H2), for $g\in L^+_\infty(D\times V)$, there exist  constants $C_1,C_2>0$ such that  $\psi_t[g]$, as given in \eqref{semigroup}, is uniformly bounded by $ C_1\exp(C_2 t)$, for all $t\geq 0$. Moreover, $(\psi_t[g], t\geq 0)$ is the unique solution, which is bounded in time, to  the so-called mild equation (also called a {\it Duhamel solution} in the PDE literature):
\begin{equation}
\psi_t[g] = \emph{\texttt{U}}_t[g] + \int_0^t \emph{\texttt{U}}_s[({\ebS} + {\ebF})\psi_{t-s}[g]]\d s, \qquad t\geq 0,
\label{mild}
\end{equation}
for which \eqref{BC} holds.
\end{lemma}
The fact that \eqref{semigroup} solves \eqref{mild} is a simple matter of conditioning the expression in \eqref{semigroup} on the first fission or scatter event (whichever occurs first) and rearranging the resulting equation. Uniqueness is a matter of working  in  the right way with Gr\"onwall's Lemma. The association of \eqref{mild} with \eqref{bNTE} in this way  was also explored in Theorem 7.1 \cite{MultiNTE}, where it was shown that the unique  solution to \eqref{bNTE} when seen as an abstract Cauchy problem on $L_2(D\times V)$ agrees with the unique solution to \eqref{mild} in the  $L_2(D\times V)$ norm.

\smallskip

{\color{black}The reader should note that we do not need (H3) or (H4) as the result does not require information about the pathwise behaviour of any associated underlying stochastic processes. Nor does it distinguish between the settings that $\bF$ is present or not.}

\section{Perron-Frobenius asymptotics}
As alluded to above, 
one of the  classical ways in  which neutron flux is understood is to look for the leading eigenvalue and associated ground state eigenfunction. Roughly speaking, this means looking for an associated triple of eigenvalue $\lambda_*\in\mathbb{R}$,  positive right eigenfunction  $\varphi: {D}\times V\to[0,\infty)$, a left eigenmeasure $\tilde\varphi(r,\upsilon)\d r \d \upsilon$ on ${D}\times V$  in $L^+_2(D\times V)$ (the cone of non-negative square integrable functions on ${D}\times V$) such that 
$
\langle f, \psi_t[\varphi]\rangle={\rm e}^{\lambda_* t} \langle f,\varphi  \rangle$  and 
$ \langle\tilde\varphi , \psi_t[f]\rangle
 ={\rm e}^{\lambda_* t}\langle\tilde{\varphi}, f \rangle,$ for $t\geq 0$. Here, we again abuse our notation (see the use of $\langle\cdot,\cdot\rangle$ in \eqref{Erv} and \eqref{semigroup}) and write, for $f, g\in L^+_2(D\times V)$,  $\langle f, g\rangle =  \textstyle{\int_{D\times V}}f(r,\upsilon)g(r,\upsilon)\d r\d\upsilon.
$
With the eigentriple in hand, it is a common point of analysis that, to leading order, the NTE \eqref{bNTE} is solved through the approximation
\begin{equation}
\psi_t(r,\upsilon) = {\rm e}^{\lambda_* t}\langle g, \tilde{\varphi}\rangle \varphi(r, \upsilon) + o({\rm e}^{\lambda_* t}), \qquad t\geq 0, r\in D, \upsilon\in V,
\label{sim}
\end{equation}
where the sense of the equality depends on how one interprets the NTE (i.e. as an abstract Cauchy problem or in its mild form).

\begin{figure}[h!]
\includegraphics[width= 0.23\textwidth]{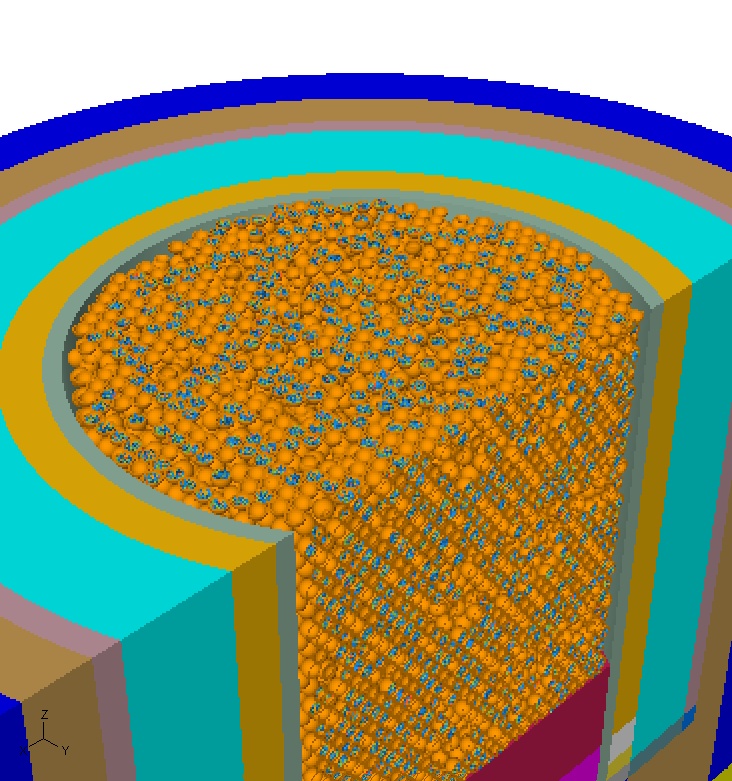}
\includegraphics[width= 0.23\textwidth]{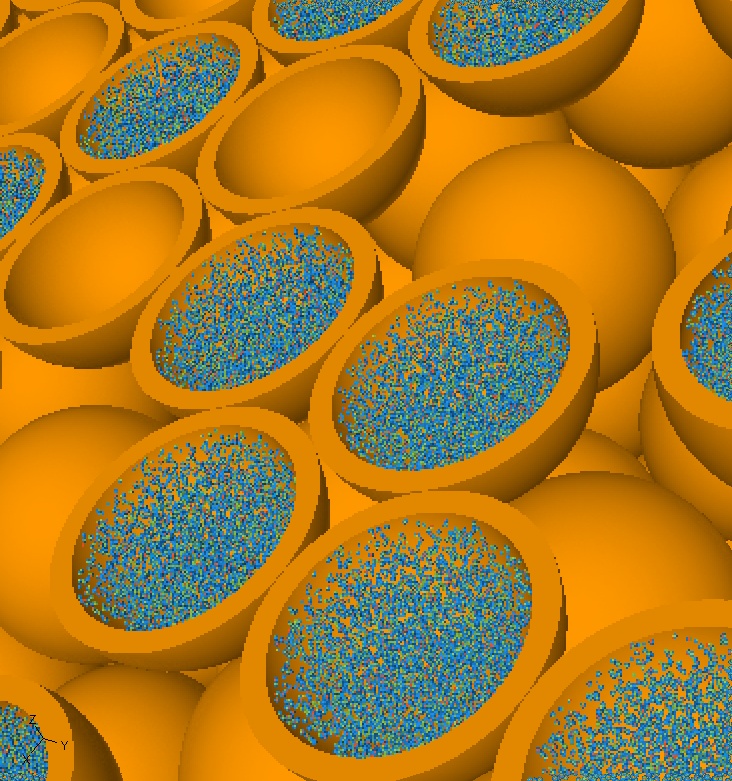}
\includegraphics[width= 0.23\textwidth]{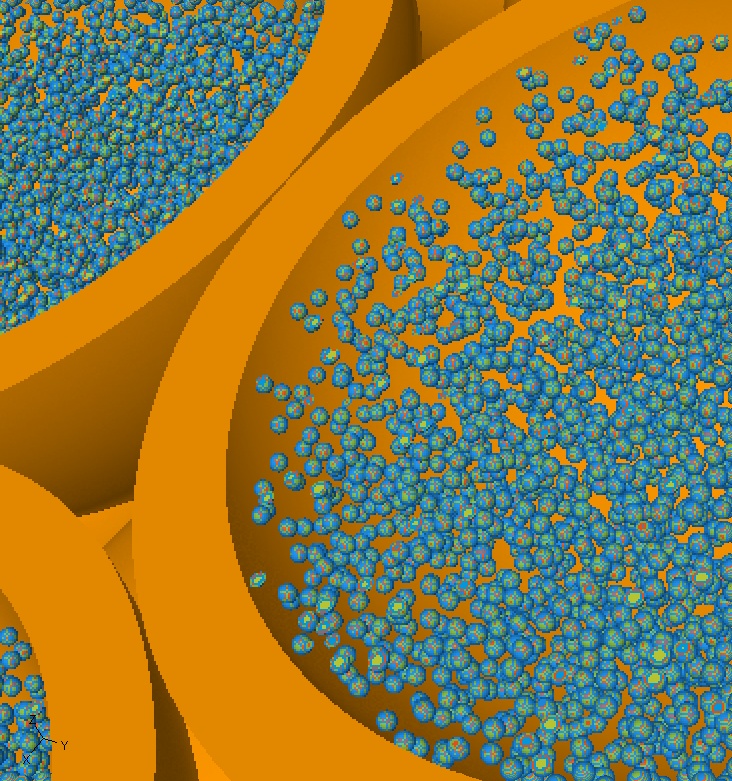}
\includegraphics[width= 0.23\textwidth]{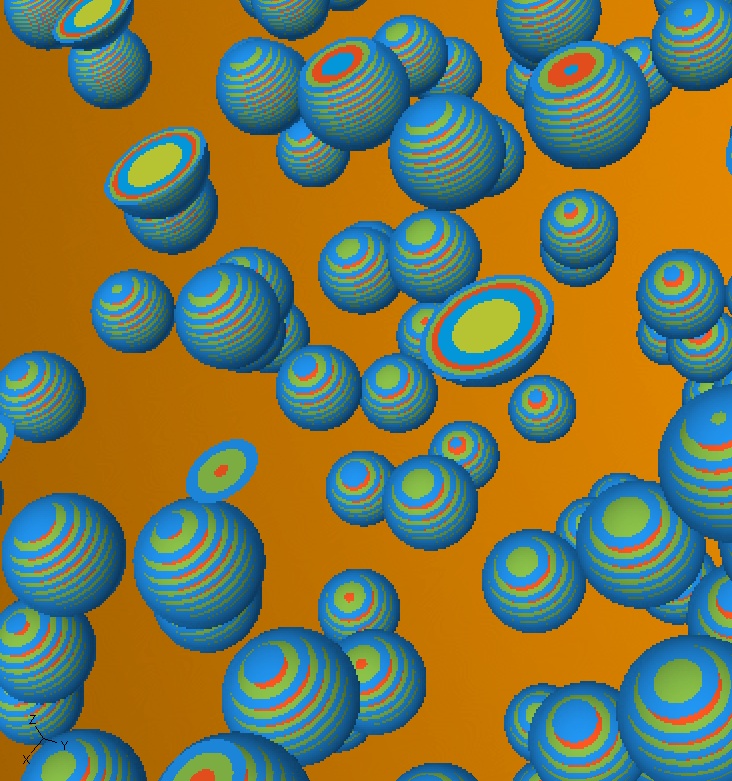}
\caption{\rm Zooming into a  virtual model of a pebble bed nuclear reactor core made from tennis ball sized orbs which  encapsulate uranium pellets. from left to right, the diagrams illustrate detail from metres down to millimetres. Colour indicating the different regions in which the cross sections $ \sigma_{\texttt{s}},  \sigma_{\texttt{f}},  \pi_{\texttt{s}},  \pi_{\texttt{f}}$ are constant.  The structural design of such a reactor  can easily be stored as a virtual environment (i.e. storing the coordinates of the different geometrical domains and the material properties in each domain) with just circa 150MB of data, on to which extensive  data libraries of industrial numerical values for the respective quantities $ \sigma_{\texttt{s}},  \sigma_{\texttt{f}},  \pi_{\texttt{s}},  \pi_{\texttt{f}}$  can be mapped. }
\label{reactorcore}
\end{figure}

\smallskip

The eigenfunction $\tilde\varphi$ is called the {\it importance  map} and offers a quasi-stationary  profile of radioactive activity (unless $\lambda_*=0$, in which case it is a stationary profile). 
Indeed, in
modern nuclear reactor  design and safety regulation, it is usually the case that
virtual reactor models such as the one seen in Figure \ref{reactorcore} (an example of a uranium pebble bed reactor) are
designed such that $\lambda_* = 0$ and the behaviour of $\tilde\varphi$
within spatial domains on the human scale remains within regulated levels.
Existing physics and engineering literature with focus on applications in the nuclear regulation industry, has largely been concerned with different numerical methods for estimating the value of the eigenvalue $\lambda_*$ as well as the eigenfunction $\varphi$ and eigenmeasure $\tilde{\varphi}(r,\upsilon)\d r\d\upsilon$. 
Giving a sensible meaning to \eqref{sim} will play an important part in unraveling the analysis of stochastic representations of solutions to the NTE as well.  Moreover, in additional forthcoming work \cite{MCNTE}, we will also see that our asymptotic \eqref{sim}, together with the accompanying stochastic analysis developed here, has influence on a number of completely new Monte Carlo methods associated with the NTE that, in turn, bears relevance to the applied NTE literature.  
\smallskip

The approximation \eqref{sim} can be seen as a functional version of the Perron-Frobenius Theorem, in particular when noting via \eqref{mild} that we can understand  $\psi_t[g]$ as  a semigroup.
Many attempts have been made to generalise the notion of the Perron-Frobenius decomposition  to  semigroups of Markov processes  with countable and  uncountable state spaces, as well as with killing and mass creation (see  for example \cite{DMT, T1, T2, T3}), using what has come to be known as $R$-theory.  The conditions there seem difficult to verify in the current setting. 

\smallskip

More recently,  \cite{CV, CVecp} have provided an alternative approach to the $R$-theory presented in aforementioned works. In the current context,  Theorem 2.1  and Proposition 2.3 of \cite{CV} will help us to achieve the  global result, given below. {\color{black}
To state it we need to introduce the quantity 
\begin{align}
\alpha(r,\upsilon)\pi(r, \upsilon, \upsilon') = \sigma_{\texttt{s}}(r,\upsilon)\pi_{\texttt{s}}(r, \upsilon, \upsilon') + \sigma_{\texttt{f}}(r,\upsilon) \pi_{\texttt{f}}(r, \upsilon, \upsilon')\qquad r\in D, \upsilon, \upsilon'\in V,
\label{alpha}
\end{align}
where $\pi$ is taken to be a probability density. As such it necessarily follows that 
\begin{equation}
\alpha(r,\upsilon) = \sigma_{\texttt s}(r,\upsilon) + \sigma_{\texttt{f}}(r,\upsilon) \int_V\pi_{\texttt{f}}(r, \upsilon, \upsilon')\d \upsilon'.
\label{aalone}
\end{equation}
}

\begin{theorem}\label{CVtheorem}
Suppose that (H1) holds as well as  
\smallskip

{\bf {\color{black}(H2)$^*$: 
$\textstyle{\inf_{r\in D, \upsilon, \upsilon'\in V} \alpha(r,\upsilon)\pi(r,\upsilon,\upsilon')>0.}$
}
}
\smallskip

Then, for the semigroup  $(\psi_t,t\geq0)$ identified by \eqref{mild},  there exists  a $\lambda_*\in\mathbb{R}$, a  positive\footnote{To be precise, by a positive eigenfunction, we mean a mapping from $D\times V\to (0,\infty)$. This does not prevent it being valued zero on $\partial D$, as $D$ is an open bounded, convex domain.} right eigenfunction $\varphi \in L^+_\infty(D\times V)$ and a left eigenmeasure which is absolutely continuous with respect to Lebesgue measure on $D\times V$ with density $\tilde\varphi\in L^+_\infty(D\times V)$, both having associated eigenvalue ${\rm e}^{\lambda_* t}$, and such that $\varphi$  (resp. $\tilde\varphi$) is uniformly (resp. a.e. uniformly) bounded away from zero on each compactly embedded subset of $D\times V$. In particular, for all $g\in L^+_{\infty}(D\times V)$,
\begin{equation}
\langle\tilde\varphi, \psi_t[g]\rangle = {\rm e}^{\lambda_* t}\langle\tilde\varphi,g\rangle\quad  \text{(resp. } 
\psi_t[\varphi] = {\rm e}^{\lambda_* t}\varphi
\text{)} \quad t\ge 0.
\label{leftandright}
\end{equation}
Moreover, there exists $\varepsilon>0$ such that 
\begin{equation}
\sup_{g\in L^+_\infty(D\times V): \norm{g}_\infty\leq 1}  \left\|{\rm e}^{-\lambda_* t}{\varphi}^{-1}{\psi_t[g]}-\langle\tilde\varphi, g\rangle\right\|_\infty = O({\rm e}^{-\varepsilon t}) \text{ as $t\rightarrow\infty$.}
\label{spectralexpsgp}
\end{equation}
%
\end{theorem}

This result differs significantly from what is already in the
literature principally through the assumptions on the cross-sections, the  strict positivity  properties and the uniform
boundedness of $\varphi, \tilde\varphi$ and uniformity in the mode of
convergence. In existing literature \eqref{spectralexpsgp} is usually given in the  $L_p$ setting, where $1<p<\infty$
is strictly enforced due to the nature of the $c_0$-semigroup
perturbation analysis; cf. \cite{DL6, M-K} and the discussion in \cite{MultiNTE}. 
\smallskip

The proof of Theorem \ref{CVtheorem}  is a
 non-trivial application of the recent theory of \cite{CV,
  CVecp} in that verifying their assumptions (which essentially leads
to the full statement of Theorem \ref{CVtheorem}) is highly
technical, taking account of the dimension of the system and the
piecewise linear (and hence irregular) nature of the neutron paths in
the underlying NBP.

\smallskip

{\color{black} Once again the assumptions (H3) and (H4) are unnecessary. As we shall shortly see, the result relies on the treatment of the sum of the operators $\bS+\bF$ as a single object, re-written as a scattering generator with action
\[
\bS' f(r,\upsilon) = \int_V (f(r,\upsilon) -f(r,\upsilon'))\alpha(r,\upsilon)\pi(r,\upsilon, \upsilon')\d\upsilon', \qquad r\in D, \upsilon \in V.
\]
 The assumption (H2)$^*$ is a condition on the intensity of this new generator. In this sense, the need for fission or for control of the pathwise behaviour of number of offspring (other than through their mean) is not needed.}

\section{Neutron random walk and many-to-one methodology}\label{sectNRW}
There is a second stochastic representation of the unique solution to \eqref{mild} which will form the basis of our proof of Theorem \ref{CVtheorem}. In order to describe it, we need to introduce the notion of a {\it neutron random walk} (NRW). 

\smallskip

A NRW on $D$ is defined by its scatter rates,
$\varsigma(r,\upsilon)$, $r\in D, \upsilon\in V$, and scatter probability
densities $\varpi(r,\upsilon,\upsilon')$,
$r\in D, \upsilon,\upsilon'\in V$ such that
$\textstyle{\int_V \varpi(r,\upsilon, \upsilon')\d\upsilon'}=1$ for all
$r\in D, \upsilon\in V$. Simply, when issued from $r$ with a velocity
$\upsilon$, the NRW will propagate linearly with that velocity until
either it exits the domain $D$, in which case it is killed, or at the
random time $T_{\texttt{s}}$ a scattering occurs, where
$ \Pr(T_{\texttt{s}}>t) = \exp\{-\textstyle{\int_0^t}
\varsigma(r+\upsilon s, \upsilon){\rm d}s \}, $ for $t\geq0.$ When the
scattering event occurs in position-velocity configuration
$(r,\upsilon)$, a new velocity $\upsilon'$ is selected with
probability $\varpi(r,\upsilon,\upsilon')\d\upsilon'$. If we denote by
$(R,\Upsilon) = ((R_t, \Upsilon_t), t\geq 0)$, the position-velocity
of the resulting continuous-time random walk on $D\times V$ with an
additional cemetery state $\{\dagger\}$ for when it leaves the domain
$D$, then it is easy to show that $(R,\Upsilon)$ is a Markov process.
Note, neither $R$ nor $\Upsilon$ alone is Markovian. We call the
process $(R,\Upsilon)$ an $\varsigma\varpi$-NRW. It is worth remarking that
when $\varsigma\varpi$ is given as a single rate function, the density
$\varpi$, and hence the rate $\varsigma$, is uniquely identified by 
normalising of the given product form $\varsigma\varpi$ to make it a probability distribution.

{\color{black}To  describe the second stochastic representation of \eqref{mild},  we  define
\begin{equation}
\beta(r,\upsilon)=\sigma_{\texttt{f}}(r,\upsilon)\left(\int_V\pi_{\texttt{f}}(r, \upsilon,\upsilon')\d\upsilon'-1\right)\geq -\sup_{r\in D, \upsilon\in V}\sigma_{\texttt{f}}(r,\upsilon)>-\infty,
\label{betadef}
\end{equation}
where the lower bound is due to assumption (H1).}
The following result was established in Lemma 7.1 of \cite{MultiNTE}.
\begin{lemma}[Many-to-one formula, \cite{MultiNTE}]\label{NRWrep} Under the assumptions of Lemma 
\ref{NBPrep}, we have the second representation 
\begin{equation}
\psi_t[g](r,\upsilon)  = \mathbf{E}_{(r,\upsilon)}\left[{\rm e}^{\int_0^t\beta(R_s, \Upsilon_s)\D s}g(R_t, \Upsilon_t) \mathbf{1}_{\{t < \tau^D\}}\right], \qquad t\geq 0,r\in D, \upsilon\in V,
\label{phi}
\end{equation}
where $\tau^D = \inf\{t>0 : R_t\not\in D\}$ and 
 ${\bf P}_{(r, v)}$ for the law of the $\alpha\pi$-NRW  starting from a single neutron with configuration $(r, \upsilon)$. 
\end{lemma}

{\color{black}Noting that $\bar\beta := \textstyle{\sup_{r\in D, \upsilon\in V}\beta(r,\upsilon)}<\infty$ thanks to (H1)}, let us introduce  $\texttt{P}^\dagger: = (\texttt{P}^\dagger_t, t\geq 0)$ for the expectation semigroup of the $\alpha\pi$-neutron random walk with potential $\beta$, such as is represented by the semigroup \eqref{phi}, but now killed at rate $ (\bar\beta - \beta)$. More precisely, for $g\in L^+_\infty (D\times V)$,
\begin{align}
\texttt{P}^\dagger_t[g](r,\upsilon) &= \psi_t[g](r,\upsilon){\rm e}^{-\bar\beta t}\notag\\
& =\mathbf{E}_{(r,\upsilon)}\left[{\rm e}^{\int_0^t (\beta (R_s, \Upsilon_s) - \bar\beta )\d s}g(R_t, \Upsilon_t) \mathbf{1}_{\{t<\tau^D\}}\right]\notag \\
&=\mathbf{E}_{(r,\upsilon)}\left[g(R_t, \Upsilon_t) \mathbf{1}_{\{t<\texttt{k}\}}\right] \notag\\
&= :\mathbf{E}^\dagger_{(r,\upsilon)}\left[g(R_t, \Upsilon_t)\right] ,%
\qquad t\geq 0, r\in D, \upsilon\in V,
\label{boldPdagger}
\end{align}
 where 
 \begin{equation}
 \texttt{k} =\inf\{t>0: \int_0^t (\bar\beta - \beta(R_s, \Upsilon_s) )\d s >\mathbf{e}\}\wedge \tau^D,
 \label{kill}
 \end{equation}
and $\mathbf{e}$ is an independent exponentially distributed random variable with mean 1.
\smallskip

We will naturally write  $\mathbf{P}^{\dagger}_{(r,\upsilon)}$ for  the (sub)probability measure associated to $\mathbf{E}^{\dagger}_{(r,\upsilon)}$, $r\in \bar D, \upsilon\in V$. The family $\mathbf{P}^{\dagger}: = (\mathbf{P}^{\dagger}_{(r,\upsilon)}, r\in \bar D, \upsilon\in V)$  now defines a Markov family of probability measures on the path space of the neutron random walk with cemetery state $\{\dagger\}$, which is where the path is sent when hitting the boundary $\partial D$ or the clock associated to the killing rate $\bar\beta - \beta$ rings.
{\it We note for future calculations that we can extend the domain of  functions  on $D\times V$ to accommodate taking a value on $\{\dagger\}$ by insisting that this value is always 0.}
\smallskip

Our strategy for proving Theorem \ref{CVtheorem}  thus boils down to understanding the evolution of the semigroup of the NRW $((R,\Upsilon), \mathbf{P}^{\dagger})$.  In this sense, we  see that the essence of Theorem \ref{CVtheorem} is, roughly speaking, a classical Perron-Frobenius-type problem for the semigroup of a Markov process; namely $((R,\Upsilon), \mathbf{P}^\dagger)$.
\smallskip

{\color{black}
It is also worthy of note that, given the role the $\alpha\pi$-NRW in the proof of Theorem \ref{CVtheorem}, we can also interpret the role of the assumptions (H1) and (H2)$^*$ in terms of this process. The condition (H1)  ensures that scattering cannot occur too fast. We can describe the condition (H2)$^*$ by saying that it  offers   `strong  irreducibility' of the $\alpha\pi$-NRW (where e.g. we could say that  (H2) only offers `weak irreducibility'). 

\smallskip

As alluded to previously, we can also see why the absence of the
assumption (H3) is not a problem. In the event that
e.g. $\sigma_{\texttt f}\pi_{\texttt f}$ is identically zero, the
original NBP is nothing more than a
$\sigma_{\texttt s}\pi_{\texttt s}$-NRW, i.e.
$\alpha\pi = \sigma_{\texttt s}\pi_{\texttt s}$ and $\beta\leq 0$. As
such the analysis in the proof of Theorem \ref{CVtheorem}, which
fundamentally concerns a NRW with a `strictly irreducible' state space
and killing is still valid. Similarly, the inclusion of (H4) is not
necessary as we only need control over the kernel $\alpha\pi$ for the
purpose of analysing the associated NRW and not the pathwise
behaviour of the otherwise associated NBP.}

\section{The ground state martingale}\label{groundstatemgsection}
As an application of the Perron-Frobenius behaviour of the linear semigroups discussed in Theorem \ref{CVtheorem}, we complete the summary of the main results of this paper by discussing how the existence of the right eigenfunction $\varphi$ plays directly into the existence of a classical (ground state) martingale for the underlying physical process. Analogues of this martingale   appear in the setting of all spatial branching processes and is sometimes referred to there as `the additive martingale' (see for example the recent monograph \cite{Shi} which discusses the analogous setting for branching random walks, or \cite{Bertbook2} for fragmentation processes). 

\smallskip

Under the assumptions of Theorem  \ref{CVtheorem} thanks to the semigroup property of \eqref{semigroup} and the invariance of $\varphi$ in Theorem \ref{CVtheorem}, it is now easy to see that 
\begin{equation}
W_t: = {\rm e}^{-\lambda_* t}\frac{\langle \varphi, X_t\rangle}{\langle\varphi, \mu\rangle}, \qquad t\geq 0,
\label{mgdef}
\end{equation}
is a unit mean martingale under $\mathbb{P}_\mu$ where $\mu\in\mathcal{M}(D\times V)$. It is worth noting that this claim is not so easy to make under analogues of Theorem \ref{CVtheorem} found in previous literature (cf. \cite{D, DL6, M-K}) as the setting of the eigenfunction $\varphi$ in an $L_p$ space would make it difficult to make sense of expectations of inner product $\langle \varphi, X_t\rangle$ without saying more about the mean semigroup of $(X_t, t\geq0)$.  

\smallskip

As a non-negative martingale, the almost sure limit of \eqref{mgdef} is assured. Our second main result tells us precisely when this martingale  limit is non-zero. Before stating the theorem, we require one more assumption on the fission rate and kernel, which is a stricter version of (H3). 

\smallskip
{\color{black}
{\bf  (H3)$^*$: There exists an open ball $B$, compactly embedded in $D$, such that
$$
\inf_{r \in B, \upsilon, \upsilon' \in V}\sigma_{\texttt{f}}(r, \upsilon)\pi_{\texttt{f}}(r, \upsilon, \upsilon') > 0.
$$}}

\begin{theorem}\label{Kesten} For the $(\sigma_{\emph{\texttt s}}, \pi_{\emph{\texttt s}}, \sigma_{\emph{\texttt f}}, \mathcal{P})$-NBP satisfying  (H1), (H2)$^*$ and (H4), we have the following cases for the martingale $W = (W_t, t\geq 0)$:
\begin{itemize}
\item[(i)]  If $\lambda_*>0$ and (H3) holds, then $W$ is $L_1(\mathbb{P})$ convergent;
\item[(ii)] If $\lambda_*<0$ and (H3) holds, then $W_\infty =0$ almost surely;
\item[(iii)] If $\lambda_*=0$ and (H3)$^*$ holds, then $W_\infty =0$ almost surely.
\end{itemize} 
\end{theorem}

Although a subtle inclusion, the following theorem also frames Theorem \ref{Kesten} in a more tidy way, showing the zero set of the martingale limit agrees with extinction.

\begin{theorem}\label{zeroset}
{\color{black}In each of the three cases of Theorem \ref{Kesten}, we also have that the events $\{W_\infty=0\}$ and $\{\zeta <\infty\}$ almost surely agree, where $\zeta = \inf\{t>0: \langle 1, X_t\rangle = 0\}$ is the time of extinction of the NBP. In particular, there is almost sure extinction if and only if $\lambda_*\leq 0$.}
\end{theorem}

As we can see from the above theorem, the critical case requires slightly more stringent conditions than the super- or sub-critical cases. However, it we assume the conditions of the critical case across the board, we get the aesthetically more pleasing corollary below.

\begin{corollary}
For the  $(\sigma_{\emph{\texttt s}}, \pi_{\emph{\texttt s}}, \sigma_{\emph{\texttt f}}, \mathcal{P})$-NBP satisfying (H1), (H2)$^*$, (H3)$^*$ and (H4), the martingale $W$ is $L_1(\mathbb{P})$ convergent if and only if $\lambda_*>0$ and otherwise $W_\infty=0$. Irrespective of $\lambda_*$, $\{W_\infty = 0\} = \{\zeta <\infty\}$ almost surely.
\end{corollary}

Note that, unlike many spatial branching process (e.g. the classical result of \cite{Biggins}), there is no `$xlog x$' condition thanks to the assumption (H2) and a precise dichotomy on $\lambda_*$ emerges.
The result mimics a behavioural trait that has been observed for branching diffusions in compact domains in e.g. \cite{EK}.   In essence it states that in the competing physical processes of fission, transport, scattering and absorption, it is the lead eigenvalue which dictates growth or decay of mass. 
In this respect we can also mimic other similar results in the spatial branching process literature  (cf. \cite{Ik, HH, KM}), the proof of which falls out of the proof of Theorem \ref{Kesten}.
{\color{black}
\begin{corollary}\label{L2}
For the  $(\sigma_{\emph{\texttt s}}, \pi_{\emph{\texttt s}}, \sigma_{\emph{\texttt f}}, \mathcal{P})$-NBP satisfying the assumptions (H1), (H2)$^*$, (H3) and (H4), when $\lambda_*>0$, the martingale $(W_t, t\geq 0)$ is $L_2(\mathbb{P})$ convergent.
\end{corollary}
}

It is particularly interesting to note that in the setting of a critical system, $\lambda_* = 0$, which is typically what is envisaged for a nuclear reactor, the above results evidences the hypothesis that the fission process eventually dies out (similarly to other examples of critical branching processes).

\smallskip
To verify the aforementioned hypothesis rigorously, one needs an almost sure growth result for the particle system which would take the format 
\begin{equation}
\lim_{t\to\infty}{\rm e}^{-\lambda_* t} \frac{\langle g, X_t \rangle}{\langle \varphi, \mu\rangle} = \langle g, \tilde\varphi\rangle W_\infty,
\label{SLLN}
\end{equation}
$\mathbb{P}_\mu$-almost surely, for all $g\in L^+_\infty(D\times V)$. This is a much more difficult result than the one stated in Theorem \ref{Kesten} and is addressed in a second instalment to this paper; see \cite{SNTE-II}. The reader should note that \eqref{SLLN} verifies what has been known in the nuclear industry for a long time. Namely that critical nuclear reactors will not persist in energy generation, but will eventually cease working, corresponding to the case that $W_\infty = 0$.

\section{Neutron random walk and spine decomposition} As with many spatial branching processes, the most efficient way to analyse martingale convergence is through the pathwise behaviour of the particle system (known as a {\it spine decomposition}) when considered under a change of measure induced by the martingale itself. 
 Whilst classical in the branching process literature, it is unknown in the setting of neutron transport.
 We will devote the remainder of this section to describing the pathwise spine decomposition of the physical process, our final main contribution. 
 
\smallskip

We are  interested in the change of measure 
\begin{equation}
\left.\frac{\d \mathbb{P}^\varphi_{\mu}}{\d\mathbb{P}_{\mu}}\right|_{\mathcal{F}_t}  =
 W_t, \qquad t\geq 0,
\label{mgCOM}
\end{equation}
for the NBP with  characteristics  $\sigma_{\texttt s}, \pi_{\texttt s}, \sigma_{\texttt f}, \mathcal{P}$ (cf. Remark \ref{nonuniqueNBP}), where $\mu$ belongs to the space of finite atomic measures $\mathcal{M}(D\times V)$.
\smallskip

In the next theorem we will formalise an understanding of this change of measure in terms of another $\mathcal{M}(D\times V)$-valued stochastic process 
 \begin{equation}
 X^\varphi := (X^\varphi_t, t\geq 0)\text{ with probabilities }\tilde{\mathbb{P}}^\varphi: = (\tilde{\mathbb{P}}^\varphi_{\mu}, \mu \in \mathcal{M}(D\times V)),
 \label{dressedprocess}
 \end{equation}
 which we will now describe through an algorithmic construction. 
 \begin{itemize}
\item[1.] From the initial configuration $\mu\in \mathcal{M}(D\times V)$ with an arbitrary enumeration of particles, the $i$-th neutron is selected and marked `{\it spine}' with empirical probability 
\[
\frac{\varphi(r_i, \upsilon_i)}{\langle\varphi, \mu\rangle}.
\]
\item[2.] The neutrons $j\neq i$ in the initial configuration that are not marked `{\it spine}', each issue independent copies of $(X, \mathbb{P}_{\delta_{(r_j, \upsilon_j)}})$ respectively.
\item [3.] For the marked neutron, issue a NRW characterised by the rate function 
\[
\sigma_{\emph{\texttt s}}(r,\upsilon)\frac{\varphi(r,\upsilon')}{\varphi(r,\upsilon)}\pi_{\emph{\texttt s}}(r,\upsilon,\upsilon'), \qquad r\in D, \upsilon,\upsilon'
\in V.
\]

\item[4.] The marked neutron undergoes fission at the accelerated rate  
$\varphi(r,\upsilon)^{-1}( \bF +\sigma_{\texttt f}{\texttt I})\varphi(r,\upsilon)$, when  in physical configuration $(r,\upsilon)\in D\times V$, at which point, it scatters a random number of particles according to the random measure on $V$ given by $(\mathcal{Z}, {\mathcal P}^\varphi_{(r,\upsilon)})$ where 
\begin{equation}
\frac{\d{\mathcal P}^\varphi_{(r,\upsilon)}}{\d {\mathcal P}_{(r,\upsilon)}} = \frac{\langle\varphi, \mathcal{Z}\rangle}{{\mathcal E}_{(r,\upsilon)}[\langle\varphi, \mathcal{Z}\rangle]}.
\label{pointprocessCOM}
\end{equation}
\item[5.] When fission of the marked neutron occurs in physical configuration $(r,\upsilon)\in D\times V$, set
\[
\mu = \sum_{i = 1}^n\delta_{(r,\upsilon_i)},\text{ where, in the previous step, }\mathcal{Z} = \sum_{i = 1}^n\delta_{\upsilon_i},
\]
and repeat  step 1.
\end{itemize}

 The process $X^\varphi_t$ describes the physical configuration (position and velocity) of all the particles in the system at time $t$, for $t\geq 0$ (i.e. ignoring the marked genealogy). We will also be interested in the configuration of the single genealogical line of descent which has been marked `{\it spine}'. This process, referred to simply as {\it the spine}, will be denoted by $(R^\varphi, \Upsilon^\varphi): = ((R^\varphi_t, \Upsilon^\varphi_t), t\geq 0)$. Together, the processes $(X^\varphi, (R^\varphi,\Upsilon^\varphi))$ make a Markov pair, whose probabilities we will denote by $(\tilde{\mathbb{P}}^\varphi_{\mu, (r,\upsilon)}, \mu \in \mathcal{M}(D\times V), (r,\upsilon)\in V\times D)$. Note in particular that 
 \[
 \tilde{\mathbb{P}}^\varphi_{\mu} = \sum_{i = 1}^n\frac{\varphi (r_i,\upsilon_i)}{\langle\varphi, \mu\rangle}\tilde{\mathbb{P}}^\varphi_{\mu, (r_i,\upsilon_i)} 
 \]
 when $\mu = \textstyle{\sum_{i =1}^n\delta_{(r_i, \upsilon_i)}}$.
 
\begin{theorem}\label{pathspine} 
Under assumptions (H1), (H2) and (H4), the process $(X^\varphi, \tilde{\mathbb{P}}^\varphi)$ is Markovian and equal in law to $(X, \mathbb{P}^\varphi)$, where  $ \mathbb{P}^\varphi = (\mathbb{P}^\varphi_\mu, \mu\in \mathcal{M}(D\times V))$. 

\end{theorem}

It is also worth understanding the dynamics of the spine $(R^\varphi, \Upsilon^\varphi)$. For convenience, let us denote the family of probabilities of the latter by $\tilde{\mathbf{P}}^\varphi = (\tilde{\mathbf{P}}^\varphi_{(r,\upsilon)}, (r,\upsilon)\in D\times V)$, in other words, the marginals of  $(\tilde{\mathbb{P}}^\varphi_{\mu, (r,\upsilon)}, \mu \in \mathcal{M}(D\times V), (r,\upsilon)\in V\times D)$. 

\smallskip

\smallskip

Next we define the probabilities 
 $\mathbf{P}^\varphi  : =(\mathbf{P}^\varphi _{(r,\upsilon)}, (r,  \upsilon)\in  D\times V)$ to describe the law of an $\alpha^\varphi \pi^\varphi$-NRW, where
\begin{equation}
\alpha^\varphi(r,\upsilon)\pi^\varphi(r,\upsilon,\upsilon') = \frac{\varphi(r,\upsilon')}{\varphi(r,\upsilon)}\left(
\sigma_{\texttt{s}}(r,\upsilon)\pi_{\texttt{s}}(r, \upsilon, \upsilon') + \sigma_{\texttt{f}}(r,\upsilon) \pi_{\texttt{f}}(r, \upsilon, \upsilon')
\right),
\label{spineap}
\end{equation}
for $r\in D$, $\upsilon,\upsilon'\in V$
 We are now ready to identify the spine.

 \begin{lemma}\label{spinemarkov}Under assumptions (H1), (H2) and (H4), the process $((R^\varphi, \Upsilon^\varphi), \tilde{\mathbf{P}}^\varphi )$ is a NRW equal in law to $((R, \Upsilon), \mathbf{P}^\varphi )$ and, moreover, 
  \begin{equation}
\left.\frac{\d \mathbf{P}^\varphi_{(r,\upsilon)}}{\d \mathbf{P}_{(r,\upsilon)}}\right|_{\mathcal{F}_t} =
{\rm e}^{-\lambda_* t +\int_0^t \beta(R_s, \Upsilon_s)\d s}\frac{\varphi(R_t, \Upsilon_t)}{\varphi (r,\upsilon)}\mathbf{1}_{\{t<\tau^D\}},\qquad t\geq 0,r\in D,\upsilon\in V,
\label{NRWCOM}
\end{equation}
from which we deduce that $((R, \Upsilon), \mathbf{P}^\varphi )$ is conservative with a stationary distribution $\varphi\tilde\varphi(r,\upsilon)\d r\d\upsilon$ on $D\times V$.
(Recall that $(R,\Upsilon)$ under $\mathbf{P}$ is the $\alpha\pi$-NRW that appears in the many-to-one Lemma \ref{NRWrep}.)
\end{lemma}

\smallskip

Now that we have stated all of our main results, it is worth noting that, in places, the analysis echoes very similar issues that have very recently appeared in the analysis of growth-fragmentation equations, see e.g. \cite{BW} and \cite{Bnew}, and for good reason. Growth-fragmentation equations, although dealing with a particle system in which particles' mass is positive-valued and for which there is no consideration of classical `velocity', the dynamics of fragmentation  shares the phenomenon of non-local branching. This explains the appearance of  integral operators. Moreover, a combination of L\'evy-type and piecewise linear movement of particles in the growth-fragmentation setting also mirrors the phenomenon of advection and scattering in the NTE and the associated operators.

\smallskip

In the rest of the paper we prove Theorem \ref{CVtheorem}, Theorem \ref{pathspine}, Lemma \ref{spinemarkov}, Theorem \ref{Kesten} and Corollary \ref{L2} in that order.

\section{Proof of Theorem \ref{CVtheorem}}\label{pf7.1}

Our approach to  proving Theorem \ref{CVtheorem} will be to extract the existence of the eigentriple $\lambda_*$, $\varphi$ and  $\tilde\varphi$ for the expectation semigroup $(\psi_t, t\geq0)$ from the existence of a similar triple of the semigroup $(\texttt{P}^\dagger_t, t\geq 0)$ defined in~\eqref{boldPdagger}.
Indeed, from \eqref{boldPdagger}, it is clear that when the latter exists, the eigenfunctions of the former are the same and the eigenvalues differ  only by the constant $\overline\beta$.
\smallskip

{\color{black}{\it Throughout this section, we assume the assumptions of Theorem \ref{CVtheorem} are in force.}}
\smallskip

As alluded to earlier, what lies at the  core of our proof is the  general result of  Theorem 2.1  and Proposition 2.3 of \cite{CV} and Theorem~2.1 and the discussion around (1.5) of \cite{CVecp}, which, combined in the current context,  reads as follows. 

\begin{theorem}\label{7CVtheoremBis}
Suppose that there exists a probability measure $\nu$ on $D\times V$ such that
\begin{enumerate}
\item[\namedlabel{itm:A1}{(A1)}] there exist $t_0$, $c_1 > 0$ such that for each $(r, \upsilon)\in D\times V$,
\[
\mathbf{P}_{(r, \upsilon)}((R_{t_0}, \Upsilon_{t_0}) \in \cdot \;| t_0 < \emph{\texttt{k}}) \ge c_1 \nu(\cdot);
\]
\item[\namedlabel{itm:A2}{(A2)}] there exists a constant $c_2 > 0$ such that for each $(r, \upsilon)\in D\times V$ and for every $t\ge 0$,
\[
\mathbf{P}_{\nu}(t< \emph{\texttt{k}}) \ge c_2\mathbf{P}_{(r, \upsilon)}(t <\emph{ \texttt{k}}),
\]
\end{enumerate}
where $\emph{ \texttt{k}}$ was defined in~\eqref{kill}.
Then, there exists $\lambda_c < 0$ such that, there exists an
eigenmeasure $\eta$ on $D\times V$ and a positive right eigenfunction
$\varphi$ of $\emph{\texttt{P}}^\dagger$
with eigenvalue
${\rm e}^{\lambda_c t} $, such that $\eta$ is a probability
measure and $\varphi\in L^+_\infty(D\times V) $, i.e. for all
$g\in L_{\infty}(D\times V)$
\begin{equation}
\eta [\emph{\texttt{P}}^{\dagger}_t[g] ]= {\rm e}^{\lambda_c t}\eta[g]\quad  \text{and}\quad 
 \emph{\texttt{P}}^{\dagger}_t[\varphi] = {\rm e}^{\lambda_c t}\varphi
 \quad t\ge 0.
 \label{eta}
\end{equation}
Moreover, there exist $C,\varepsilon>0$ such that, for $g\in L^+_\infty(D\times V)$ and $t$ sufficiently large (which does not depend on $g$),
\begin{equation}
\left\| {\rm e}^{-\lambda_c t}\varphi^{-1}\emph{\texttt{P}}_t^{\dagger}[g]-\eta[g]\right\|_\infty\leq C{\rm e}^{-\varepsilon t}\|g\|_\infty.
\label{7spectralexpsgp}
\end{equation}
In particular, setting $g  \equiv 1$, as $t\to\infty$,
\begin{equation}
\left\| {\rm e}^{-\lambda_c t} \varphi^{-1}\mathbf{P}_{\cdot}(t < \emph{\texttt{k}}) - 1\right\|_\infty\leq C{\rm e}^{-\varepsilon t}.
\label{7die}
\end{equation}
\end{theorem}

We aim to prove that assumptions \ref{itm:A1} and \ref{itm:A2} are
satisfied, so that the conclusions of the above theorem hold. Then we
prove that $\varphi$ is uniformly bounded away from $0$ on each
compactly embedded subset of $D\times V$ and that $\eta$ admits a
positive bounded density with respect to the Lebesgue measure on
$D\times V$ (see Lemma~\ref{QSDdensity}), which concludes the proof of
Theorem~\ref{CVtheorem}.  In order to do so, we start by introducing two  
alternative assumptions to (A1) and (A2):

\smallskip

There exists an $ \varepsilon>0$ such that
\begin{itemize}
\item[\namedlabel{itm:B1}{(B1)}] $\textstyle{D_{\varepsilon} \coloneqq \{r \in D : \inf_{y\in \partial D}|r - y| > \varepsilon{\texttt v}_{\texttt{max}}\}}$ is non-empty and connected.
\item[\namedlabel{itm:B2}{(B2)}] there exist $0 < s_{\varepsilon} < t_{\varepsilon}$ and ${\color{black}\gamma} > 0$ such that, for all $r \in D \backslash D_{\varepsilon}$, there exists $K_r \subset V$ measurable such that $\text{Vol}({K_r}) \ge {\color{black}\gamma} >0$ and for all $\upsilon \in K_r$, $r + \upsilon s \in D_{\varepsilon}$ for every $s \in [s_{\varepsilon}, t_{\varepsilon}]$ and $r+\upsilon s \notin \partial D$ for all $s\in [0, s_{\varepsilon}]$.
\end{itemize}

\smallskip

It is easy to verify that \ref{itm:B1} and \ref{itm:B2} are implied when we assume that $D$ is a non-empty  convex domain, as we have done in the introduction. They are also satisfied if, for example, the boundary of $D$ is a smooth, connected, compact manifold and $\varepsilon$ is sufficiently small. Geometrically, \ref{itm:B2} means that each of the sets
\begin{equation}
L_r \coloneqq \left\{z \in \mathbb{R}^3 : \frac{\Vert z-r \Vert}{\Vert \upsilon \Vert} \in [s_{\varepsilon}, t_{\varepsilon}], \upsilon \in K_r \right\},\qquad r\in D\backslash D_\varepsilon
\label{L_r}
\end{equation}
is included in $D_{\varepsilon}$ and has Lebesgue measure at least ${{\color{black}\gamma}}(t_{\varepsilon}^2 - s_{\varepsilon}^2)/{2}$. Roughly speaking, for each $r\in D$ which is within $\varepsilon {\texttt v}_{\texttt{max}}$ of the boundary $\partial D$, $L_r$ is the set of points  from which one can issue a neutron with a velocity chosen from $\upsilon\in K_r$ such that (ignoring scattering and fission) we can   ensure that it passes through $D\backslash D_\varepsilon$  during the time interval $[s_\varepsilon,t_\varepsilon]$.

\smallskip
Our proof of Theorem \ref{CVtheorem} thus  consists of proving that assumptions \ref{itm:B1} and \ref{itm:B2} imply assumptions \ref{itm:A1} and \ref{itm:A2}. Our method is motivated by~\cite[Section 4.2]{CV}, however, we note that our approach accommodates for the more general setting we have here (e.g. $V \subset \mathbb{R}^3$ is bounded and $d=3$) at the cost of greater technicalities.

\smallskip

We begin by considering several technical lemmas. The first is a straightforward consequence of $D$ being a bounded subset of $\mathbb{R}^3$.
\begin{lem}\label{TB}
Let $B(r, \upsilon)$ be the ball in $\mathbb{R}^3$ centred at $r$ with radius $\upsilon$.
\begin{enumerate}
\item[(i)] There exists an integer ${\color{black}\mathfrak{n}} \ge 1$ and $r_1, \dots, r_n \in D_{\varepsilon}$ such that $D_{\varepsilon} \subset \bigcup_{i=1}^{\color{black}\mathfrak{n}}B(r_i, \emph{{\texttt v}}_{\emph{\texttt{max}}}\varepsilon/32)$ and $D_{\varepsilon} \cap B(r_i, \emph{{\texttt v}}_{\emph{\texttt{max}}}\varepsilon/32){\color{black}\neq \emptyset}$ for each $i \in \{1, \dots, {\color{black}\mathfrak{n}}\}$.
\item[(ii)] For all $r, r' \in D_{\varepsilon}$, there exists $m \le n$ and $i_1, \dots, i_m$ distinct in $\{1, \dots, {\color{black}\mathfrak{n}}\}$ such that $r\in B(r_{i_1}, \emph{{\texttt v}}_{\emph{\texttt{max}}}\varepsilon/32)$, $r'\in B(r_{i_m}, \emph{{\texttt v}}_{\emph{\texttt{max}}}\varepsilon/32)$ and for all $1 \le j \le m-1$, $B(r_{i_j}, \emph{{\texttt v}}_{\emph{\texttt{max}}}\varepsilon/32) \cap B(r_{i_{j+1}}, \emph{{\texttt v}}_{\emph{\texttt{max}}}\varepsilon/32) \neq \emptyset$.
\end{enumerate}
\end{lem}

Heuristically, the above lemma ensures that there is a universal covering of $D_\varepsilon$ by the balls  $B(r_i, {\texttt v}_{\texttt{max}}\varepsilon/32)$, $1\leq i\leq {\color{black}\mathfrak{n}}$ such that between any two points $r, r'$ in $D_\varepsilon$, there is a sequence of overlapping balls  $B(r_{i_1}, {\texttt v}_{\texttt{max}}\varepsilon/32),\cdots, B(r_{i_m}, {\texttt v}_{\texttt{max}}\varepsilon/32)$ that one may pass through in order to get from $r$ to $r'$.

\smallskip

The next lemma provides a minorization of the law of $(R_t, \Upsilon_t)$ under $\pdag$. The result is similar to~\cite[Lemma 4.5]{CV}, however, we provide a less geometrical proof by considering a change of variables from Cartesian to polar coordinates. In the statement of the lemma, we use $ \texttt{dist}(r, \partial D)$ for the distance of $r$ from the boundary $\partial D$. 
\smallskip

{\color{black}
Define $
\textstyle{
\underline{\alpha} =  \inf_{r\in D, \upsilon \in V}\alpha(r,\upsilon)>0
}$ and $\textstyle{\underline\pi = \inf_{r\in D, \upsilon, \upsilon' \in V}\pi(r,\upsilon, \upsilon')}$.
We will also similarly write $\overline{\alpha}$ and $\overline{\pi}$ with obvious meanings.
We note that due to the assumption (H1) we have $\overline\alpha<\infty$ and $\overline\pi<\infty$ and hence, combining this with (H2)$^*$ it follows that,
\[
\underline{\alpha} = \frac{1}{\overline\pi}\inf_{r\in D, \upsilon\in V} \alpha(r,\upsilon)\overline{\pi} \geq \frac{1}{\overline\pi}\inf_{r\in D, \upsilon, \upsilon'\in V} \alpha(r,\upsilon)\pi(r,\upsilon,\upsilon')>0,
\]
and a similar calculation shows that $\underline\pi>0$.
}
\begin{lem}\label{minlem}
For all $r \in D$, $\upsilon \in V$ and $t > 0$ such that ${\texttt v}_{\emph{\texttt{max}}}t < \emph{\texttt{dist}}(r, \partial D)$, the law of $(R_t, \Upsilon_t)$ under $\mathbf{P}_{(r, \upsilon)}^{\dagger}$, defined in~\eqref{boldPdagger}, satisfies
\begin{align}
\mathbf{P}_{(r, \upsilon)}^{\dagger}(R_t \in {\rm d}z, \Upsilon_t \in {\rm d}\upsilon) &\ge \frac{C{\rm e}^{-\overline{\alpha} t}}{t^2}\bigg[\frac{t}{2}\left(\emph{{\texttt v}}_{\emph{\texttt{max}}}^2 - \left(\emph{{\texttt v}}_{\emph{\texttt{min}}}\vee \frac{|z-r|}{t}\right)^2\right)\notag\\
 &\quad -|z-r|\left(\emph{{\texttt v}}_{\emph{\texttt{max}}} - \emph{{\texttt v}}_{\emph{\texttt{min}}}\vee \frac{|z-r|}{t}\right)\bigg]\mathbf{1}_{\{z\in B(r, \emph{{\texttt v}}_{\emph{\texttt{max}}}t\})}\,{\rm d}z\,{\rm d}\upsilon,
\label{minor}
\end{align}
where  $C > 0$ is a positive constant.
\end{lem}
\begin{proof}
{\color{black}Fix $r_0\in D$}.
Let $J_k$ denote the $k^{th}$ jump time of $(R_t, \Upsilon_t)$ under $\mathbf{P}_{(r, \upsilon)}^{\dagger}$ and let $\Upsilon_0$ be uniformly distributed on $V$. Assuming that ${\texttt v}_{\texttt{max}}t < \texttt{dist}({\color{black}r_0}, \partial D)$, we first give a minorization of the density of $(R_t, \Upsilon_t)$, with initial configuration $(r_0, \Upsilon_0)$, on the event $\left\{J_1 \le t < J_2\right\}$. Note that, on this event, we have 
\[
R_t =r_0+J_1\Upsilon_0 + (t-J_1)\Upsilon_{J_1},
\]
where $\Upsilon_{J_1}$ is the velocity of the process after the first jump. Then
{\color{black}
\begin{align}
\mathbf{E}^\dagger_{(r_0, \Upsilon_0)}&[f(R_t, \Upsilon_t)\mathbf{1}_{\{J_1 \le t < J_2\}}]
\notag\\
&  = \int_0^t \D s \int_V\D \upsilon_0\int_V \D \upsilon_1\alpha(r_0+\upsilon_0 s, \upsilon_0){\rm e}^{-\int_0^s\alpha(r_0+\upsilon_0 u, \upsilon_0)\D u}{\rm e}^{-\int_0^{t-s}\alpha(r_0+\upsilon_0 s + \upsilon_1u, \upsilon_1)\D u}\notag\\
&\hspace{5cm}\times \pi(r_0+\upsilon_0 s, \upsilon_0, \upsilon_1)f(r_0+\upsilon_0 s + (t-s)\upsilon_0, \upsilon_1)\notag\\
&\ge \underline{\alpha}{\rm e}^{-\overline{\alpha} t}\underline{\pi}\int_V\D \upsilon_1\int_0^t \D s \int_V{\rm d}\upsilon_0 f(r_0+s\upsilon_0 + (t-s)\upsilon_1, \upsilon_1),\label{LB1}
\end{align}
}
where we have used the bounds on $\alpha$ and $\pi$. We now make the change of variables $\upsilon_0 \mapsto (\rho_0, \theta_0, \varphi_0)$ and $\upsilon_1 \mapsto (\rho_1, \theta_1, \varphi_1)$ so that~\eqref{LB1} becomes
{\color{black}
\begin{align}
\mathbf{E}^\dagger_{(r_0, \Upsilon_0)}&[f(R_t, \Upsilon_t)\mathbf{1}_{\{J_1 \le t < J_2\}}] \notag\\
&\ge C_1 \underline{\alpha}{\rm e}^{-\overline{\alpha} t}\underline{\pi} \int_0^t \D s\int_{v_{\texttt{min}}}^{v_{\texttt{max}}}\D \rho_1 \int_0^\pi\D \varphi_1 \int_0^{2\pi}\D \theta_1 \int_{v_{\texttt{min}}}^{v_{\texttt{max}}}\D \rho_0 \int_0^\pi\D \varphi_0 \int_0^{2\pi}\D \theta_0 \label{LB2}\\
&\hspace{3cm}f(r_0+\Theta_{\rho_0, \rho_1, \theta_1, \varphi_1}(s, \theta_0, \varphi_0), \widetilde{\Theta}(\rho_1, \theta_1, \varphi_1))\Delta(\rho_0, \theta_0, \varphi_0)\Delta(\rho_1, \theta_1, \varphi_1)\notag,
\end{align}
}
where 
\begin{align}
    \Theta_{\rho_0, \rho_1, \theta_1, \varphi_1}(s, \theta_0, \varphi_0) &= \begin{bmatrix}
           s\rho_0\sin\varphi_0\cos\theta_0 + (t-s)\rho_1\sin\varphi_1\cos\theta_1 \\
           s\rho_0\sin\varphi_0\sin\theta_0 + (t-s)\rho_1\sin\varphi_1\sin\theta_1\\
           s\rho_0\cos\varphi_0 +  (t-s)\rho_1\cos\varphi_1
         \end{bmatrix},
\label{cov}
\end{align}
represents the spatial variable $s\upsilon_0 + (t-s)\upsilon_1$ in polar coordinates,
\begin{align}
	\widetilde{\Theta}(\rho_1, \theta_1, \varphi_1)&=
	\begin{bmatrix}
		\rho_1\sin\varphi_1\cos\theta_1\\
		\rho_1\sin\varphi_1\sin\theta_1\\
		\rho_1\cos\varphi_1
	\end{bmatrix}
\label{cov2}
\end{align}
represents $\upsilon_1$ in polar coordinates,
\begin{equation}
\Delta(\rho, \theta, \varphi) = \rho^2\sin\varphi,
\label{det1}
\end{equation}
is the determinant of the Jacobian matrix for the change of variables from Cartesian to polar coordinates, and $C_1$ is an unimportant normalising constant. 
\smallskip

For fixed $\rho_0$, $\rho_1$, $\theta_1$ and $\varphi_1$, we first consider the part of~\eqref{LB2} given by
{\color{black}
\begin{equation}
(s, \theta_0, \varphi_0) \mapsto \int_0^t\D s \int_0^{\pi}\D \varphi_0 \int_0^{2\pi}\D \theta_0 f(r_0+\Theta_{\rho_0, \rho_1, \theta_1, \varphi_1}(s, \theta_0, \varphi_0), \widetilde{\Theta}(\rho_1, \theta_1, \varphi_1))\Delta(\rho_0, \theta_0, \varphi_0),
\label{innera}
\end{equation}
}
The Jacobian of $\Theta_{\rho_0, \rho_1, \theta_1, \varphi_1}$, as a function of $(s, \theta_0, \varphi_0)$, is given by
\begin{equation*}
	\begin{bmatrix}
		\rho_0\cos\theta_0\sin\varphi_0 - \rho_1\cos\theta_1\sin\varphi_1 & -s\rho_0\sin\theta_0\sin\varphi_0 & s\rho_0\cos\varphi_0\cos\theta_0\\
		\rho_0\sin\theta_0\sin\varphi_0 - \rho_1\sin\theta_1\sin\varphi_1 & s\rho_0\cos\theta_0\sin\varphi_0 & s\rho_0\cos\varphi_0\sin\theta_0 \\
		\rho_0\cos\varphi_0 - \rho_1\cos\varphi_1 & 0 & -s\rho_0\sin\varphi_0
	\end{bmatrix}.
\end{equation*}
whose determinant, $\texttt{det}(D_{\rho_0, \rho_1, \theta_1, \varphi_1}(s, \theta_0, \varphi_0))$ satisfies
{\color{black}
\[
\frac{\Delta(\rho_0, \theta_0, \varphi_0)}{\texttt{det}(D_{\rho_0, \rho_1, \theta_1, \varphi_1}(s, \theta_0, \varphi_0))} \ge \frac{1}{4s^2{\texttt v}^3_{\texttt{max}}}\geq \frac{1}{4t^2{\texttt v}^3_{\texttt{max}}}, \qquad s\leq t.
\]
}
We thus have the following lower bound for~\eqref{innera}
\begin{align}
  &
    \frac{1}{4t^2{\texttt v}^3_{\texttt{max}}}\int_0^t\D s \int_0^{\pi}\D \varphi_0 \int_0^{2\pi}\D \theta_0 f(r_0+\Theta_{\rho_0, \rho_1, \theta_1, \varphi_1}(s, \theta_0, \varphi_0), \widetilde{\Theta}(\rho_1, \theta_1, \varphi_1))\label{inner2}\\
&\quad\hspace{6cm} \times\texttt{det}(D_{\rho_0, \rho_1, \theta_1, \varphi_1}(s, \theta_0, \varphi_0)).\notag
\end{align}
Making another change of variables $(s, \theta_0, \varphi_0) \mapsto r \in \mathbb{R}^3$ and using the fact that, regardless of the values of $\rho_1$, $\theta_1$ and $\varphi_1$, $\Theta_{\rho_0, \rho_1, \theta_1, \varphi_1}$ maps $(0,t) \times (0,\pi) \times (0,2\pi)$ surjectively onto a set that contains {\color{black}${\color{black}B(\rho_0 t)}$, where $B(r)$ is the ball in $\mathbb{R}^3$ of radius $r$ centred at the origin}, ~\eqref{inner2}, and hence \eqref{innera}, is bounded below by
\begin{equation}
\frac{1}{4t^2\upsilon^3_{\texttt{max}}}\int_{{\color{black}B(\rho_0 t)}}f(r, \widetilde{\Theta}(\rho_1, \theta_1, \varphi_1))\D r.
\label{inner3}
\end{equation}
Substituting this equation back into~\eqref{LB2} and changing $(\rho_1, \theta_1, \varphi_1)$ back to Cartesian coordinates, we have
\begin{equation}
\mathbf{E}_{(r_0,\Upsilon_0)}^\dagger[f(R_t, \Upsilon_t)\mathbf{1}_{\{J_1 \le t < J_2\}}] \ge \frac{C_2{\rm e}^{-\overline{\alpha} t}}{t^2}\int_{{\texttt v}_{\texttt{min}}}^{{\texttt v}_{\texttt{max}}}\D \rho_0\int_{{\color{black}B(\rho_0 t)}}\D r\int_V\D \upsilon_1f(r, \upsilon_1),
\label{LB3}
\end{equation}
where  $C_2 = \underline{\alpha}\underline{\pi}C_1/(4{\texttt v}_{\texttt{max}}^3)$. 
\smallskip

Now suppose we fix an initial configuration $(r_0, \upsilon_0) \in D \times V$, with $t{\texttt v}_{\texttt{max}} < \texttt{dist}(r_0, \partial D)$. By considering the event $\left\{J_2 \le t < J_3\right\}$ and noting that the scattering kernel is bounded below by $\underline{\pi}$, we may apply the Markov property together with~\eqref{LB3} to the process at time $J_1$ before choosing the new velocity. Using the bounds on $\alpha$ and $\pi$ as before, and recalling that $\Upsilon_0$ is uniformly distributed, we have
\begin{align}
\mathbf{E}^\dagger_{(r_0, \upsilon_0)}&[f(R_t, \Upsilon_t)\mathbf{1}_{\{J_2 \le t < J_3\}}] \notag\\
&\ge \int_0^t \D s \, \underline{\alpha}{\rm e}^{-\overline{\alpha}s}\underline{\pi}\mathbf{E}^\dagger_{(r_0 + s\upsilon_0, \Upsilon_0)}[f(R_{t-s}, \Upsilon_{t-s})\mathbf{1}_{\{J_1 \le t-s < J_2\}}]\notag\\
&\ge \int_0^t \D s \, \underline{\alpha}{\rm e}^{-\overline{\alpha}s}\underline{\pi}\frac{C_2{\rm e}^{-\overline{\alpha} (t-s)}}{(t-s)^2}\int_V\D \upsilon_1\int_{{\texttt v}_{\texttt{min}}}^{{\texttt v}_{\texttt{max}}}\D \rho_0\int_{\rho_0(t-s)B}\D r f(r_0 + s\upsilon_0 + r, \upsilon_1)\notag \\
&\ge \frac{C_3{\rm e}^{-\overline{\alpha} t}}{t^2}\int_0^t{\rm d}s\int_{V}\D\upsilon_1\int_{{\texttt v}_{\texttt{min}}}^{{\texttt v}_{\texttt{max}}}\D\rho_0\int_{\rho_0(t-s)B}{\rm d}rf(r_0 +s\upsilon_0+r, \upsilon_1)\notag\\
&= \frac{C_3{\rm e}^{-\overline{\alpha} t}}{t^2}\int_0^t{\rm d}s\int_{V}\D\upsilon_1\int_{{\texttt v}_{\texttt{min}}}^{{\texttt v}_{\texttt{max}}}\D\rho_0\int_{r_0 + s\upsilon_0 + \rho_0(t-s)B}{\rm d}yf(y, \upsilon_1),
\label{LB5}
\end{align}
where we have used the substitution $y = r_0 + s\upsilon_0 + r$ to obtain the final line and $C_3$ is another constant in $(0,\infty)$. Now note that for $s \le {\rho_0t}/({\rho_0 + {\texttt v}_{\texttt{max}}})$ we have $r_0 + {\color{black}B(\rho_0 t - (\rho_0 + {\texttt v}_{\texttt{max}})s)} \subset r_0 + s\upsilon_0 + {\color{black}B(\rho_0(t-s))}$. Combining this with~\eqref{LB5} and using Fubini, we have
\begin{align}
\mathbf{E}^\dagger_{(r_0, \upsilon_0)}&[f(R_t, \Upsilon_t)\mathbf{1}_{\{J_2 \le t < J_3\}}]\notag \\
&\ge \frac{C_3{\rm e}^{-\overline{\alpha} t}}{t^2}\int_V\D \upsilon_1\int_{{\texttt v}_{\texttt{min}}}^{{\texttt v}_{\texttt{max}}}\D\rho_0\int_{\mathbb{R}}\mathbf{1}_{\left\{0 \le s \le \frac{\rho_0}{\rho_0 + {\texttt v}_{\texttt{max}}}t\right\}}{\rm d}s\int_{\mathbb{R}^3}{\rm d}y\mathbf{1}_{\left\{|y-r_0| 
\le \rho_0 t - (\rho_0 + {\texttt v}_{\texttt{max}})s\right\}}f(y, \upsilon_1)\notag\\
&=\frac{C_3{\rm e}^{-\overline{\alpha} t}}{t^2}\int_{V}\D\upsilon_1\int_{{\texttt v}_{\texttt{min}}}^{{\texttt v}_{\texttt{max}}}\D\rho_0\int_{\mathbb{R}}\D s \int_{\mathbb{R}^3}\D y\mathbf{1}_{\left\{0 \le s \le \frac{\rho_0 t - |y-r_0|}{\rho_0 + {\texttt v}_{\texttt{max}}}\right\}}f(y, \upsilon_1)\notag\\
&=\frac{C_3{\rm e}^{-\overline{\alpha} t}}{t^2}\int_{V}\D\upsilon_1\int_{{\texttt v}_{\texttt{min}}}^{{\texttt v}_{\texttt{max}}}\D\rho_0\int_{\mathbb{R}^3}{\rm d}y \mathbf{1}_{\left\{|y-r_0| \le \rho_0 t\right\}}\left(\frac{\rho_0 t - |y-r_0|}{\rho_0 + {\texttt v}_{\texttt{max}}} \right)f(y, \upsilon_1).
\label{LB6}
\end{align}
We finally compute the integral with respect to $\rho_0 \in ({\texttt v}_{\texttt{min}}, {\texttt v}_{\texttt{max}})$. In order to do so, we first note that since $\rho_0 < {\texttt v}_\texttt{max}$, the integrand in \eqref{LB6} is bounded below by
\[
\frac{\rho_0 t - |y-r_0|}{2{\texttt v}_{\texttt{max}}}.
\]
Absorbing $1/2{\texttt v}_{\texttt{max}}$ into the constant $C_3$, applying Fubini and computing the ${\rho}_0$ integral yields
\begin{align}
\mathbf{E}^\dagger_{(r_0, \upsilon_0)}[f(R_t, \Upsilon_t)] &\ge \frac{C_3{\rm e}^{-\overline{\alpha} t}}{t^2}\int_{V}\D\upsilon_1\int_{\mathbb{R}^3}{\rm d}y \bigg[\frac{t}{2}\left({\texttt v}_{\texttt{max}}^2 - \left({\texttt v}_{\texttt{min}}\vee \frac{|y-r|}{t}\right)^2\right)\notag\\
 &\quad -|y-r|\left({\texttt v}_{\texttt{max}} - {\texttt v}_{\texttt{min}}\vee \frac{|y-r|}{t}\right)\bigg]\mathbf{1}_{\left\{|y-r_0| \le {\texttt v}_{\texttt{max}} t\right\}}f(y, \upsilon_1),
\label{LB7}  
\end{align}
as required. 
\end{proof}

We now turn to the proof of \ref{itm:A1} under the assumptions of (B1) and (B2). 

\begin{proof}[Proof of \ref{itm:A1}]
In this proof, we will follow a similar strategy to the one presented in~\cite[Section 4.2]{CV}. We therefore start by proving \ref{itm:A1} for initial configurations in $ D_{\varepsilon}\times V$. 

\smallskip

To this end, fix $(r, \upsilon) \in D_{\varepsilon} \times V$. From  Lemma~\ref{TB}, there exists an $i \in \{1, \dots, {\color{black}\mathfrak{n}}\}$ such that $r \in B(r_i, {\texttt v}_{\texttt{max}}\varepsilon/32) \cap D_{\varepsilon}$. Then, for each $t \in [\varepsilon/2, \varepsilon)$, Lemma~\ref{minlem} yields
\begin{align}
\notag\pdag_{(r, \upsilon)}(R_t \in {\rm d}z, \Upsilon_t \in {\rm d}w) &\ge\frac{C{\rm e}^{-\overline{\alpha} t}}{t^2}\bigg[\frac{t}{2}\left({\texttt v}_{{\texttt{max}}}^2 - \left({\texttt v}_{{\texttt{min}}}\vee \frac{|z-r|}{t}\right)^2\right)\notag\\
 &\qquad -|z-r|\left({\texttt v}_{{\texttt{max}}} - {\texttt v}_{{\texttt{min}}}\vee \frac{|z-r|}{t}\right)\bigg]\mathbf{1}_{\{z\in B(r, {\texttt v}_{\texttt{max}}t)\}}\,\D z \,\D w. \label{density}
\end{align}
Now, if $j\in \{1, \dots {\color{black}\mathfrak{n}}\}$ is such that $B(r_i, {\texttt v}_{\texttt{max}}\varepsilon/32) \cap B(r_j, {\texttt v}_{\texttt{max}}\varepsilon/32) \neq \emptyset$, the triangle inequality implies that $D_{\varepsilon} \cap (B(r_i, {\texttt v}_{\texttt{max}}\varepsilon/32) \cup B(r_j, {\texttt v}_{\texttt{max}}\varepsilon/32)) \subset B(r, {\texttt v}_{\texttt{max}}\varepsilon/8) \subset B(r, {\texttt v}_{\texttt{max}}t)$, with the latter inclusion following from the fact that $t \in [\varepsilon/2, \varepsilon)$.

\smallskip

Hence, for $z \in B(r_i, {\texttt v}_{\texttt{max}}\varepsilon/32) \cup B(r_j, {\texttt v}_{\texttt{max}}\varepsilon/32)$ and $t \in [\varepsilon/2, \varepsilon)$, the density on the right-hand side of~\eqref{density} is bounded below by a constant $C_{\varepsilon} > 0$, which is independent of $r, \upsilon, i$ and $j$. Hence,
\begin{equation}
\pdag_{(r, \upsilon)}(R_t \in {\rm d}z, \Upsilon_t \in {\rm d}w)\ge C_{\varepsilon}\mathbf{1}_{\{z\in D_{\varepsilon}\cap (B(r_i, \varepsilon/32) \cup B(r_j, \varepsilon/32))\}}\,{\rm d}z\,{\rm d}w, \qquad z\in D, w\in V.
\label{iteration}
\end{equation}

\smallskip

Now let $t \ge ({\color{black}\mathfrak{n}}+1)\varepsilon/2$. By writing $t = k \varepsilon/2 + t'$, for some $k \ge {\color{black}\mathfrak{n}}$ and $t' \in [\varepsilon/2, \varepsilon)$. {\color{black}We will demonstrate that a repeated application of \eqref{iteration} will lead to the inequality}
\begin{equation}
\pdag_{(r, \upsilon)}(R_t \in {\rm d}z, \Upsilon_t \in {\rm d}w) \ge C_{\varepsilon}c_{\varepsilon}^{k}\mathbf{1}_{\{z\in D_{\varepsilon}\}}{\rm d}z{\rm d}w, \qquad z\in D, w\in V,
\label{toprove}
\end{equation}
for $(r, \upsilon) \in D_{\varepsilon} \times V$, where $c_{\varepsilon} > 0$ {\color{black}is another unimportant constant which depends only on $\varepsilon$ and is defined in the following analysis.}

\smallskip

{\color{black}To this end, we start by noting that, since} $r \in D_{\varepsilon}$ and $\upsilon \in V$, there exists $i_0, i_1 \in \{1, \dots, {\color{black}\mathfrak{n}}\}$ such that $r \in B(r_{i_0}, {\texttt v}_{\texttt{max}}\varepsilon/32)$ and $B(r_{i_0}, {\texttt v}_{\texttt{max}}\varepsilon/32) \cap B(r_{i_1}, {\texttt v}_{\texttt{max}}\varepsilon/32) \cap D_{\varepsilon} \neq \emptyset$. Applying~\eqref{iteration} at time $t'$ {\color{black}(recall that we have identified $t = k \varepsilon/2 + t'$ for some $k\geq {\color{black}\mathfrak{n}}$)}  we obtain,
\begin{align}
\pdag_{(r, \upsilon)}(R_t &\in {\rm d}z, \Upsilon_t \in {\rm d}w) \notag\\
&= \pdag_{(r, \upsilon)}(R_{t' + k\varepsilon/2} \in \D z, \Upsilon_{t' + k\varepsilon/2} \in \D w)\notag \\
& \ge \edag_{(r, \upsilon)}\left[\mathbf{1}_{\{R_{t'} \in B(r_{i_1}, {\texttt v}_{\texttt{max}}\varepsilon/32) \cap D_{\varepsilon}, \Upsilon_{t'} \in V\}}\pdag_{(R_{t'}, \Upsilon_{t'})}(R_{k\varepsilon/2} \in \D z, \Upsilon_{k\varepsilon/2} \in \D w) \right]\notag\\
&= \int_{B(r_{i_1}, {\texttt v}_{\texttt{max}}\varepsilon/32)\cap D_{\varepsilon}}\int_V \pdag_{(r', \upsilon')}(R_{k\varepsilon/2} \in \D z, \Upsilon_{k\varepsilon/2} \in \D w)\pdag_{(r, \upsilon)}(R_{t'} \in \D r', \Upsilon_{t'} \in \D \upsilon')\notag\\
&\ge C_{\varepsilon}\int_{B(r_{i_1}, {\texttt v}_{\texttt{max}}\varepsilon/32)\cap D_{\varepsilon}}\int_V \pdag_{(r', \upsilon')}(R_{k\varepsilon/2} \in \D z, \Upsilon_{k\varepsilon/2} \in \D w)\notag \\
&\hspace{6cm} \times \mathbf{1}_{\{r'\in (B(r_{i_0}, {\texttt v}_{\texttt{max}}\varepsilon/32) \cup B(r_{i_1}, {\texttt v}_{\texttt{max}}\varepsilon/32))\cap D_{\varepsilon} \}}\D r' \D \upsilon' \notag \\
&=   C_{\varepsilon}\int_{B(r_{i_1}, {\texttt v}_{\texttt{max}}\varepsilon/32)\cap D_{\varepsilon}}\int_V \pdag_{(r', \upsilon')}(R_{k\varepsilon/2} \in \D z, \Upsilon_{k\varepsilon/2} \in \D w)\D r' \D \upsilon'.\label{step1}
\end{align}
We now turn our attention to $\pdag_{(r', \upsilon')}(R_{k\varepsilon/2} \in \D z, \Upsilon_{k\varepsilon/2} \in \D w)$, for $(r', \upsilon') \in (B(r_{i_1}, {\texttt v}_{\texttt{max}}\varepsilon/32) \cap D_{\varepsilon}) \times V$ and $k \ge {\color{black}\mathfrak{n}}$. Thanks to Lemma~\ref{TB}, for all $i_{k+1} \in \{1, \dots, {\color{black}\mathfrak{n}}\}$, there exist $i_2, \dots, i_k \in \{1, \dots, {\color{black}\mathfrak{n}}\}$ such that $B(r_{i_j}, \varepsilon/32) \cap B(r_{i_{j+1}}, \varepsilon/32) \neq \emptyset$ for every $j \in \{1, \dots, k\}$. {\color{black} Note, here we see the importance of choosing $k\geq \mathfrak{n}$, to ensure the validity of the previous statement.}

\smallskip

Applying~\eqref{iteration} and following the same steps that lead to~\eqref{step1}, we obtain
\begin{align}
\pdag_{(r', \upsilon')}&(R_{k\varepsilon/2} \in \D z, \Upsilon_{k\varepsilon/2} \in \D w) \notag \\
&\ge  C_{\varepsilon}\int_{B(r_{i_2}, \varepsilon/32)\cap D_{\varepsilon}}\int_V\pdag_{(r'', \upsilon'')}(R_{(k-1)\varepsilon/2} \in \D z, \Upsilon_{(k-1)\varepsilon/2} \in \D w)\D r'' \D \upsilon'' \label{step2}.
\end{align}
Iterating this step a further $k-2$ times, we obtain
\begin{align}
\pdag_{(r', \upsilon')}&(R_{k\varepsilon/2} \in \D z, \Upsilon_{k\varepsilon/2} \in \D w) \notag\\
&\ge  C_{\varepsilon}c_{\varepsilon}^{k-2}\int_{B(r_{i_k}, {\texttt v}_{\texttt{max}}\varepsilon/32)\cap D_{\varepsilon}}\int_V\pdag_{(r'', \upsilon'')}(R_{\varepsilon/2} \in \D z, \Upsilon_{\varepsilon/2} \in \D w)\D r'' \D \upsilon'', \label{step3}
\end{align}
where $c_{\varepsilon} = C_{\varepsilon}{\rm Vol}(V)\min_{i = 1, \dots, n}{\rm Vol}(B(r_i, {\texttt v}_{\texttt{max}}\varepsilon/32) \cap D_{\varepsilon})$. Using this inequality to bound the right-hand side of~\eqref{step1} yields
\begin{align}
\pdag_{(r, \upsilon)}(R_t &\in {\rm d}z, \Upsilon_t \in {\rm d}w) \notag\\
&\ge C_{\varepsilon}c_{\varepsilon}^{k-1}\int_{B(r_{i_{k}}, \varepsilon/32)\cap D_{\varepsilon}}\int_V\pdag_{(r', \upsilon')}(R_{\varepsilon/2} \in \D z, \Upsilon_{\varepsilon/2} \in \D w)\D r' \D \upsilon'. \label{step4}
\end{align}
We now apply~\eqref{iteration} a final time at time $\varepsilon/2$ to obtain
\begin{equation}
\pdag_{(r, \upsilon)}(R_t \in {\rm d}z, \Upsilon_t \in {\rm d}w) \ge C_{\varepsilon}c_{\varepsilon}^{k}\mathbf{1}_{\{z\in B(r_{i_{k+1}}, \varepsilon/2) \cap D_{\varepsilon}\}}\, \D z\,  \D w.
\label{step5}
\end{equation}
Since this inequality holds for every $i_{k+1} \in \{1, \dots, {\color{black}\mathfrak{n}}\}$, it also follows that
\begin{align}
\pdag_{(r, \upsilon)}(R_t \in {\rm d}z, \Upsilon_t \in {\rm d}w) &\ge C_{\varepsilon}c_{\varepsilon}^{k}\sup_{i_{k+1} \in \{1, \dots, {\color{black}\mathfrak{n}}\}}\mathbf{1}_{\{z\in B(r_{i_{k+1}}, \varepsilon/2) \cap D_{\varepsilon}\}}\, \D z \, \D w \notag \\
&\ge C_{\varepsilon}c_{\varepsilon}^{k}\,\mathbf{1}_{\{z\in D_{\varepsilon}\}} \, \D z \, \D w, \notag
\end{align}
where the final line follows from Lemma~\ref{TB} since $k+1 > {\color{black} \mathfrak{n}}$. This is the lower bound claimed in~\eqref{toprove}. 

\smallskip

{\color{black} Finally, noting that for any two events $A,B$, $\Pr(A|B) = \Pr(A \cap B)/\Pr(B)\geq \Pr(A \cap B)$, we have that for initial conditions $(r, \upsilon) \in D_{\varepsilon}\times V$, any $t_0 \ge (\mathfrak{n}+1)\varepsilon/2$ and $\nu$ equal to Lebesgue measure on $D_{\varepsilon}\times V$, there exists a constant $c_1\in(0,\infty)$ such that 
\[
\mathbf{P}_{(r, \upsilon)}((R_{t_0}, \Upsilon_{t_0}) \in \cdot \;| t_0 < {\texttt{k}}) \ge c_1 \nu(\cdot),
\] 
as required by  \ref{itm:A1}.
}
\bigskip

We now prove \ref{itm:A1} for initial conditions in $(D\backslash D_{\varepsilon}) \times V$. Once again, we recall that assumptions \ref{itm:B1} and \ref{itm:B2} are in force.

\smallskip

Choose $r\in D\backslash D_{\varepsilon}$, $\upsilon \in V$ and define the (deterministic) time
\[
\kappa^{D\backslash D_{\varepsilon}}_{r, \upsilon} \coloneqq \inf\{t > 0 : r + t\upsilon \not\in \partial D\backslash D_{\varepsilon}\},
\]
{\color{black}which is the time it would take a neutron released at $r$ with velocity $\upsilon$ to hit the boundary of $D\backslash D_\varepsilon$ if no scatter or fission took place. Note in particular that ${\color{black}\kappa^{D\backslash D_{\varepsilon}}_{r, \upsilon} }$ is not a random time but entirely deterministic.} 
We first consider the case $r + {\color{black}\kappa^{D\backslash D_{\varepsilon}}_{r, \upsilon} } \upsilon \in \partial D_{\varepsilon}$
\begin{equation}
\pdag_{(r, \upsilon)}(R_{{\color{black}\kappa^{D\backslash D_{\varepsilon}}_{r, \upsilon} }} \in \partial D_{\varepsilon}) \ge {\rm e}^{-\bar{\alpha}{\color{black}\kappa^{D\backslash D_{\varepsilon}}_{r, \upsilon} }} \ge {\rm e}^{-\bar{\alpha}{\rm diam}(D)/{\texttt v}_{\texttt{min}}}.
\end{equation}
Combining this with~\eqref{toprove} and the Markov property, for all $t \ge ({\color{black}\mathfrak{n}}+1)\varepsilon/2$
\begin{align}
\mathbf{P}_{(r, \upsilon)}&(R_{{\color{black}\kappa^{D\backslash D_{\varepsilon}}_{r, \upsilon} }+t} \in \D z, \Upsilon_{{\color{black}\kappa^{D\backslash D_{\varepsilon}}_{r, \upsilon} }+t} \in \D w | {\color{black}\kappa^{D\backslash D_{\varepsilon}}_{r, \upsilon} }+t < \texttt{k})\notag \\
& \ge \pdag_{(r, \upsilon)}(R_{{\color{black}\kappa^{D\backslash D_{\varepsilon}}_{r, \upsilon} }+t} \in \D z, \Upsilon_{{\color{black}\kappa^{D\backslash D_{\varepsilon}}_{r, \upsilon} }+t} \in \D w)\notag\\
&\ge {\rm e}^{-\bar{\alpha}{\rm diam}(D)/{\texttt v}_{\texttt{min}}}C_{\varepsilon}c_{\varepsilon}^{k}\mathbf{1}_{\{z\in D_{\varepsilon}\}}\, {\rm d}z\,{\rm d}w,
\end{align}
where $k \ge{\color{black}\mathfrak{n}}$ is such that $t = k\varepsilon/2 + t'$ for some $t' \in [\varepsilon/2, \varepsilon)$.

\smallskip

On the other hand, suppose $r+{\color{black}\kappa^{D\backslash D_{\varepsilon}}_{r, \upsilon} } \upsilon \in \partial D$. {\color{black}Then, recalling the assumptions \ref{itm:B1} and  \ref{itm:B2}} it follows that $\{J_1 < {\color{black}\kappa^{D\backslash D_{\varepsilon}}_{r, \upsilon} } \wedge (t_{\varepsilon} - s_{\varepsilon}), \Upsilon_{J_1} \in K_{r+\upsilon J_1}, J_2 > t_{\varepsilon}\} \subset \{R_{t_{\varepsilon}} \in D_{\varepsilon}, t_{\varepsilon} < \texttt{k}\}$. Heuristically speaking, this is because if the first jump occurs before time ${\color{black}\kappa^{D\backslash D_{\varepsilon}}_{r, \upsilon} } \wedge (t_{\varepsilon} - s_{\varepsilon})$, then the process hasn't hit the boundary and there are still (at least) $s_{\varepsilon}$ units of time left until $t_{\varepsilon}$. {\color{black}By then choosing the new velocity, $\Upsilon_{J_1}$, from $K_{r+{\upsilon}J_1}$, thanks to the assumption \ref{itm:B1} and the remarks around \eqref{L_r},} this implies that the process will remain in {\color{black}$D\backslash D_\varepsilon$} for $s_{\varepsilon}$ units of time, {\color{black}at some point in time after which, it will move into $D_{\varepsilon}$, providing the process doesn't jump again before entering $D_{\varepsilon}$.} Combining this with the usual bounds on $\alpha$, and recalling from \ref{itm:B2} that ${\rm Vol}(K_r)>\gamma>0$ for all $r\in D\backslash D_\varepsilon$ {\color{black}and $\upsilon\in V$}, we have
\begin{align}
\mathbf{P}_{(r, \upsilon)}(R_{t_{\varepsilon}} \in D_{\varepsilon}, t_{\varepsilon} < \texttt{k})& \ge \pdag_{(r, \upsilon)}(J_1 < {\color{black}\kappa^{D\backslash D_{\varepsilon}}_{r, \upsilon} } \wedge (t_{\varepsilon} - s_{\varepsilon}), \Upsilon_{J_1} \in K_{r+\upsilon J_1}, J_2 > t_{\varepsilon})\notag\\
&\ge {\color{black}\underline{\pi}\gamma}{\rm e}^{-\overline{\alpha}t_{\varepsilon}}\pdag_{(r, \upsilon)}(J_1 < {\color{black}\kappa^{D\backslash D_{\varepsilon}}_{r, \upsilon} } \wedge (t_{\varepsilon} - s_{\varepsilon})).
\end{align}
Along with~\eqref{toprove}, this implies that, for all {\color{black}$r\in D\backslash D_\varepsilon$, $\upsilon\in V$ and} $t \ge (\mathfrak{n}+1)\varepsilon/2$ such that $t + t_{\varepsilon} \ge {\color{black}\kappa^{D\backslash D_{\varepsilon}}_{r, \upsilon} }$
\begin{align}
\mathbf{P}_{(r, \upsilon)}(R_{t+t_{\varepsilon}} \in \D z, &\Upsilon_{t+ t_{\varepsilon}}\in \D w | t + t_{\varepsilon} < \texttt{k})\notag\\
&\ge \frac{\mathbf{P}_{(r, \upsilon)}(R_{t_{\varepsilon}} \in D_{\varepsilon}, t_{\varepsilon} < \texttt{k}, R_{t+t_{\varepsilon}} \in \D z, \Upsilon_{t+t_\varepsilon} \in \D w)}{\mathbf{P}_{(r, \upsilon)}(t+t_{\varepsilon} < \texttt{k})}\notag \\
& = \frac{\pdag_{(r, \upsilon)}(R_{t+t_\varepsilon} \in \D z ; \Upsilon_{t+t_\varepsilon} \in \D w | R_{t_\varepsilon} \in D_{\varepsilon}, t_\varepsilon < \texttt{k})\mathbf{P}_{(r, \upsilon)}(R_{t_\varepsilon} \in D_{\varepsilon}, t_\varepsilon < \texttt{k})}{\mathbf{P}_{(r, \upsilon)}(t+t_{\varepsilon} < \texttt{k})}\notag \\
&{\color{black}\ge 
\inf_{r\in D_\varepsilon, \upsilon \in V}\pdag_{(r, \upsilon)}(R_{t}\in \D z ; \Upsilon_{t} \in \D w )}\notag\\
&{\color{black}\hspace{3cm}\times\frac{\pdag_{(r, \upsilon)}(J_1 < {\color{black}\kappa^{D\backslash D_{\varepsilon}}_{r, \upsilon} } \wedge (t_{\varepsilon} - s_{\varepsilon}))}{\mathbf{P}_{(r, \upsilon)}(t+t_{\varepsilon} < \texttt{k})}
\underline\pi\gamma}{\rm e}^{-\overline{\alpha}t_{\varepsilon}}c_{\varepsilon}^k\mathbf{1}_{\{z\in D_{\varepsilon}\}}\,\D z\,\D w\notag\\
&\ge \frac{\pdag_{(r, \upsilon)}(J_1 < {\color{black}\kappa^{D\backslash D_{\varepsilon}}_{r, \upsilon} } \wedge (t_{\varepsilon} - s_{\varepsilon}))}{\mathbf{P}_{(r, \upsilon)}(t+t_{\varepsilon} < \texttt{k})}{\color{black}\underline\pi\gamma}{\rm e}^{-\overline{\alpha}t_{\varepsilon}}C_{\varepsilon}c_{\varepsilon}^k\mathbf{1}_{\{z\in D_{\varepsilon}\}}\,\D z\,\D w.
\label{laststep}
\end{align}
Now, since we are considering the case $r+{\color{black}\kappa^{D\backslash D_{\varepsilon}}_{r, \upsilon} } \upsilon \in \partial D$ and $t+t_{\varepsilon} \ge {\color{black}\kappa^{D\backslash D_{\varepsilon}}_{r, \upsilon} }$, it follows that $\{t+t_{\varepsilon} < \texttt{k}\} \subset \{J_1 < {\color{black}\kappa^{D\backslash D_{\varepsilon}}_{r, \upsilon} }\}$. Then, 
\begin{align}
\frac{\pdag_{(r, \upsilon)}(J_1 < {\color{black}\kappa^{D\backslash D_{\varepsilon}}_{r, \upsilon} } \wedge (t_{\varepsilon} - s_{\varepsilon}))}{\mathbf{P}_{(r, \upsilon)}(t+t_{\varepsilon} < \texttt{k})} &\ge \frac{\pdag_{(r, \upsilon)}(J_1 < {\color{black}\kappa^{D\backslash D_{\varepsilon}}_{r, \upsilon} } \wedge (t_{\varepsilon} - s_{\varepsilon}))}{\mathbf{P}_{(r, \upsilon)}(J_1 < {\color{black}\kappa^{D\backslash D_{\varepsilon}}_{r, \upsilon} })}\notag\\
&\ge \frac{1- {\rm e}^{-\underline{\alpha}({\color{black}\kappa^{D\backslash D_{\varepsilon}}_{r, \upsilon} } \wedge (t_{\varepsilon} - s_{\varepsilon}))}}{1-{\rm e}^{-\overline{\alpha}{\color{black}\kappa^{D\backslash D_{\varepsilon}}_{r, \upsilon} }}},
\end{align}
with the  bound on the right-hand side above being itself bounded below by a constant that does not depend on $(r, \upsilon)$. {\color{black} Substituting this back into~\eqref{laststep}, this proves \ref{itm:A1} with $\nu$ taken as Lebesgue measure on $D_\varepsilon\times V$ as before, $t_0$ can be sufficiently taken as $( \mathfrak{n}+1)\varepsilon/2 + {\rm diam}(D)/{\texttt v}_{\texttt{min}}$ and we may start with any initial configurations in $D\backslash D_{\varepsilon} \times V$.}
\end{proof}

In order to prove~\ref{itm:A2} we require the following lemma, the proof of which will be given after that of~\ref{itm:A2}.
\begin{lem}\label{jumps}
For all $r\in D$ and $\upsilon \in V$, recalling that $J_k$ denotes the $k^{th}$ jump time of the process $(R, \Upsilon)$, we have
\begin{equation}
\pdag_{(r, \upsilon)}(J_7 < \emph{\texttt{k}}, R_{J_7} \in {\rm d}z) \le C\mathbf{1}_{\{z\in D\}}\,{\rm d}z,
\label{jump7}
\end{equation}
for some constant $C > 0$, and
\begin{equation}
\pdag_{\nu}(J_1 <\emph{\texttt{k}}, R_{J_1} \in {\rm d}z) \ge c\mathbf{1}_{\{z\in D\}}\,{\rm d}z,
\label{jump1}
\end{equation}
for another constant $c>0$, {\color{black}where $\nu$, from the proof of (A1), is Lebesgue measure on $D_\varepsilon\times V$.}
\end{lem}
\begin{proof}[Proof of \ref{itm:A2}]
Again, we follow the proof given by the authors in~\cite{CV}. Let $t \ge {7{\rm diam}(D)}/{{\texttt v}_{\texttt{min}}}$ and note that on the event $\{\texttt{k} > t\}$, we have $J_7 \le {7{\rm diam}(D)}/{{\texttt v}_{\texttt{min}}}$ almost surely. This inequality along with the strong Markov property imply that,
\begin{align}
\mathbf{P}_{(r, \upsilon)}(t < \texttt{k}) &\le \edag_{(r, \upsilon)}\left[\mathbf{1}_{\{J_7<t\}}\mathbf{P}_{(R_{J_7}, \Upsilon_{J_7})}\left(t-s < \texttt{k} \right)_{s = J_7}\right]\notag\\
& \le\edag_{(r, \upsilon)}\left[\mathbf{P}_{(R_{J_7}, \Upsilon_{J_7})}\left( t-\frac{7{\rm diam}(D)}{{\texttt v}_{\texttt{min}}} < \texttt{k} \right) \right].\label{UB1}
\end{align}
Since $\pi$ is uniformly bounded above, 
conditional on $\{J_7 < \infty, R_{J_7} \in \D z\}$, the density of $\Upsilon_{J_7}$ is bounded above by {\color{black}$\overline{\pi}$} multiplied by Lebesgue measure on $V$. Combining this with~\eqref{jump7} and~\eqref{UB1}, we obtain
\begin{equation}
\mathbf{P}_{(r, \upsilon)}(t < \texttt{k})\le {\color{black}C'}\int_D\int_V\mathbf{P}_{(z, w)}\left(t-\frac{7{\rm diam}(D)}{{\texttt v}_{\texttt{min}}} < \texttt{k}\right){\rm d}w\,{\rm d}z,
\label{UB2}
\end{equation}
for some $C'\in(0,\infty)$
Similarly, {\color{black} for $t\ge {\rm diam}(D)/{\texttt v}_{\texttt{min}}$, equation~\eqref{jump1},
 the fact that  the inclusion $\{t < \texttt{k}\}\subset\{J_1 \le {\rm diam}(D)/{\texttt v}_{\texttt{min}}\}$,} 
the strong Markov property and the fact that $\pi$ is uniformly bounded below entail that,
\begin{align*}
\mathbf{P}_{\nu}(t < \texttt{k})&
= \edag_{\nu}\left[\mathbf{1}_{\{J_1 \le \texttt{k}\}}\mathbf{P}_{(R_{J_1}, \Upsilon_{J_1})}\left(t-s < \texttt{k} \right)_{s = J_1}\right]\notag\\
&\geq  \edag_{\nu}\left[\mathbf{1}_{\{J_1 \le \texttt{k}\}}\mathbf{P}_{(R_{J_1}, \Upsilon_{J_1})}\left(t< \texttt{k} \right)\right]\notag\\
& \ge {\color{black}c'}\int_D\int_V\mathbf{P}_{(z, w)}(t < \texttt{k})\,{\rm d}w\,{\rm d}z,
\end{align*}
for some $c'\in(0,\infty)$, where {\color{black} $\nu$ is Lebesgue measure on $D_\varepsilon\times V$}. Putting \eqref{UB1} and \eqref{UB2} together, for all $t \ge 8{\rm diam}(D)/{\texttt v}_{\texttt{min}}$, we have
\begin{equation}
\mathbf{P}_{(r, \upsilon)}(t < \texttt{k}) \le \frac{C'}{c'}\mathbf{P}_{\nu}\left(t- \frac{7{\rm diam}(D)}{{\texttt v}_{\texttt{min}}} < \texttt{k}\right).
\label{A2UB1}
\end{equation}
Now, recalling $t_0$ and $\nu$ from the proof of \ref{itm:A1}, it follows from \ref{itm:A1} that 
\begin{equation}
\label{reminder}
\pdag_{\nu}((R_{t_0}, \Upsilon_{t_0}) \in \cdot) \ge c_1\mathbf{P}_{\nu}(t_0 < \texttt{k})\nu(\cdot).
\end{equation}  The event $\{t < \texttt{k}\}$ occurs if the particle has either been killed on the boundary of $D$ or if it has been absorbed by fissile material, which occurs at rate $\bar\beta - \beta$. Since $t_0$ and $\nu$ are fixed, and {\color{black}$\overline\beta-\beta\leq \overline\beta+1<\infty$} by assumption, $\mathbf{P}_{\nu}(t_0 < \texttt{k}) \ge K$ for some constant $K > 0$. Thus, keeping $t \ge 8{\rm diam}(D)/{\texttt v}_{\texttt{min}}$, using \eqref{reminder}
\begin{align}
\mathbf{P}_{\nu}\left(t-\frac{7{\rm diam}(D)}{{\texttt v}_{\texttt{min}}} + t_0 < \texttt{k} \right) &= \mathbf{E}_{\nu}\left[\mathbf{1}_{\{t_0<\texttt{k}\}}\pdag_{(R_{t_0}, \Upsilon_{t_0})}\left(t-\frac{7{\rm diam}(D)}{{\texttt v}_{\texttt{min}}} < \texttt{k} \right)\right]\notag\\
&\ge \tilde{c}_1 \mathbf{P}_{\nu}\left(t-\frac{7{\rm diam}(D)}{{\texttt v}_{\texttt{min}}}< \texttt{k} \right), \label{A2LB1}
\end{align}
where $\tilde{c}_1 = Kc_1$.

\smallskip

Now define $N=\lceil 7 \text{diam}(D)/({\texttt v}_{\texttt{min}}t_0) \rceil$. Then, for any $t > 0$, $t - 7{\rm diam}(D)/{\texttt v}_{\texttt{min}}+Nt_0 \ge t$ so that, trivially,
\begin{equation}
\mathbf{P}_{\nu}(t < \texttt{k}) \ge \mathbf{P}_{\nu}\left(t-\frac{7{\rm diam}(D)}{{\texttt v}_{\texttt{min}}} + Nt_0 < \texttt{k}\right).
\label{A2LB2}
\end{equation}
Applying~\eqref{A2LB1} $N$ times implies that
\begin{equation}
\mathbf{P}_{\nu}(t < \texttt{k}) \ge \tilde{c}_1^N \mathbf{P}_{\nu}\left(t-\frac{4{\rm diam}(D)}{{\texttt v}_{\texttt{min}}} < \texttt{k}\right).
\label{A2LB3}
\end{equation}
Combining this with~\eqref{A2UB1} completes the proof of \ref{itm:A2}.
\end{proof}

\begin{proof}[Proof of Lemma \ref{jumps}]
Let us first prove~\eqref{jump7}. Again, following the proof given in~\cite{CV}, we couple the neutron transport random walk in $D$ with one on the whole of $\mathbb{R}^3$. Denote by $(\hat{R}_t, \hat{\Upsilon}_t)$ the neutron random walk in $\hat{D} = \mathbb{R}^3$, coupled with $(R, \Upsilon)$ such that $\hat{R}_t = R_t$ and $\hat{\Upsilon}_t = \Upsilon_t$ for all $t < \texttt{k}$ and $(R_0 , \Upsilon_0) = (\hat{R}_0 , \hat{\Upsilon}_0) = (r,\upsilon)$, for $r\in D$, $\upsilon\in V$. Denote by $\hat{J}_1 < \hat{J}_2 < \dots$ the jump times of $\hat{\Upsilon}_t$. Then for each $k \ge 1$ such that $J_k < \texttt{k}$, we have $\hat{J}_k = J_k$. Due to the inequality
\begin{equation}
\edag_{(r, \upsilon)}[f(R_{J_7}) ; J_7 < \texttt{k}] \le \mathbf{E}_{(r, \upsilon)}[f(\hat{R}_{\hat{J}_7})], \qquad r\in D, \upsilon \in V,
\label{ineq1}
\end{equation}
we will consider the distribution of $\hat{R}_{\hat{J}_i}$ for $i \ge 2$. We first look at the case when $i=2$. For $(r, \upsilon) \in D\times V$ and non-negative, bounded, measurable functions $f$,
\begin{align}
\mathbf{E}_{(r, \upsilon)}[f(\hat{R}_{\hat{J}_2})] &= \mathbf{E}_{(r, \upsilon)}[f(r + \upsilon \hat{J}_1 + \hat{\Upsilon}_{\hat{J}_1}(\hat{J}_2 - \hat{J}_1)]\notag \\
& \le \bar{\alpha}^2\bar{\pi}\int_0^\infty\D j_1 \int_V\D \upsilon_1 \int_0^\infty \D j_2 {\rm e}^{- \underline{\alpha}(j_1 + j_2)}f(r + \upsilon j_1 + \upsilon_1 j_2)\label{j1}
\end{align}
For $j_1$ fixed, we consider the integrals over $\upsilon_1$ and $j_2$ in~\eqref{j1}. Making the change of variables $\upsilon_1 \mapsto (\rho, \varphi, \theta)$, we have
\begin{align}
\int_V&\D \upsilon_1 \int_0^\infty \D j_2 {\rm e}^{- \underline{\alpha}j_2}f(r + \upsilon j_1 + \upsilon_1 j_2) \notag\\
&\le \int_{{\texttt v}_{\texttt{min}}}^1\D \rho \int_0^{2\pi} \D\theta\int_0^\pi \D \varphi \int_0^\infty \D j_2 {\rm e}^{- \underline{\alpha}j_2}f\left(r + \upsilon j_1 + \widetilde{\Theta}(\rho j_2, \theta, \varphi)\right)\rho^2\sin\varphi, \label{j2}
\end{align}
where $\widetilde{\Theta}$ was defined in~\eqref{cov2}. Now making the substitution $u = \rho j_2$ in~\eqref{j2},
\begin{align}
\int_V&\D \upsilon_1 \int_0^\infty \D j_2 {\rm e}^{- \underline{\alpha}j_2}f(r + \upsilon j_1 + \upsilon_1 j_2) \notag\\
&\le \int_{{\texttt v}_{\texttt{min}}}^{{\texttt v}_{\texttt{max}}}\D \rho \int_0^{2\pi} \D\theta\int_0^\pi \D \varphi \int_0^\infty \D u {\rm e}^{- \underline{\alpha}u/\rho}f\left(r + \upsilon j_1 + \widetilde{\Theta}(u, \theta, \varphi)\right)\rho\sin\varphi \notag\\
&\le C\int_0^{2\pi} \D\theta\int_0^\pi \D \varphi \int_0^\infty \D u {\rm e}^{- \underline{\alpha}u/{\texttt v}_{\texttt{max}}}f\left(r + \upsilon j_1 + \widetilde{\Theta}(u, \theta, \varphi)\right)\sin\varphi, \label{j3}
\end{align}
where $C = {\texttt v}_{\texttt{max}}({\texttt v}_{\texttt{max}} - {\texttt v}_{\texttt{min}})$. Making a final change of variables $(u, \theta, \varphi) \mapsto x \in \mathbb{R}^3$, we have
\begin{align}
\int_V\D \upsilon_1 \int_0^\infty \D j_2 {\rm e}^{- \underline{\alpha}j_2}f(r + \upsilon j_1 + \upsilon_1 j_2) \le C\int_{\mathbb{R}^3}\D x \, f(r + \upsilon j_1 + x)\frac{{\rm e}^{-\underline{\alpha}|x|/{\texttt v}_{\texttt{max}}}}{|x|^2}.
\label{j4}
\end{align}
Substituting this back into~\eqref{j1} yields
\begin{align}
\mathbf{E}_{(r, \upsilon)}[f(\hat{R}_{\hat{J}_2})] &\le \bar{\alpha}K\int_0^\infty\D j_1{\rm e}^{-\underline{\alpha}j_1}\int_{\mathbb{R}^3}\D x f(r + \upsilon j_1 + x)\frac{{\rm e}^{-\underline{\alpha}|x|/{\texttt v}_{\texttt{max}}}}{|x|^2},
\end{align}
where $K = \bar\alpha\bar{\pi} C$. Iterating this process over the next five jumps of the process gives
\begin{align}
\mathbf{E}_{(r, \upsilon)}[f(\hat{R}_{\hat{J}_7})] &\le \bar\alpha K^6\int_0^\infty\D j_1{\rm e}^{-\underline{\alpha}j_1}\int_{\mathbb{R}^3}\D x_1 \dots \int_{\mathbb{R}^3} \D x_6  f(r + \upsilon j_1 + x_1 + \dots + x_6)g(x_1)\dots g(x_6)\notag\\
\end{align}
where $g(x) = {{\rm e}^{-\underline{\alpha}|x|/{\texttt v}_{\texttt{max}}}}/ {|x|^2}$, $x\in \mathbb{R}^3$. Now, $g \in L^p(\mathbb{R}^3)$ for each $p < 3/2$ so that, in particular, $g \in L^{6/5}(\mathbb{R}^3)$. Hence, repeatedly applying Young's inequality implies that the six-fold convolution {\color{black}$ \ast^6 g  \in L^\infty(\mathbb{R}^3)$. (The reader will note that this is the fundamental reason we have focused our calculations around the 7th jump time $J_7$, rather than it being an arbitrary choice.)} Making the substitution $x = x_1 + \dots + x_6$,
\begin{align}
\edag_{(r, \upsilon)}[f(\hat{R}_{\hat{J}_7})] \le &\bar\alpha K^6\Vert \ast^6 g \Vert_\infty \int_0^\infty\D j_1{\rm e}^{-\underline{\alpha}j_1}\int_{\mathbb{R}^3}\D x_1 \dots \int_{\mathbb{R}^3} \D x_6  f(r + \upsilon j_1 + x).
\end{align}
Finally, setting $z = r + \upsilon j_1 + x$ yields
\begin{equation}
\edag_{(r, \upsilon)}[f(R_{J_7}) ; J_7 < \texttt{k}] \le \mathbf{E}_{(r, \upsilon)}[f(\hat{R}_{\hat{J}_7})] \le C' \int_{\mathbb{R}^3}f(z) \D z,
\end{equation}
where $C' = \bar\alpha K^6 \Vert g\ast \dots \ast g\Vert_{\infty}$, which completes the proof of~\eqref{jump7}.

\smallskip

We now prove~\eqref{jump1}. For $r, r' \in \mathbb{R}^3$, let $[r, r']$ denote the line segment between $r$ and $r'$. For all $f\in \mathcal{B}(\mathbb{R}^3)$, recalling the definition of $\nu$ from the proof of \ref{itm:A1} and using the usual bounds on $\alpha$,
\begin{align}
\mathbf{E}_{\nu}[f(R_{J_1}); J_1 < \texttt{k}] &\ge \int_{D_{\varepsilon}}\frac{{\rm d}r}{{\rm Vol}(D_{\varepsilon})}\int_V\frac{{\rm d}\upsilon}{{\rm Vol}(V)}\int_0^{\infty}{\rm d}s \,\mathbf{1}_{\{[r, r+s\upsilon]\subset D\}} \, \underline{\alpha}{\rm e}^{-\overline{\alpha} s}f(r+s\upsilon)\label{b1},
\end{align}
where $\textstyle{ {\rm Vol}(D_\varepsilon)  = \int_{D_\varepsilon}\d r}$ and $\textstyle{ {\rm Vol}(V)  = \int_{V}\d \upsilon}$.
Following a similar method to those employed in the proof of Lemma~\ref{minlem} and~\eqref{jump7} and changing first to polar coordinates via $\upsilon \mapsto (\rho, \theta, \varphi)$, followed by the substitution $u = s\rho$, and finally changing back to Cartesian coordinates via $(u, \theta, \varphi) \mapsto x$, the right-hand side of~\eqref{b1} is bounded below by
\begin{align}
C\int_{D_{\varepsilon}}{\rm d}r \int_{\mathbb{R}^3} \D x \,\mathbf{1}_{\{[r, r+x]\subset D\}} \, \frac{\underline{\alpha}{\rm e}^{-\overline{\alpha} s/{\texttt v}_{\texttt{min}}}}{|x|^2}f(r+x),
\label{b2}
\end{align}
where $C > 0$ is a constant. Making a final substitution of $x = z-r$, yields
\begin{align}
\mathbf{E}_{\nu}[f(R_{J_1}); J_1 < \texttt{k}] &\ge C\int_D{\rm d}z\mathbf{1}_{[r, z] \subset D}\frac{\underline{\alpha}{\rm e}^{-\overline{\alpha} |z-r|/{\texttt v}_{\texttt{min}}}}{|z-r|^2}f(z)\notag \\
&\ge C \frac{{\texttt v}_{\texttt{min}}^2\underline{\alpha}{\rm e}^{-\overline{\alpha} {\rm diam}(D)/\upsilon^2_{\texttt{min}}}}{({\rm diam}(D))^2} \int_D{\rm d}z\mathbf{1}_{\{[r, z] \subset D\}}f(z).
\label{b3}
\end{align}

\smallskip

For all $z \in D\backslash D_{\varepsilon}$, \ref{itm:B1} and the discussion thereafter now imply that 
\begin{equation}
\int_{D_{\varepsilon}}\mathbf{1}_{\{[r, z]\subset D\}}{\rm d}r \ge {\rm Vol}(L_z) \ge \frac{{\color{black}\gamma}}{2}(t_{\varepsilon}^2 - s_{\varepsilon}^2),
\label{outer}
\end{equation}
where $s_{\varepsilon}$ and $t_{\varepsilon}$ are defined in~\ref{itm:B2}, and $L_z$ is defined in \eqref{L_r}. On the other hand, for all $z \in D_{\varepsilon}$,
\begin{equation}
\int_{D_{\varepsilon}}\mathbf{1}_{\{[r, z] \subset D\}}{\rm d}r \ge {\rm Vol}(D_{\varepsilon} \cap B(r, \varepsilon)).
\label{inner}
\end{equation}
Since the map $z\mapsto {\rm Vol}(D_{\varepsilon}\cap B(z, \varepsilon))$ is continuous and positive on the compact set $\bar{D_{\varepsilon}}$, the latter equation is uniformly bounded below by a strictly positive constant. It then follows that for every $z \in D$, the integral
$\textstyle{
\int_{D_{\varepsilon}}\D r\mathbf{1}_{\{[r,z] \subset D\}}
}$
is bounded below by a positive constant. Using this to bound the right-hand side of~\eqref{b3} yields the result.
\end{proof}

\smallskip

We thus have proved that the conclusions of
Theorem~\ref{7CVtheoremBis} are valid under our assumptions. In order
to conclude that Theorem~\ref{CVtheorem} holds true, it remains to
prove that $\varphi$ is uniformly bounded away from 0 on each
compactly embedded subset of $D\times V$ and the existence of a
positive bounded density for the left eigenmeasure $\eta$.

\begin{lem}
\label{QSDdensity}
The right eigenfunction $\varphi$ is uniformly bounded away from $0$
on each compactly embedded subset of $D\times V$ and the probability
measure $\eta$ admits a positive density with respect to the Lebesgue
measure on $D\times V$, which corresponds to the quantity
$\tilde\varphi$ and which is uniformly bounded from above and
a.e. uniformly bounded from below on each compactly embedded subset of
$D\times V$.
\end{lem}

\begin{proof}
  For all $\varepsilon>0$, we deduce from the eigenfunction property
  of $\varphi$ (cf. Theorem~\ref{7CVtheoremBis}) and
  from~\eqref{toprove} that there exist a time $t_\varepsilon>0$ and a
  constant $\tilde{C}_\varepsilon>0$ such that
  \[
    \varphi(r,\upsilon)={\rm e}^{-\lambda_c t_\varepsilon}
    \texttt{P}_{t_\varepsilon}[\varphi](r,\upsilon) \geq {\rm e}^{-\lambda t_\varepsilon}
    \tilde{C}_\varepsilon \int_{D_\varepsilon \times V} \varphi(z,w)
    \d z \d w>0,
  \]
  for all $(r,\upsilon)\in D_\varepsilon\times V$.  It follows that
  $\varphi$ is uniformly bounded away from 0 on each compactly
  embedded domain of $D\times V$.

\smallskip

  Using the same notations as in the proof of Lemma~\ref{jumps}, we
  consider the neutron transport random walk
  $(\hat{R}_t, \hat{\Upsilon}_t)$ in $\hat{D} = \mathbb{R}^3$, coupled
  with $(R, \Upsilon)$ such that $\hat{R}_t = R_t$ and
  $\hat{\Upsilon}_t = \Upsilon_t$ for all $t < \texttt{k}$. We also
  denote by $\hat{J}_1 < \hat{J}_2 < \dots$ the jump times of
  $(\hat{\Upsilon}_t)_{t\geq 0}$. Let $T\geq 0$ be a random time
  independent of $(\hat R,\hat \Upsilon)$ with uniform law on
  $[\underline T,\bar T]$, where $\underline T<\bar T$ are fixed and
  $\underline T \ge {7{\rm diam}(D)}/{{\texttt v}_{\texttt{min}}}$.
  We first prove that the law of $(\hat{R}_T, \hat{\Upsilon}_T)$ after
  the $7^{\text{th}}$ jump admits a uniformly bounded density with
  respect to the Lebesgue measure. We conclude by using the coupling
  with $(R,\Upsilon)$ and the quasi-stationary property of $\eta$ in \eqref{eta}.

  \smallskip
  
 For all $k\geq 7$ and for any positive, bounded and measurable
  function $f$ vanishing outside of $D\times V$, we have
  \begin{align*}
    \mathbf{E}&[f(\hat R_T, \hat \Upsilon_T)\mathbf{1}_{\{\hat J_k \le T < \hat J_{k+1}\}}\mid \hat R_0,\hat \Upsilon_0, T] \\
                                          & = \mathbf{E}[f(\hat R_0+\hat J_1 \hat \Upsilon_0+\cdots + \hat J_k \hat \Upsilon_{k-1} + (T-\hat J_1-\cdots-\hat J_k)\hat \Upsilon_k,\hat\Upsilon_k)\mathbf{1}_{\{\hat J_k \le T < \hat J_{k+1}\}}\mid  \hat R_0,\hat \Upsilon_0, T]\\
                                  &  = \int_0^T \D s_1\, \alpha(\hat R_0+\upsilon_0s_1,\upsilon_0){\rm e}^{-\int_0^{s_1}\alpha(\hat R_0+\upsilon_0 u, \upsilon_0)\D u}
    \\
                                  &\quad\quad \times \int_V\D \upsilon_1 \pi(r_0+\upsilon_0 s_1, \upsilon_0, \upsilon_1)  \times \int_0^{T-s_1} \D s_2 \, \alpha(\hat R_0+\upsilon_0 s_1+\upsilon_1 s_2, \upsilon_1){\rm e}^{-\int_0^{s_2}\alpha(\hat R_0+\upsilon_0 s_1+\upsilon_1 u, \upsilon_1)\D u}\\
                                  &\quad\quad\times \cdots\\
                                  &\quad\quad\times \int_V \D\upsilon_{k-1}  \pi(\hat R_0+\upsilon_0 s_1+\cdots+\upsilon_{k-2} s_{k-1}, \upsilon_{k-2}, \upsilon_{k-1})\\
                                  &\quad\quad\times \int_0^{T-s_1-\cdots-s_{k-1}} \D s_k \, \alpha(\hat R_0+\upsilon_0 s_1+\cdots+\upsilon_{k-1} s_{k},\upsilon_{k-1})\\
                                  &\hspace{5cm} {\rm e}^{-\int_0^{s_k} \alpha(\hat R_0+\upsilon_0 s_1+\cdots+\upsilon_{k-2} s_{k-1} +\upsilon_{k-1} u, \upsilon_{k-1})\D u  }\\
                                  &\quad\quad\times \int_V \D \upsilon_k \pi(\hat R_0+\upsilon_0 s_1+\cdots+\upsilon_{k-1} s_{k}, \upsilon_{k-1}, \upsilon_{k})\\
                                  &\quad\quad\times {\rm e}^{-\int_0^{T-s_1-\cdots-s_k} \alpha(\hat R_0+\upsilon_0 s_1+\cdots+\upsilon_{k-1} s_{k}+\upsilon_k u,\upsilon_k)\D u}\\
                                  &\hspace{5cm}\times f(\hat R_0+\upsilon_0 s_1+\cdots+\upsilon_{k-1} s_{k}+\upsilon_k (t-s_1-\cdots-s_k), \upsilon_k).\\
  \end{align*}
  Henceforth
  \begin{align*}
  &  \mathbf{E}[f(\hat R_T, \hat \Upsilon_T)\mathbf{1}_{\{\hat J_k \le T < \hat J_{k+1}\}}\mid  \hat R_0,\hat \Upsilon_0, T] \\
    &\leq \bar\alpha^k \bar\pi^k {\rm e}^{-T \underline{\alpha}} \int_0^T \D s_1\int_V\D \upsilon_1  \cdots \int_0^{T-s_1-\cdots-s_{k-1}} \D s_k \int_V \D \upsilon_k\\
  & \hspace{2cm} \times f(\hat R_0+\upsilon_0 s_1+\cdots+\upsilon_{k-1} s_{k}+\upsilon_k (T-s_1-\cdots-s_k), \upsilon_k).
  \end{align*}
 Taking the expectation with respect to $T$,  we obtain
  \begin{align*}
   & \mathbf{E}[f(\hat R_T, \hat \Upsilon_T)\mathbf{1}_{\{\hat J_k \le T < \hat J_{k+1}\}}\mid  \hat R_0,\hat \Upsilon_0] \\
   &\leq \frac{\bar\alpha^k \bar\pi^k}{\bar T} \int_{0}^{\bar T} \D t \int_0^t \D s_1\int_V\D \upsilon_1  \cdots \int_0^{t-s_1-\cdots-s_{k-1}} \D s_k \int_V \D \upsilon_k\\
   &\hspace{2cm} \times f(\hat R_0+\upsilon_0 s_1+\cdots+\upsilon_{k-1} s_{k}+\upsilon_k (t-s_1-\cdots-s_k), \upsilon_k).
  \end{align*}
  Using the change of variable $(u_1,\ldots,u_k,u_{k+1})=(s_1,\ldots,s_k,t-s_1-\cdots-s_k)$ yields
  \begin{align*}
   & \mathbf{E}[f(\hat R_T, \hat \Upsilon_T)\mathbf{1}_{\{\hat J_k \le T < \hat J_{k+1}\}}\mid  \hat R_0,\hat \Upsilon_0] \\
   &\leq \frac{\bar\alpha^k \bar\pi^k}{\bar T} \int_{[0,\bar T]^{k+1}} \D u \, \mathbf{1}_{0\leq u_1+\cdots+u_{k+1}\leq \bar T} \int_{V^k}\D \upsilon \\
  &\hspace{2cm}      \times f(\hat R_0+\upsilon_0 u_1+\cdots+\upsilon_{k-1} u_{k}+\upsilon_k u_{k+1}, \upsilon_k).
  \end{align*}
  The same approach as in Lemma~\ref{jumps} shows that there exists a constant $C>0$ (which does not depend on $\hat R_0$ nor on $\hat \Upsilon_0$) such that, for all measurable function $g:\mathbb{R}^3\rightarrow [0,\infty)$,
  \begin{align*}
    \int_{[0,\bar T]^7} \D u \int_{V^6} \D \upsilon \, g(\hat R_0+\hat\Upsilon_0 u_1+\cdots+\upsilon_6 u_7) \leq C \int_{\R^d} \D x g(x).  
  \end{align*}
  Hence,
  \begin{align*}
  \mathbf{E}[f(\hat R_T, \hat \Upsilon_T) & \mathbf{1}_{\{\hat J_k \le
      T < \hat J_{k+1}\}}\mid   \hat R_0,\hat \Upsilon_0]\\
  & \leq \frac{C \bar\alpha^k \bar\pi^k}{\bar T}  \int_{[0,\bar T]^{k+1-7}} \D u \, \mathbf{1}_{0\leq u_8+\cdots+u_{k+1}\leq \bar T} \int_{V^{k-6}}\D \upsilon  \\
                              &\hspace{4cm} \times \int_{\R^3} \D x f(x+\upsilon_7\,u_8+\cdots+\upsilon_k u_{k+1}, \upsilon_k)\\
                              &= \frac{C \bar\alpha^k \bar\pi^k}{\bar T}  \int_{[0,\bar T]^{k+1-7}} \D u \, \mathbf{1}_{0\leq u_8+\cdots+u_{k+1}\leq \bar T} \int_{V^{k-6}}\D \upsilon \int_{\R^3} \D y f(y,\upsilon_k)\\
                              &= C \bar\alpha^k \bar\pi^k {\rm Vol}(V)^{k-8} \frac{\bar T^{k+1-8}}{(k+1-7)!} \int_D \D y \int_V \D \upsilon_k f(y,\upsilon_k)
  \end{align*}
  where we used the change of variable
  $y=x+\upsilon_7\,u_8+\cdots+\upsilon_k u_{k+1}$ and the fact that $f$ vanishes outside $D\times V$. Summing over $k\geq 7$, we deduce that there exists a constant $C'>0$ (which only depends on $C,\bar\alpha,\bar\pi$ and $\bar T$) such that
  \begin{equation*}
    \mathbf{E}[f(\hat R_T, \hat \Upsilon_T)\mathbf{1}_{\{\hat J_7 \le
    T \}}\mid   \hat R_0,\hat \Upsilon_0]
                \leq C' \int_D \D y \int_V \D \upsilon \, f(y,\upsilon).  
  \end{equation*}

  Similarly as in the proof of~\ref{itm:A2}, we chose
  $\underline T \ge {7{\rm diam}(D)}/{{\texttt v}_{\texttt{min}}}$, so that,
  on the event $\{\texttt{k} > T\}$, we have
  $J_7 \le {7{\rm diam}(D)}/{{\texttt v}_{\texttt{min}}}\leq T$ almost
  surely. Hence, we obtain that, for any $(r_0, \upsilon_0)\in D\times V$,
  \begin{align*}
    \edag_{(r_0, \upsilon_0)}[f(R_T,\Upsilon_T) ; T< \texttt{k}]&=\edag_{(r_0, \upsilon_0)}[f(R_T,\Upsilon_T) ; T< \texttt{k},J_7\leq T] \\
                                                                &\leq \mathbf{E}_{(r_0,\upsilon_0)}[f(\hat R_T,\hat\Upsilon_T);\hat J_7\leq  T]\\
                                                                &\leq C'\int_D \D y \int_V \D \upsilon \, f(y,\upsilon).
  \end{align*}
  Integrating with respect to $\eta$ and using the quasi-stationary property \eqref{eta} and Fubini's Theorem (recall
  that $T$ and the process $(R,\Upsilon)$ are independent), we obtain
  \begin{align}
    \frac{1}{\bar T - \underline T}\int_{\underline T}^{\bar T}\D t\, e^{\lambda_c t} \eta[f]
    &=\frac{1}{\bar T - \underline T}\int_{\underline T}^{\bar T} \D t\,\edag_{\eta}[f(R_t,\Upsilon_t) ; t< \texttt{k}]\notag\\
    &= \edag_{\eta}[f(R_T,\Upsilon_T) ; T< \texttt{k}]\notag\\
    &\leq C'\int_D \D y \int_V \D \upsilon \, f(y,\upsilon).
    \label{tildephibounded}
  \end{align}
 Since $f$ was chosen arbitrarily,  this proves that $\eta$ admits a uniformly bounded density (from above) with respect to the Lebesgue measure on $D\times V$.
\smallskip


  Finally, using the quasi-stationarity of $\eta$ \eqref{eta} and
integrating inequality~\eqref{iteration} with respect to $\eta$
implies that (here the time $t$ and the constants
$k,C_\varepsilon, c_{\varepsilon}$ depend on $\varepsilon$ as in
inequality \eqref{toprove}), for all bounded measurable functions $f$ on
$D\times E$,
\begin{align*}
  e^{\lambda_c t} \int_{D\times V} f(x)\eta( \D x)&= \edag_{\eta}[f(R_t,\Upsilon_t); t< \texttt{k}]\\
                                                                                 &\geq \eta(D_{\varepsilon}\times V)\,C_{\varepsilon} c_{\varepsilon}^k \int_{D_\varepsilon\times V}f(z,w)\,\D z \,\D w .
\end{align*}
This implies that $\tilde\varphi$ is a.e. lower bounded by
$e^{-\lambda_c t} \eta(D_{\varepsilon}\times
V)\,C_{\varepsilon} c_{\varepsilon}^k$ on $D_{\varepsilon}\times
V$. Since this inequality can be proved for any $\varepsilon>0$ small
enough, one deduces that, on any subset $D_\varepsilon\times V$ with
$\varepsilon>0$ and hence on any compactly embedded subset of $D\times V$,
$\tilde \varphi$ is a.e. uniformly bounded away from zero.
\end{proof}

\section{Proof of Theorem \ref{pathspine}}\label{spineproof}
There are three  main steps to the proof. The first is to characterise the law of transitions of  the Markov process $(X, {\mathbb{P}}^\varphi)$, defined in the change of measure \eqref{mgCOM}; note that the latter ensures the Markov property is preserved. The second step is to show that they agree  with those of $(X^\varphi, \tilde{\mathbb{P}}^\varphi)$. The third step is to show that  $(X^\varphi, \tilde{\mathbb{P}}^\varphi)$ is Markovian. Together these three imply the statement of the theorem.

\bigskip

{\it Step 1}. Next we look at the multiplicative semigroup which characterises uniquely the transitions of $(X^\varphi, \mathbb{P}^\varphi)$ (cf. \cite{INW1, INW2, INW3})
\begin{equation}
\label{uphi}
u^\varphi_t[g](r,\upsilon): 
= \mathbb{E}_{\delta_{(r,\upsilon)}}\left[W_t \prod_{i = 1}^{N_t}g (R_t^i, \Upsilon^i_t)\right]
=\mathbb{E}_{\delta_{(r,\upsilon)}}\left[
{\rm e}^{-\lambda_* t}\frac{\langle \varphi, X_t\rangle}{\varphi(r,\upsilon)} {\rm e}^{\langle \log g, X_t\rangle }\right]
\end{equation}
for $t\geq 0$ and $g\in L^+_\infty(D\times V)$ which is uniformly bounded by unity. Note, we keep to the convention that an empty product is understood as 1, however we also define the empty inner product as zero (corresponding to all functions scoring zero when particles arrive at the cemetery state $\{\dagger\}$). As such, if we are to extend the domain of test functions in the product to include the cemetery state $\{\dagger\}$, we need to insist on the default value $g(\{\dagger\}) = 1$;  see \cite{INW1, INW2, INW3}.

\smallskip
We start in the usual way by splitting the expectation in the second equality of \eqref{uphi} according to whether a scattering or fission event occurs. (The reader may wish to recall the role of the quantities $\sigma_{\texttt{s}}$, $\sigma_{\texttt f}$, $\sigma = \sigma_{\texttt s} +\sigma_{\texttt f}$, $\pi_{\texttt s}$ and $\sigma_{\texttt f}$ in \eqref{bNTE}). We get
\begin{align}
u^\varphi_t[g](r,\upsilon)
& = g(r+\upsilon(t \wedge \kappa^D_{(r,\upsilon)}), \upsilon)\frac{\varphi(r+\upsilon t, \upsilon)}{\varphi(r,\upsilon)}{\rm e}^{-\int_0^t \lambda_* + \sigma(r+\upsilon s, \upsilon)\d s}\mathbf{1}_{\{t<\kappa^D_{(r,\upsilon)}\}}\notag\\
&+\int_0^{t\wedge \kappa^D_{(r,\upsilon)} }
\sigma_{\texttt{s}}(r+\upsilon s, \upsilon){\rm e}^{-\int_0^s \lambda_* + \sigma(r+\upsilon \ell, \upsilon)\d \ell}
\frac{\varphi(r+\upsilon s,\upsilon)}{\varphi(r,\upsilon)}\notag\\
&\hspace{5cm}\int_V u^\varphi_{t-s}[g](r+\upsilon s, \upsilon') \frac{\varphi(r+\upsilon s, \upsilon')}{\varphi(r+\upsilon s, \upsilon)}\pi_{\texttt s}(r+\upsilon s, \upsilon, \upsilon')\d\upsilon'\notag\\
&+\int_0^{t\wedge \kappa^D_{(r,\upsilon)} }
\sigma_{\texttt{f}}(r+\upsilon s, \upsilon){\rm e}^{-\int_0^s \lambda_* + \sigma(r+\upsilon \ell, \upsilon)\d \ell}
\frac{\varphi(r+\upsilon s,\upsilon)}{\varphi(r,\upsilon)}\notag\\
&\hspace{1.5cm}
 {\mathcal E}_{(r+\upsilon s, \upsilon)}\otimes\mathbb{E}_{\delta_{(r,\upsilon)}}\left[\sum_{i=1}^N 
\frac{ \varphi(r+\upsilon s, \upsilon_i) }{\varphi(r+\upsilon s, \upsilon)} W^{ i}_{t-s}(r+\upsilon s, \upsilon_i)
 \prod_{j = 1}^N {\rm e}^{\langle \log g , X^j_{t-s}(r+\upsilon s, \upsilon_i)\rangle}\right],
\label{2factsneeded}
\end{align}
where, for $r\in D$, $v\in V$, $W^{ i}(r,\upsilon)$ and $X^i(r,\upsilon)$ are independent copies of the pair  $W$ and $X$ under $\mathbb{P}_{\delta_{(r, \upsilon)}}$. Note that the first term on the right-hand side of~\eqref{2factsneeded} contains includes $g(r+\upsilon(t \wedge \kappa^D_{(r,\upsilon)}), \upsilon)$ to account for the fact that $g(\{\dagger\}) = 1$. Before developing the right-hand side above any further, we need to make two additional observations and to introduce some more notation.

\smallskip

The first observation is that, since $W$ is a martingale, by sampling at the time of the first scattering event, fission event or when it leaves the domain $D$, whichever happens first, thanks to Doob's Optional Sampling Theorem, its mean must remain equal to 1 and we get the functional equation
\begin{align}
&\varphi(r,\upsilon)\notag\\
& = \varphi(r+\upsilon t, \upsilon){\rm e}^{-\int_0^t \lambda_* + \sigma(r+\upsilon s, \upsilon)\d s}\mathbf{1}_{\{t<\kappa^D_{r,\upsilon}\}}\notag\\
&+\int_0^{t\wedge \kappa^D_{r,\upsilon}} {\rm e}^{-\int_0^s \lambda_* +\sigma(r+\upsilon \ell, \upsilon)\d \ell}
\sigma_{\texttt f}(r+\upsilon s, \upsilon) \int_V\frac{ \varphi (r+\upsilon s, \upsilon')}{\varphi(r+\upsilon s, \upsilon)}\pi_{\texttt f}(r+\upsilon s,\upsilon, \upsilon')\varphi(r+\upsilon s, \upsilon)\d s\notag\\
&+\int_0^{t\wedge \kappa^D_{r,\upsilon}} {\rm e}^{-\int_0^s\lambda_* + \sigma(r+\upsilon \ell, \upsilon)\d \ell}
\sigma_{\texttt s}(r+\upsilon s, \upsilon) \int_V \frac{\varphi (r+\upsilon s, \upsilon')}{\varphi(r+\upsilon s, \upsilon)} \pi_{\texttt s}(r+\upsilon s,\upsilon, \upsilon')\varphi(r+\upsilon s, \upsilon)\d s
\label{useDynkin}
\end{align}
for $r\in D$, $\upsilon\in V$. Now appealing to Lemma 1.2, Chapter 4 in \cite{Dynkin2}, to treat the last two terms of \eqref{useDynkin} as potentials, with a little bit of algebra  we can otherwise write the above as
\[
\varphi(r,\upsilon) = \varphi(r+\upsilon t, \upsilon)\exp\left\{\int_0^t \dfrac{(\bS + \bF-\lambda_* {\texttt I})\varphi(r+\upsilon s, \upsilon)}{\varphi(r+\upsilon s, \upsilon)}\d s\right\},
\]
for $t<\kappa^D_{r,\upsilon}$, where $\texttt{I}$ is the identity operator,
which is to say, for $t<\kappa^D_{r,\upsilon}$,
\begin{align}
\frac{\varphi(r+\upsilon t, \upsilon)}{\varphi(r,\upsilon)} &{\rm e}^{-\int_0^t \lambda_* + \sigma(r+\upsilon s, \upsilon)\d s} 
\notag\\
& = 
\exp\left\{-\int_0^t \dfrac{(\bS + \bF + \sigma\texttt{I})\varphi(r+\upsilon s, \upsilon)}{\varphi(r+\upsilon s, \upsilon)}\d s\right\}.
\label{obs1}
\end{align}
\smallskip

Our second observation pertains to the manipulation of the expectation on the right-hand side of \eqref{2factsneeded}. Define for $g\in L^+_{\infty}(D\times V)$,  $(r,\upsilon)\in D\times V$ and $t\geq 0$,
\begin{equation}
u_t[g](r,\upsilon): 
= \mathbb{E}_{\delta_{(r,\upsilon)}}\left[ \prod_{i = 1}^{N_t}g (R_t^i, \Upsilon^i_t)\right]
\label{ut}
\end{equation}

We have that for all $(r,\upsilon) \in D\times V$,
\begin{align}
&\sigma_{\texttt{f}}(r, \upsilon){\mathcal E}_{(r+\upsilon s, \upsilon)}\otimes\mathbb{E}_{\delta_{(r,\upsilon)}}\left[\sum_{i=1}^N 
\frac{ \varphi(r, \upsilon_i) }{\varphi(r, \upsilon)} W^{ i}_{t-s}(r, \upsilon_i)
 \prod_{j = 1}^N {\rm e}^{\langle \log g , X^j_{t-s}(r, \upsilon_i)\rangle}\right]\notag\\
&= \sigma_{\texttt{f}}(r, \upsilon){\mathcal E}_{(r+\upsilon s, \upsilon)}\otimes\mathbb{E}_{\delta_{(r,\upsilon)}}\Bigg[\frac{\langle\varphi, \mathcal{Z}\rangle}{ \varphi(r, \upsilon)}\sum_{i=1}^N \frac{ \varphi(r, \upsilon_i)}{\langle\varphi, \mathcal{Z}\rangle}
W^{ i}_{t-s}(r, \upsilon_i)
{\rm e}^{\langle \log g , X^i_{t-s}(r, \upsilon_i)\rangle}\prod_{\stackrel{j =1}{i\neq j}}^N {\rm e}^{\langle \log g , X^j_{t-s}(r, \upsilon_i)\rangle}\Bigg]\notag\\
&= \sigma_{\texttt{f}}(r, \upsilon)\frac{{\mathcal E}_{(r,\upsilon)}[\langle\varphi, \mathcal{Z}\rangle]}{ \varphi(r, \upsilon)}
 {\mathcal E}_{(r, \upsilon)}\Bigg[\frac{\langle\varphi, \mathcal{Z}\rangle}{{\mathcal E}_{(r,\upsilon)}[\langle\varphi, \mathcal{Z}\rangle]}\sum_{i=1}^N \frac{ \varphi(r, \upsilon_i)}{\langle\varphi, \mathcal{Z}\rangle}
u^\varphi_{t-s}[g](r, \upsilon_i)\prod_{\stackrel{j =1}{i\neq j}}^N u_{t-s}[g](r, \upsilon_j)\Bigg]\notag\\
&=
G^\varphi_{\texttt f}[u_{t-s}^\varphi[g], u_{t-s}[g]](r,\upsilon) +  \frac{( \bF +\sigma_{\texttt f}{\texttt I})\varphi(r,\upsilon)}{\varphi(r,\upsilon)}u_{t-s}^\varphi[g](r,\upsilon)
,
\label{obs2}
\end{align}
where in the penultimate equality we have taken expectations conditional on the fission event and
\begin{align}
G^\varphi_{\texttt f}[f, g](r,\upsilon)&: = \frac{( \bF +\sigma_{\texttt f}{\texttt I})\varphi(r,\upsilon)}{\varphi(r,\upsilon)}
{\mathcal E}^\varphi_{(r,\upsilon)}\Bigg[\sum_{i=1}^N \frac{ \varphi(r, \upsilon_i)}{\langle\varphi, \mathcal{Z}\rangle}
f(r, \upsilon_i)\prod_{\stackrel{j =1}{i\neq j}}^N g(r, \upsilon_j) \Bigg]\notag\\
&\hspace{2cm}- \frac{( \bF +\sigma_{\texttt f}{\texttt I})\varphi(r,\upsilon)}{\varphi(r,\upsilon)}f(r, \upsilon)
\label{Gphif}
\end{align}
for $f,g\in L^+_\infty(D\times V)$, which are uniformly bounded by unity, and for $r\in D$, $\upsilon\in V$, where we recall that ${\mathcal P}^\varphi_{(r,\upsilon)}$ was defined in \eqref{pointprocessCOM}.
Note in particular that
\begin{equation}
\frac{{\mathcal E}_{(r,\upsilon)}[\langle\varphi, \mathcal{Z}\rangle]}{\varphi(r,\upsilon)} = 
\int_V \frac{\varphi(r, \upsilon')}{\varphi(r,\upsilon)}\pi_{\texttt f}(r,\upsilon,\upsilon')\d \upsilon' =\frac{( \bF +\sigma_{\texttt f}{\texttt I})\varphi(r,\upsilon)}{\sigma_{\texttt f}(r,\upsilon)\varphi(r,\upsilon)}.
\label{<Z>}
\end{equation}
We will also make use of the notation 
\begin{equation}
G_{\texttt{f}}[f](r,\upsilon) =  \sigma_{\texttt f}(r,\upsilon){\mathcal E}_{(r,\upsilon)}\Bigg[\prod_{j =1}^N g(r, \upsilon_j) - g(r,\upsilon)\Bigg], 
\label{Gf}
\end{equation}
for $r\in D$, $\upsilon \in V$ and $g\in L^+_\infty(D\times V)$, which is uniformly bounded by unity. Recall that  the empty product in the definition \eqref{ut} is defined as unity.
\smallskip

In a similar manner to \eqref{2factsneeded} we can break the expectation over the event of scattering or fission in \eqref{ut}, which defines of $u_t[g]$, to see that 
the operator  $G_{\texttt f}$ appears in the decomposition
\begin{equation}
u_t[g] (r,\upsilon) = \hat\U_t[g] +\int_0^t\U_s[\bS u_{t-s}[g] + G_{\texttt{f}}[u_{t-s}[g]]\d s, \qquad t\geq 0,
\label{ut}
\end{equation}
for $g\in L^+_\infty(D\times V)$, which is uniformly bounded by unity. Here,  we have adjusted the definition of the semigroup $\U$ to 
\begin{equation}
\hat\U_t[g](r,\upsilon) = g(r+\upsilon (t\wedge \kappa^D_{r,\upsilon}), \upsilon), \qquad t\geq 0, r\in D, \upsilon\in V.
\label{tildeU}
\end{equation}
\smallskip

Now returning to \eqref{2factsneeded} with the above observations and definitions  in hand, whilst again appealing to Lemma 1.2, Chapter 4 in \cite{Dynkin2}, we have 
\begin{align}
u^\varphi_t[g](r,\upsilon)&=
g(r+\upsilon (t\wedge \kappa^D_{r,\upsilon}), \upsilon)\notag\\
&+\int_0^{t\wedge \kappa^D_{(r,\upsilon)}} \sigma_{\texttt{s}}(r+\upsilon s, \upsilon)
\int_V  u_{t-s}^\varphi[g](r+\upsilon s, \upsilon')\frac{\varphi(r+\upsilon s, \upsilon')}{\varphi(r+\upsilon s, \upsilon)} \pi_{\texttt s}(r+\upsilon s, \upsilon, \upsilon')\d\upsilon'\d s\notag\\
&+\int_0^{t\wedge \kappa^D_{(r,\upsilon)}}     
G^\varphi_{\texttt f}[u_{t-s}^\varphi[g], u_{t-s}[g]](r+\upsilon s, \upsilon)+  \frac{( \bF +\sigma_{\texttt f}{\texttt I})\varphi(r+\upsilon s,\upsilon)}{\varphi(r+\upsilon s,\upsilon)}u_{t-s}^\varphi[g](r+\upsilon s,\upsilon)\d s\notag\\
 &-\int_0^{t\wedge \kappa^D_{(r,\upsilon)}}    \dfrac{(\bS + \bF +\sigma\texttt{I})\varphi(r+\upsilon s, \upsilon)}{\varphi(r+\upsilon s, \upsilon)}
 u^\varphi_{t-s}[g](r+\upsilon s, \upsilon)\d s\label{spinesg2}\\
 &= \hat\U_t[g] 
 +\int_0^t \U_s\left[
\bS_\varphi u_{t-s}^\varphi[g]
\right]
\d s+\int_0^t   \U_s\bigg[
G^\varphi_{\texttt f}[u^\varphi_{t-s}[g], u_{t-s}[g]]
\bigg]\d s,
\label{spinesg}
\end{align}
where
where
\begin{equation*}
{\bS}_\varphi f(r,v) : = \int_V [f(r,\upsilon') -f(r,\upsilon)]\sigma_{\texttt s}(r,\upsilon)\frac{\varphi(r,\upsilon')}{\varphi(r,\upsilon)}\pi_{\texttt s}(r,\upsilon,\upsilon')\d \upsilon'
\end{equation*}
on $D\times V$ and otherwise equal to zero,

\bigskip

{\it Step 2}.  Define 
\begin{equation}
\tilde{u}^\varphi_t[g](r,\upsilon)= \tilde{\mathbb{E}}^\varphi_{\delta_{(r,\upsilon)}}\left[\prod_{i = 1}^{N_t}g (R_t^i, \Upsilon^i_t)\right], \qquad t\geq 0,
\label{uttild}
\end{equation}
for $g\in L^+_\infty(D\times V)$,
where $\{(R_t^i, \Upsilon^i_t): i = 1,\cdots, N_t\}$ are the physical configurations of the particles alive in the system at time $t\geq 0$ in $X^\varphi$. 

\smallskip

By conditioning $\tilde{u}^\varphi_t$ on  the first time a scattering of fission event occurs, it is a straightforward exercise to show that it also solves \eqref{spinesg}. For the sake of brevity, we leave this as an exercise to the reader as the arguments are similar to those mentioned previously. In order to show that \eqref{spinesg} has a unique solution, we consider $v_t^\varphi[g] \coloneqq \varphi u_t^\varphi[g]$ and $\tilde{v}_t^\varphi[g] \coloneqq \varphi \tilde{u}_t^\varphi[g]$. Since $u_t^\varphi$ and $\tilde{u}_t^\varphi$ both satisfy \eqref{spinesg2}, applying ~\cite[Lemma 1.2, Chapter 4]{Dynkin2} along with \eqref{obs1}, it is straightforward to show that $v_t^\varphi$ and $\tilde{v}_t^\varphi$ both satisfy
\begin{align}
v_t[g](r, \upsilon) &= g(r+\upsilon(t \wedge \kappa^D_{(r,\upsilon)}), \upsilon)\varphi(r+\upsilon t, \upsilon){\rm e}^{-\int_0^t\lambda_* + \sigma(r+\upsilon s, \upsilon){\rm d}s} \mathbf{1}_{\{t < \kappa^D_{(r, \upsilon)}\}}\notag\\
&\quad + \int_0^{t \wedge \kappa^D_{(r, \upsilon)}}{\rm e}^{-\int_0^s\lambda_* + \sigma(r+\upsilon l, \upsilon){\rm d}l}{\texttt U}_s[(\bS + \sigma_\texttt{s})v_{t-s}[g]]{\rm d}s \notag\\
&\quad+ \int_0^{t \wedge \kappa^D_{(r, \upsilon)})}{\rm e}^{-\int_0^s\lambda_* + \sigma(r+\upsilon l, \upsilon){\rm d}l}{\texttt U}_s[\tilde{G}_\texttt{f}[v_{t-s}[g], u_{t-s}[g]] + \sigma_\texttt{f}v_{t-s}[g]]{\rm d}s, \label{uniquesg}
\end{align}
where 
\begin{equation*}
\tilde{G}_\texttt{f}[f, g](r, \upsilon) = \sigma_\texttt{f}(r, \upsilon)\left\{\mathcal{E}_{(r, \upsilon)}\left[\sum_{i=1}^Nf(r, \upsilon_i)\prod_{\stackrel{j=1}{j\neq i}}^Ng(r, \upsilon_j) \right] - f(r, \upsilon)\right\}.
\end{equation*}
Due to the assumptions (H1) and (H4), an application of Gr\"onwall's inequality implies uniqueness of \eqref{uniquesg}, which in turn implies uniqueness of \eqref{spinesg}. We leave this as an exercise to the reader as it is a relatively standard computation and very similar to the calculations given in \cite{MultiNTE}.

\bigskip

{\it Step 3}.  We start by noting that the joint process $(X^\varphi, (R^\varphi, \Upsilon^\varphi))$ is, by construction, Markovian under $\tilde{\mathbb{P}}^\varphi$, we thus need to show that the marginalisation of the coupled system to just $X^\varphi$ retains the Markov property. We do this by showing that  for $f\in L^+_\infty(D\times V)$, $\mu \in \mathcal{M}(D\times V)$ and $(r,\upsilon)\in D\times V$,
\begin{equation}
\tilde{\mathbb{E}}^\varphi_{\mu}\left[f(R^\varphi_t, \Upsilon^\varphi_t)| X^\varphi_t\right] = \frac{\langle f\varphi, X^\varphi_t\rangle}{\langle\varphi, X^\varphi_t\rangle}, \qquad t\geq 0.
\label{empirical}
\end{equation}
This says that knowing $X^\varphi_t$ only allows one to construct the law of $(R^\varphi_t, \Upsilon^\varphi_t)$ through an empirical distribution using $\varphi$. Hence, for $g\in L^+_\infty(D\times V)$ which is bounded by unity and $\mu\in \mathcal{M}(D\times V)$,
\begin{align*}
\tilde{\mathbb{E}}^\varphi_\mu\left[{\rm e}^{\langle\log g , X^\varphi_{t+s}\rangle}|\mathcal{F}_t\right]&=
\tilde{\mathbb{E}}^\varphi_\mu\left[ \sum_{i = 1}^{N_t} 
\frac{\varphi(R^i_t, \Upsilon^i_t)}{\langle \varphi, X^\varphi_t\rangle} \tilde{\mathbb{E}}^\varphi_{\mu', (r,\upsilon)} \left[  {\rm e}^{\langle\log g , X^\varphi_{s}\rangle}\right]_{\mu' = X^\varphi_t, (r,\upsilon) (R^i_t, \Upsilon^i_t)}\right]\\
&=\tilde{\mathbb{E}}^\varphi_\mu\left[ \tilde{\mathbb{E}}^\varphi_{\mu'}\left[{\rm e}^{\langle\log g , X^\varphi_{s}\rangle} \right]_{\mu' = X^\varphi_t}\right],
\end{align*}
where we have written $X^\varphi_t =\textstyle{ \sum_{i =1}^{N_t}} \delta_{(R^i_t, \Upsilon^i_t)}$,
and thus the Markov property of $X^\varphi, \tilde{\mathbb{P}}^\varphi$ follows.
\smallskip

We are thus left with proving \eqref{empirical} to complete this step. To do so we note that it suffices to show that for $f,g\in L^+_\infty (D\times V)$ bounded by unity, $\mu \in \mathcal{M}(D\times V)$ and $(r,\upsilon)\in D\times V$,

\begin{equation}
\tilde{\mathbb{E}}^\varphi_{\mu}\left[f(R^\varphi_t, \Upsilon^\varphi_t){\rm e}^{\langle\log g, X^\varphi_t\rangle}\right] = \tilde{\mathbb{E}}^\varphi_{\mu}\left[\frac{\langle f\varphi, X^\varphi_t\rangle}{\langle\varphi, X^\varphi_t\rangle}{\rm e}^{\langle\log g, X^\varphi_t\rangle}\right] , \qquad t\geq 0.
\label{checkthis}
\end{equation}

On the left-hand side  of \eqref{checkthis}, we have 
\begin{align*}
&\tilde{\mathbb{E}}^\varphi_{\mu}\left[f(R^\varphi_t, \Upsilon^\varphi_t){\rm e}^{\langle\log g, X^\varphi_t\rangle}\right] \\
&=\tilde{\mathbb{E}}^\varphi_{\mu}\left[\tilde{\mathbb{E}}^\varphi_{\mu}\left[f(R^\varphi_t, \Upsilon^\varphi_t){\rm e}^{\langle\log g, X^\varphi_t\rangle}|R^\varphi_t, \Upsilon^\varphi_t\right] \right] \\
&=\sum_{k = 1}^n \frac{\varphi(r_k, \upsilon_k)}{\langle\varphi, \mu\rangle}
\tilde{\mathbb{E}}^\varphi_{\delta_{(r_k,\upsilon_k)}}\left[f(R^\varphi_t, \Upsilon^\varphi_t)
\prod_{i \ge 1 : T_i\leq t}
\prod_{ j = 1}^{N^i} \frac{u_{t-T_i}[g](R^\varphi_{T_i}, \upsilon^i_j)}{u_{t-T_i}[g](R^\varphi_{T_i}, \Upsilon^\varphi_{T_i})}\right]
\end{align*}
where $\mu =\textstyle{ \sum_{k =1}^{n} } \delta_{(r_i, \upsilon_i)}$, $T_i$, $i\geq 1$ are the times of fission along the spine at which point $N^i$ particles are issued at $\upsilon^i_j$, $j = 1\cdots, N^i_j$.  On the right-hand side of  of \eqref{checkthis}, we may appeal to Step 1 and Step 2 to deduce that 
\begin{align*}
\tilde{\mathbb{E}}^\varphi_{\mu}\left[\frac{\langle f\varphi, X^\varphi_t\rangle}{\langle\varphi, X^\varphi_t\rangle}{\rm e}^{\langle\log g, X^\varphi_t\rangle}\right]  &= 
{\rm e}^{-\lambda_* t}\mathbb{E}_{\mu}\left[\frac{\langle f\varphi, X_t\rangle}{\langle\varphi, \mu\rangle}{\rm e}^{\langle\log g, X_t\rangle}\right] \\
& = 
\sum_{i = 1}^n \frac{\varphi(r_i,\upsilon_i)}{\langle\varphi, \mu\rangle}{\rm e}^{-\lambda_* t}\mathbb{E}_{\delta_{(r_i,\upsilon_i)}}\left[\frac{\langle f\varphi, X_t\rangle}{\varphi(r_i,\upsilon_i)}{\rm e}^{\langle\log g, X_t\rangle}\right].
\end{align*}
The proof of this final step is thus complete as soon as we can show that 
\begin{align}
&\tilde{\mathbb{E}}^\varphi_{\delta_{(r_k,\upsilon_k)}}\left[f(R^\varphi_t, \Upsilon^\varphi_t)
\prod_{1\leq i : T_i\leq t}
\prod_{ j = 1}^{N^i} \frac{u_{t-T_i}[g](R^\varphi_{T_i}, \upsilon^i_j)}{u_{t-T_i}[g](R^\varphi_{T_i}, \Upsilon^\varphi_{T_i})}\right]={\rm e}^{-\lambda_* t}\mathbb{E}_{\delta_{(r_i,\upsilon_i)}}\left[\frac{\langle f\varphi, X_t\rangle}{\varphi(r_i,\upsilon_i)}{\rm e}^{\langle\log g, X_t\rangle}\right]
\label{finalbit}
\end{align}
To this end, we note that splitting the expectation on the right-hand of \eqref{finalbit} side at either a scattering or fission event results in a calculation that is almost identical to the one above that concludes with \eqref{spinesg}. More precisely, the expectation on the right-hand side solves \eqref{spinesg} albeit the role of $\hat{\U}_t[g]$ is replaced by $\U_t[fg]$. 
Similarly splitting the expectation on the left-hand side of \eqref{finalbit} also results in a solution to \eqref{spinesg} (with the aforementioned adjustment). The uniqueness of \eqref{spinesg}, with $\hat{\U}$ replaced by $\U$, follows from the same arguments and hence the equality in \eqref{finalbit} now follows, as required.  \hfill $\square$

\section{Proof of Lemma \ref{spinemarkov}}\label{spineproofsection}
The fact that the spine is Markovian is immediate from its definition of $(R^\varphi, \Upsilon^\varphi)$. Indeed, once given its initial configuration, it evolves as  the NRW associated to the rate  $\varphi^{-1}(r,\upsilon)\sigma_{\texttt s}(r,\upsilon)\varphi(r,\upsilon')\pi_{\texttt s}(r,\upsilon,\upsilon')$. Moreover, when in configuration $(r,\upsilon)\in D\times V$,  at rate $\varphi(r,\upsilon)^{-1}( \bF +\sigma_{\texttt f}{\texttt I})\varphi(r,\upsilon)$, it experiences an additionally scattering with new velocity $\upsilon'$, with distribution 
\begin{align*}
{\mathcal E}_{(r,\upsilon)}\left[\frac{\langle\varphi, \mathcal{Z}\rangle}{{\mathcal E}_{(r,\upsilon)}[\langle\varphi, \mathcal{Z}\rangle]}\frac{ \langle \varphi\mathbf{1}_{(\cdot\in \d \upsilon')}, \mathcal{Z}\rangle}{\langle\varphi, \mathcal{Z}\rangle}\right]
=\frac{\sigma_{\texttt f}(r,\upsilon)}{( \bF +\sigma_{\texttt f}{\texttt I})\varphi(r,\upsilon)}\varphi(r,\upsilon')\pi_{\texttt f}(r,\upsilon,\upsilon')\d\upsilon',
\end{align*}
for $\upsilon'\in V$,  where we have used \eqref{<Z>}. The total scatter rate is thus 
\begin{align}
&\sigma_{\texttt s}(r,\upsilon)\frac{\varphi(r,\upsilon')}{\varphi(r,\upsilon)}\pi_{\texttt s}(r,\upsilon,\upsilon') + 
\sigma_{\texttt f}(r,\upsilon)\frac{\varphi(r,\upsilon')}{\varphi(r,\upsilon)}\pi_{\texttt f}(r,\upsilon,\upsilon')\d\upsilon'\notag \\
&= \alpha(r,\upsilon)\frac{\varphi(r,\upsilon')}{\varphi(r,\upsilon)}\pi(r,\upsilon, \upsilon')\\
&=\alpha^\varphi(r,\upsilon)\pi^\varphi(r,\upsilon, \upsilon')\notag
\label{phipi}
\end{align}
as required.

\smallskip

For the second statement, write 
\begin{equation}
\psi^\varphi_t[g](r,\upsilon): =\mathbf{E}_{(r,\upsilon)}\left[
{\rm e}^{-\lambda_* t +\int_0^t \beta(R_s, \Upsilon_s)\d s}\frac{\varphi(R_t, \Upsilon_t)}{\varphi (r,\upsilon)}g(R_t,\Upsilon_t)\mathbf{1}_{\{t<\tau^D\}}
\right].
\label{psiphi}
\end{equation}

 By conditioning the expectation on the right-hand side on the first scattering event  we have, for $t\geq 0$, $r\in D$ and $\upsilon\in V$, 
 \begin{align}
 \psi^\varphi_t[g](r,\upsilon) &=
 {\rm e}^{-\lambda_* t +\int_0^t \beta(r+\upsilon\ell, \upsilon) -\alpha(r+\upsilon\ell, \upsilon)\d \ell}
 \frac{\varphi(r+\upsilon t, \upsilon)}{\varphi (r,\upsilon)}
 g(r+\upsilon t, \upsilon)\mathbf{1}_{(t<\kappa^D_{r,\upsilon})}\notag\\
 &+\int_0^{t\wedge \kappa^D_{r,\upsilon}}  \alpha(r+\upsilon s, \upsilon){\rm e}^{-\lambda_* s +\int_0^s \beta(r+\upsilon\ell, \upsilon) -\alpha(r+\upsilon\ell, \upsilon)\d \ell}\frac{\varphi(r+\upsilon s, \upsilon)}{\varphi (r,\upsilon)}\notag\\
 &\hspace{2cm}
 \int_V  \psi^\varphi_{t-s}[g](r + \upsilon s,\upsilon') \frac{\varphi(r+\upsilon s, \upsilon')}{\varphi (r+\upsilon s,\upsilon)} \pi(r + \upsilon s,\upsilon,\upsilon')\d \upsilon'\d s
 \label{mess}
 \end{align}
   Now appealing to \eqref{obs1}, then using the   standard trick of replacing the role of an additive potential by the role of a multiplicative potential in such semigroup evolutions, see e.g. Lemma 1.2, Chapter 4 in \cite{Dynkin2},  and noting \eqref{psiphi} we get
  \begin{equation}
  \psi^\varphi_t[g](r,\upsilon) = {\texttt U}_t[g](r,\upsilon)+ \int_0^{t} {\texttt U}_s [ (\bL_\varphi +\alpha^\varphi {\texttt I})\psi^\varphi_{t-s}[g]-(\varphi^{-1}(\bS+\bF)\varphi  +(\beta -\alpha){\texttt I})\psi^\varphi_{t-s}[g]](r,\upsilon) \d s
  \label{Lphieq}
 \end{equation}
 where  
  \begin{equation}
  \bL_\varphi f(r,\upsilon) = \alpha^\varphi(r,\upsilon)\int_V [f(r,\upsilon')- f(r,\upsilon)] \pi^\varphi(r,\upsilon,\upsilon')\d\upsilon',
  \label{Lphidef}
 \end{equation}
 for $f\in L^+_\infty(D\times V)$. Referring to \eqref{S}, \eqref{F} and \eqref{betadef}, we note that
 \begin{align*}
 &\alpha^\varphi(r,\upsilon)-\frac{(\bS+\bF)\varphi(r,\upsilon)}{\varphi(r,\upsilon)}  -\beta(r,\upsilon) +\alpha(r,\upsilon)\\
 &=\int_V \alpha(r,\upsilon)\frac{\varphi(r,\upsilon')}{\varphi(r,\upsilon)}\pi(r,\upsilon, \upsilon')\d \upsilon'\\
 & \hspace{1cm}-
 \int_V\frac{\varphi(r,\upsilon')}{\varphi(r,\upsilon)}(\sigma_{\texttt{s}}(r,\upsilon)\pi_{\texttt{s}}(r,\upsilon, \upsilon')+ \sigma_{\texttt{f}}(r,\upsilon)\pi_{\texttt{f}}(r,\upsilon, \upsilon'))\d \upsilon'-\sigma_{\texttt{s}}(r,\upsilon)-\sigma_{\texttt{f}}(r,\upsilon)\\
 &\hspace{1cm}-\sigma_{\texttt{f}}(r,\upsilon)\left(\int_V\pi_{\texttt{f}}(r, \upsilon,\upsilon')\d\upsilon'-1\right)
 +\sigma_{\texttt{s}}(r,\upsilon)+ \sigma_{\texttt{f}}(r,\upsilon)\int_V\pi_{\texttt{f}}(r,\upsilon, \upsilon')\d\upsilon'\\
 &=0.
 \end{align*}
Hence \eqref{Lphieq} reduces to the somewhat simpler recurrence equation
\[
\psi^\varphi_t[g](r,\upsilon) = {\texttt U}_t[g](r,\upsilon)+ \int_0^{t} {\texttt U}_s [ \bL_\varphi \psi^\varphi_{t-s}[g]](r,\upsilon) \d s, \qquad t\geq 0,
\]
where we recall that $\bL_\varphi$ was defined in \eqref{Lphidef}.
This is nothing more than the mild equation for the semigroup evolution $\mathbf{E}^\varphi_{(r,\upsilon)}[g(R_t,\Upsilon_t)]$, $t\geq 0$, which has a unique bounded solution from  the usual Gr\"onwall arguments. Note that when $g = 1$, we see the solution is 1. This, together with the Markov property implies that the right-hand side of \eqref{NRWCOM} is a martingale. Moreover, it follows that the martingale change of measure in \eqref{NRWCOM} describes  law of the $\alpha^\varphi\pi^\varphi$-NRW.

\smallskip 
The fact that $\psi^\varphi_t[1](r,\upsilon) = 1$ for all $r\in D, \upsilon\in V$, implies that $((R,\Upsilon), \mathbf{P}^\varphi)$ is conservative.  Moreover, 
\[
\mathbf{P}^\varphi_{(r,\upsilon)}[g(R_t,\Upsilon_t)]= \psi^\varphi_t[g](r,\upsilon) = {\rm e}^{-\lambda_*t}\frac{\psi_t[g\varphi](r,\upsilon)}{\varphi(r,\upsilon)}, \qquad r\in D, \upsilon\in V.
\]
where $g\in L^+_\infty(D\times V)$; cf. \eqref{psiphi}. Hence $\lim_{t\to\infty}\P^\varphi_{(r,\upsilon)}[g(R_t,\Upsilon_t)] = \langle g,\tilde\varphi\varphi\rangle$ for all $g\in L^+_\infty(D\times V)$.
In other words, $\tilde\varphi\varphi$ is the density of the stationary distribution of $(R,\Upsilon)$ under $\P^\varphi$.
\hfill$\square$

\section{Proof of Theorem \ref{Kesten}}\label{kestensection}
The proof we offer here is a variant of a  standard one, which has been used to analyse the convergence of many analogous martingales in the setting of different spatial branching processes. We mention \cite{Lyons}, \cite{Shi}, \cite{Bertbook2} and \cite{EngK} to name but a few of the contexts  with similar results. 

\smallskip

{\color{black}In the case that $\lambda_*<0$ and $\lambda_*>0$, we need (H3) to ensure that the NBP can  undergo fission. In the setting $\lambda_*=0$ we need the stricter condition (H3)$^*$ for technical reasons in the proof to ensure a minimal rate of reproduction.
}
\smallskip

A standard measure theoretic result (cf. p. 242 of \cite{Durrett}) tells us that the martingale change of measure in \eqref{mgCOM} is uniformly integrable if and only if 
\[
\mathbb{P}^\varphi_{\delta_{(r,\upsilon)}}\left(\limsup_{t\to\infty} W_t <\infty\right)=1.
\]
In the case that $\limsup_{t \to\infty}W_t = \infty$, $\mathbb{P}^\varphi$ almost surely, we have $\mathbb{P}_{\delta_{(r,\upsilon)}}(W_\infty = 0)=1$.

\smallskip

(i) Let us first deal with the case that $\lambda_*>0$. To this end, let us define the sigma algebra $\mathcal{S} = \sigma(T_i, \mathcal{Z}_i: i\geq 0)$, where $T_i$, $i\geq 1$, are the times at which the spine undergoes fission and $\mathcal{Z}_i$, $i\geq 1$, are point processes on $V$ that describe the velocities of fission offspring (i.e. whose law is  given by  the family \eqref{PP} under the change of measure \eqref{pointprocessCOM}). For convenience we will write $T_0 = 0$.

\smallskip

Appealing to the pathwise spine decomposition in Theorem \ref{pathspine}, we can write 
\begin{align}
\mathbb{E}^\varphi_{\delta_{(r,\upsilon)}}[W_t] &\leq  \mathbb{E}^\varphi_{\delta_{(r, \upsilon)}}\left[{\rm e}^{-\lambda_* t}\frac{\varphi(R_t, \Upsilon_t)}{\varphi(r, \upsilon)} \right]\notag\\
&\qquad +\mathbb{E}^\varphi_{\delta_{(r,\upsilon)}}\left[\mathbb{E}^\varphi_{\delta_{(r,\upsilon)}}\left[\left.
\sum_{j =1}^\infty{\rm e}^{-\lambda_* T_j}  \mathbf{1}_{(T_j\leq t)}
\sum_{i = 1}^{N^j} \varphi(R_{T_{j}},  \upsilon_i)W^{j}_{t-T_{j}} (R_{T_j}, \upsilon_i)\right|\mathcal{S}
\right]\right]\notag\\
&=\mathbb{E}^\varphi_{\delta_{(r,\upsilon)}}\left[{\rm e}^{-\lambda_* t}\frac{\varphi(R_t, \Upsilon_t)}{\varphi(r, \upsilon)} + \sum_{j =1}^\infty{\rm e}^{-\lambda_* T_j}  \mathbf{1}_{(T_j\leq t)}
{\mathcal E}^\varphi_{(R_{T_j}, \Upsilon_{T_{j-1}})}[\langle\varphi, \mathcal{Z}_j\rangle]\right]\notag \\
&=\mathbb{E}^\varphi_{\delta_{(r,\upsilon)}}\left[{\rm e}^{-\lambda_* t}\frac{\varphi(R_t, \Upsilon_t)}{\varphi(r, \upsilon)} + \sum_{j =1}^\infty{\rm e}^{-\lambda_* T_j}  \mathbf{1}_{(T_j\leq t)}
\frac{
{\mathcal E}_{(R_{T_j}, \Upsilon_{T_{j-1}})}[\langle\varphi, \mathcal{Z}\rangle^2]
}{
{\mathcal E}_{(R_{T_j}, \Upsilon_{T_{j-1}})}[\langle\varphi, \mathcal{Z}\rangle]
}
\right] \label{rhsfinite}
\end{align}
where, for a given $(r,\upsilon)\in D\times V$, the process $W^{j}_{s} (r,\upsilon)$ is an independent copy of $(W_{s}, s\geq 0)$,  under $\mathbb{P}_{\delta_{(r,\upsilon)}}$ (and consequently has unit mean, which is also used above). 
Our objective is to prove that the sum on the right-hand side of \eqref{rhsfinite} is $\mathbb{P}^\varphi_{\delta_{(r, \upsilon)}}$-almost surely finite. In that case, it will follow with the help of Fatou's Lemma that 
\begin{equation}
\infty > \limsup_{t\to\infty}\mathbb{E}^\varphi_{\delta_{(r,\upsilon)}}[W_t] \geq 
\mathbb{E}^\varphi_{\delta_{(r,\upsilon)}}\left[\liminf_{t\to\infty}W_t\right].
\label{EIS}
\end{equation}
Recalling that $W$ is a non-negative $\mathbb{P}$-martingale, it holds that $1/W$ is a non-negative $\mathbb{P}^\varphi$-supermartingale and thus its limit exists. The conditional expectation in \eqref{EIS} ensures that $\textstyle{\liminf_{t\to\infty}W_t}$ (and hence from the immediately preceding remarks $\textstyle{\limsup_{t\to\infty}W_t}$) is $\mathbb{P}^\varphi_{\delta_{(r,\upsilon)}}$-almost surely finite. 

\smallskip
We must thus show that the upper bound on the right-hand side of \eqref{rhsfinite} is $\mathbb{P}^\varphi_{\delta_{(r, \upsilon)}}$-almost surely finite. To do so, we again recall the description of the pathwise spine decomposition in Theorem \ref{pathspine} and note that fission along the spine occurs at the accelerated rate $\varphi^{-1}( \bF +\sigma_{\texttt f}{\texttt I})\varphi$. Hence  (recalling the generic point process $\mathcal Z$ defined in \eqref{PP})
\begin{align}
&\mathbb{E}^\varphi_{\delta_{(r,\upsilon)}}\left[\sum_{j =1}^\infty{\rm e}^{-\lambda_* T_j}  \mathbf{1}_{\{T_j\leq t\}}
\frac{
{\mathcal E}_{(R_{T_j}, \Upsilon_{T_{j-1}})}[\langle\varphi, \mathcal{Z}\rangle^2]
}{
{\mathcal E}_{(R_{T_j}, \Upsilon_{T_{j-1}})}[\langle\varphi, \mathcal{Z}\rangle]
}
\right]\notag\\ 
&\le (\norm{\varphi}_\infty n_{\texttt{max}})^2\mathbb{E}^\varphi_{\delta_{(r,\upsilon)}}\left[\sum_{j =1}^\infty{\rm e}^{-\lambda_* T_j} \frac{1}{
{\mathcal E}_{(R_{T_j}, \Upsilon_{T_{j-1}})}[\langle\varphi, \mathcal{Z}\rangle]
}\right]\notag\\
& =
(\norm{\varphi}_\infty n_{\texttt{max}})^2\mathbf{E}^\varphi_{(r,\upsilon)}\left[\int_0^\infty {\rm e}^{-\lambda_* t}
\frac{\sigma_{\texttt f}(R_t, \Upsilon_t)}{
{\mathcal E}_{(R_{t}, \Upsilon_{t})}[\langle\varphi, \mathcal{Z}\rangle] 
}
\int_V \frac{\varphi(R_t, \upsilon')}{\varphi(R_t, \Upsilon_t)} \pi_{\texttt f}(R_t, \Upsilon_t, \upsilon') \d \upsilon' \d t\right]\notag\\
&\leq 
\bar\sigma_{\texttt{f}}(\norm{\varphi}_\infty n_{\texttt{max}})^2\mathbf{E}^\varphi_{(r,\upsilon)}\left[\int_0^\infty 
 \frac{{\rm e}^{-\lambda_* t}}{\varphi(R_t, \Upsilon_t)} \d t\right]\notag\\
 &\leq
\bar\sigma_{\texttt{f}}\frac{(\norm{\varphi}_\infty n_{\texttt{max}})^2}{\varphi (r,\upsilon)} \mathbf{E}_{(r,\upsilon)}\left[\int_0^\infty 
{\rm e}^{-2\lambda_* t +\int_0^t \beta(R_s, \Upsilon_s)\d s}\mathbf{1}_{\{t<\tau^D\}} \d t\right]\notag\\
&\leq
\bar\sigma_{\texttt{f}}(\norm{\varphi}_\infty n_{\texttt{max}})^2 \int_0^\infty 
{\rm e}^{-2\lambda_* t } \frac{\psi_t[1](r, \upsilon)}{\varphi (r,\upsilon)} \d t,
\label{useinnextpaper}
\end{align}
 where we have used  (H4) in the first inequality, features of the spine decomposition for the first equality, \eqref{<Z>} and (H1) in the second inequality,  the change of measure \eqref{NRWCOM} in the third inequality and the semigroup \eqref{phi} for the final line. Finally, note that since $\varphi$ is uniformly bounded above, the contribution from the spine term in~\eqref{rhsfinite} is zero in the limit $t \to \infty$.  Now using 
Theorem \ref{CVtheorem}, it follows that
\[
\limsup_{t\to\infty}\mathbb{E}^\varphi_{\delta_{(r,\upsilon)}}\left[
W_t
\right]<\infty
\]
as required.
\smallskip

(ii) Next, for the case $\lambda_*< 0$, it is easy to see that, on the event $\{T_{j}\leq t< T_{j+1}\}$,
\[
W_t \geq {\rm e}^{-\lambda_* t}\varphi(R_{t},\Upsilon_{T_{j-1}} )
\]
which ensures that $\mathbb{P}^\varphi_{\delta_{(r,\upsilon)}} (\textstyle{\limsup_{t\to\infty}}W_t = \infty) =1 $ for all $(r,\upsilon)\in D\times V$ and hence $\mathbb{P}_{\delta_{(r,\upsilon)}} (W_\infty = 0) =1$.

\bigskip

(iii) Finally, for the case $\lambda_* = 0$, our aim is to  show that, for each $r\in D,\upsilon\in V$,
\[
\mathbb{P}^\varphi_{\delta_{(r, \upsilon)}}(\limsup_{t \to \infty}W_t = \infty)=1.
\]
We do this by constructing a random sequence of times $(s_n: n\geq 0)$ such that $\textstyle{\limsup_{n\to\infty}W_{s_n}= \infty}$ almost surely with respect to $\mathbb{P}^\varphi$. 

\smallskip

Lemma \ref{spinemarkov} tells us that $\tilde\varphi\varphi$ is the density of the stationary distribution of $(R,\Upsilon)$ under $\mathbf{P}^\varphi$.
Moreover, thanks to  Theorem \ref{CVtheorem}, the density $\tilde\varphi\varphi$ is a.e. uniformly bounded away from $0$ on each $\Omega\subset\subset D\times V$. It follows that $\langle\mathbf{1}_{\Omega},\varphi\tilde\varphi\rangle >0$ for all $\Omega\subseteq D_\varepsilon\times V$ and that the spine $R$ visits $\Omega$ infinitely often under $\mathbf{P}^\varphi$.

\smallskip

Fix $k \in \mathbb{N}$. We want to show that there is an $\Omega\subseteq D_\varepsilon\times V$ such that  
\begin{equation}
\inf_{(r',\upsilon')\in \Omega}   \mathbb{P}_{\delta_{(r', \upsilon')}}^\varphi(X_\iota(D_\varepsilon\times V) \ge k) > 0,
\label{inf}
\end{equation}
where $\iota =2{\rm diam}(D)/{\texttt v}_{\texttt{min}}$ (note that $\iota$ is twice the time it would take a neutron to cross the equivalent of the diameter of $D$, when moving at minimal speed).
To this end, fix $r\in D,\upsilon\in V$ and choose $\varepsilon>0$ sufficiently small such that both $r\in D_\varepsilon \coloneqq \textstyle{\{r \in D : \inf_{y\in \partial D}|r - y| > \varepsilon{\texttt v}_{\texttt{max}}\} }$ and $B$ {\color{black}(introduced in the assumption (H3)$^*$)} is in $\textstyle{D_{\varepsilon}}$. Then, define 
\[
\Omega = \left\{(r,\upsilon)\in D_\varepsilon \times V: \{r+ \upsilon s: s\geq 0\}\cap B \neq \emptyset \right\}.
\]

Write $m$ for the smallest natural number such that $m (n_{\texttt{max}}-1)+1 \geq k$. Recalling from Theorem~\ref{CVtheorem} that $\textstyle{\inf_{r \in D, \upsilon, \upsilon' \in V}\alpha(r, \upsilon)\pi(r, \upsilon, \upsilon') > 0}$, and taking account of the positivity properties of $\varphi$, we can lower bound the probability that, from any $(r,\upsilon)\in \Omega$, the spine can enter $B$. Moreover, on this event, {\color{black}due to (H3)$^*$}, we can also lower bound the probability that the spine immigrates $n_{\texttt{max}}-1$ particles on $m$  (evenly spaced in time) separate occasions, all of which are still inside of $B$ by time $\iota$. The strategy for doing so is to head into $B$ from the given point of issue in $\Omega$ by travelling in a straight line within a small cone of possible velocities (which would be guaranteed to happen within $\iota/2$ units of time), and then for neutrons to cycle around the perimeter of $B$ in an annulus by scattering within a narrow cone of velocities each time; see Fig \ref{Omega}. As such we can provide the lower bound desired in \eqref{inf}. The technical details are tedious and left to the reader. 

\begin{figure}[h!]
\includegraphics[width= 0.5\textwidth]{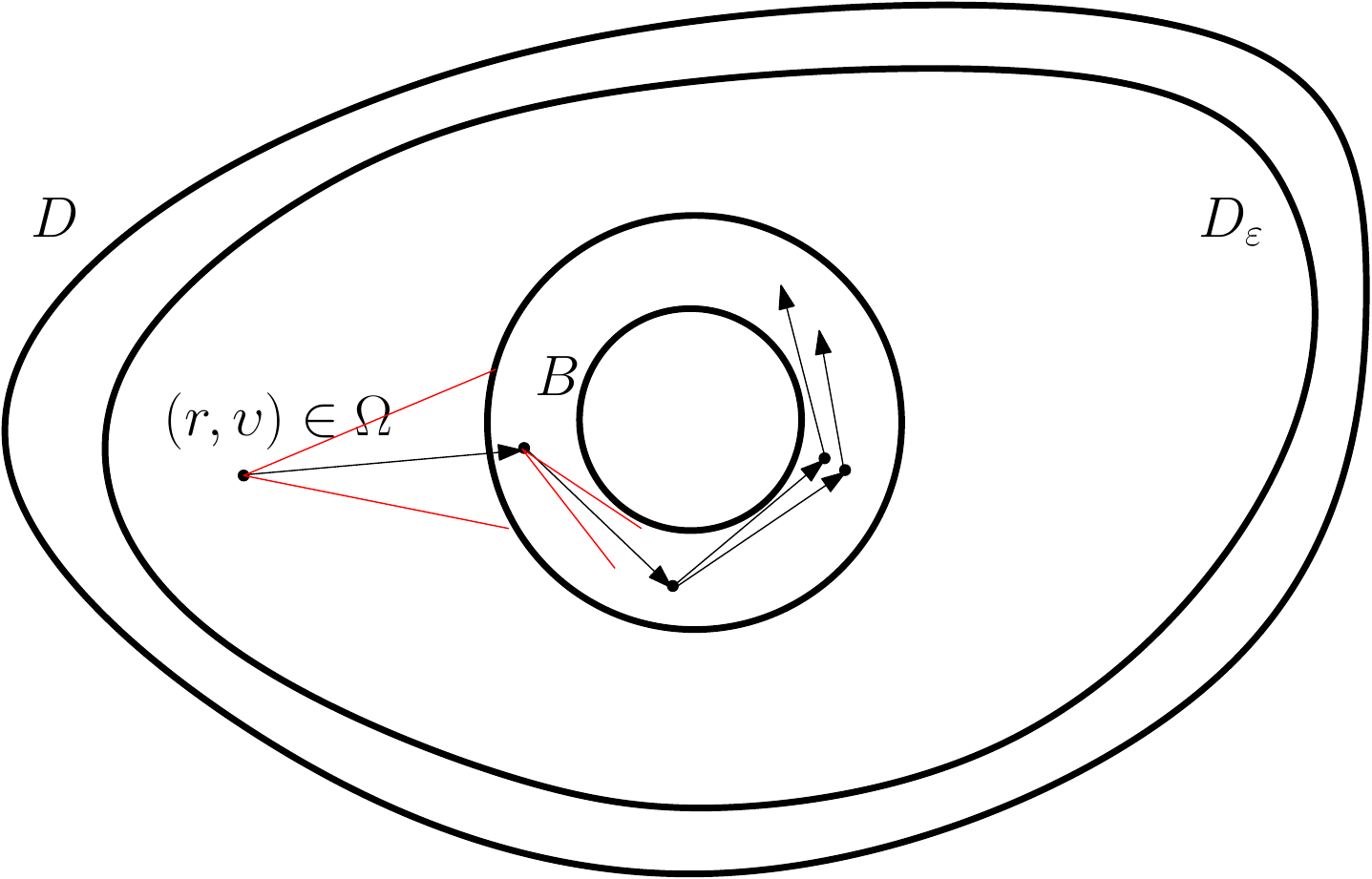}
\caption{\rm There is a uniform lower bound on the probability that the spine issued from $(r,\upsilon)\in \Omega$  heads directly into the annulus contained in $B$  and subsequently immigrates $n_{\texttt{max}}-1$ neutrons on each of $m$ separate occasions,  which then cycle around the annulus, and all this is completed over the time horizon $\iota  ={2\rm diam}(D)/{\texttt v}_{\texttt{min}}$ elapses.}
\label{Omega}
\end{figure}

With \eqref{inf} in hand, we can construct the sequence $(t_n: n\geq 0)$ by defining $t_0 =0$ and subsequently, for $n\geq 1$,
\[
t_n = \inf\{s>t_{n-1}+(10m\times\iota) : (R_s,\Upsilon_s) \in\Omega \}.
\]
Note that since $(R,\Upsilon)$ visits $\Omega$ infinitely often under $\mathbf{P}^\varphi$ we have that $t_n<\infty$,
$\mathbf{P}^\varphi$-almost surely for $n\geq 0$, and $t_n\to\infty$, $\mathbf{P}^\varphi$-almost surely. By applying the strong Markov property at the sequence of times $(t_n,n\geq 0)$, it now follows from \eqref{inf} that, in the spirit of a sequence of independent Bernoulli trials, $\textstyle{\limsup_{n\to\infty}X_{s_n}(D_\varepsilon) \geq  k}$ almost surely with respect to $\mathbb{P}^\varphi$, where $s_n = t_n + (m\times \iota)$. Since the integer $k$ can be chosen arbitrarily large, we also have that $\textstyle{\limsup_{n\to\infty}X_{s_n}(D_\varepsilon) =\infty}$ almost surely with respect to $\mathbb{P}^\varphi$.
 
\smallskip

As $\varphi$ is uniformly bounded  below away form 0 on $D_\varepsilon \times V$ (see Theorem \ref{CVtheorem}), it follows that
\[
W_t \ge cX_t(D_\varepsilon\times V),\qquad t\geq 0,
\]
for some constant $c > 0$. The analysis above, thus shows that  $\textstyle{\mathbb{P}^\varphi_{\delta_{(r, \upsilon)}}(\limsup_{t \to \infty}W_t = \infty) > 0}$, as required, for each $r\in D$, $\upsilon\in V$. 
 \hfill $\square$

\section{Proof of Theorem \ref{zeroset}} 
We need to show that in all three cases, the event $\{\zeta<\infty\}$ agrees almost surely with $\{W_\infty = 0\}$. 
To this end, first note that $\{\zeta<\infty\}\subseteq\{W_\infty =0\}$ and hence 
\begin{equation}
\mathbb{P}_{\delta_{(r,\upsilon)}}(\zeta<\infty) \leq \mathbb{P}_{\delta_{(r,\upsilon)}}(W_\infty =0),
\label{=}
\end{equation}
 for all $r\in D$, $\upsilon\in V$. It thus suffices to show that \eqref{=} is in fact an equality.

\smallskip

We will give two preparatory technical lemmas before coming to the main part of the proof of Theorem \ref{zeroset}. 

\begin{lemma}\label{1stlem}
For all $r\in D$ and $\upsilon\in V$, we have $\P_{\delta_{(r,\upsilon)}}$-almost surely that
\[
  \lim_{n\to\infty}\, \P_{X_n}(\zeta\leq t_0)=\mathbf{1}_{\{\zeta<\infty\}}.
\]

\end{lemma}
\begin{proof}
On the event $\{\zeta<\infty\}$, it is immediate that, for all $r\in D$ and $\upsilon\in V$,
\begin{equation}
  \label{eq2}
\lim_{t\to\infty}  \P_{X_t}(\zeta\leq t_0)= 1
\end{equation}
$\textstyle{\P_{\delta_{(r,\upsilon)}}}$-almost surely.
Let $(T_n)_{n\in\N}$ be any increasing sequence of stopping times.
Using the strong Markov property and \eqref{eq2}, we have that, for
all $n\in\N$,
\[
  \P_{\delta_{(r,\upsilon)}}(\zeta<\infty)=\E_{\delta_{(r,\upsilon)}}\left[\P_{X_{T_n}}(\zeta<\infty)\right]\geq \E_{\delta_{(r,\upsilon)}}\left[\P_{X_{T_n}}(\zeta\leq t_0)\right].
\]
Using this inequality and  Fatou's Lemma, we deduce that
\begin{align*}
  \P_{\delta_{(r,\upsilon)}}(\zeta<\infty)&\geq \liminf_{n\to\infty}\,\E_{\delta_{(r,\upsilon)}}\left[\P_{X_{T_n}}(\zeta\leq t_0)\right]\\
  &\geq \E_{\delta_{(r,\upsilon)}}\left[\liminf_{n\to\infty}\,\P_{X_{T_n}}(\zeta\leq t_0)\right]\\
                              &\geq \E_{\delta_{(r,\upsilon)}}[\mathbf{1}_{\{\zeta<\infty\}}]+\delta\P_{\delta_{(r,\upsilon)}}\left(\zeta=\infty\text{ and }\liminf_{n\to\infty}\, \P_{X_{T_n}}(\zeta\leq t_0)\geq \delta\right).
\end{align*}
It follows that, for all $\delta\in(0,1]$, we have
$\textstyle{\P_{\delta_{(r,\upsilon)}}\left(\zeta=\infty\text{ and
    }\liminf_{n\to\infty}\, \P_{X_{T_n}}(\zeta\leq t_0)\geq
    \delta\right) = 0}$. This implies that, on $\{\zeta = \infty\}$,
$ \textstyle{ \liminf_{n\to\infty} \P_{X_{T_n}}(\zeta\leq t_0) = 0.  }$
Since this is true for any sequence of increasing stopping times, we
deduce that, on $\{\zeta = \infty\}$, 
$ \textstyle{ \limsup_{n\to\infty} \P_{X_{n}}(\zeta\leq t_0) = 0.  } $
Together with \eqref{eq2}, this gives us
\[
  \lim_{n\to\infty}\, \P_{X_n}(\zeta\leq t_0)=\mathbf{1}_{\{\zeta<\infty\}}
\]
 $\P_{\delta_{(r,\upsilon)}}$-almost surely, as required.
\end{proof}

\begin{lemma}\label{2ndlem}For all $r\in D$ and $\upsilon\in V$, on $\{\zeta = \infty\}$, we have $\P_{\delta_{(r,\upsilon)}}$-almost surely  that
\begin{equation}
\label{step1}
\lim_{t\to\infty}\langle \varphi, X_t\rangle =\infty.
\end{equation}
\end{lemma}
\begin{proof}
Recall that $\zeta$ is the time of extinction of the NBP. For any
$r\in D$, $\upsilon\in V$  and $t\geq 0$, we have
\[
  \P_{\delta_{(r,\upsilon)}}(t<\zeta)\leq \E_{\delta_{(r,\upsilon)} }[N_t]=\psi_t[\mathbf{1}_{\{D\times V\}}](r,\upsilon).
\]
Using \eqref{spectralexpsgp}, we
deduce that there exists a $t_0>0$ such that, for all $r\in D$, $\upsilon\in V$,
\[
 \P_{\delta_{(r,\upsilon)}}(t_0<\zeta)\leq   \psi_{t_0}[\mathbf{1}_{\{D\times V\}}](r,\upsilon)\leq 2e^{\lambda_* t_0}\varphi(r,\upsilon).
\]
It is straightforward to show   that $\textstyle{\sup_{x\in E} \P_x(t_0<\zeta)<1}$. Hence,
there exists a constant $c_0\in(0,1)$ such that, uniformly for all $r\in D$ and $\upsilon\in V$,
\[
  \P_{\delta_{(r,\upsilon)}}(t_0<\zeta)\leq c_0\wedge \Big[2e^{\lambda_* t_0}\varphi(r,\upsilon)\Big].
\]
Using the branching property, we deduce that, for all $\textstyle{\mu=\sum_{i =1}^n \delta_{(r_i, \upsilon_i)}\in {\cal M}(D\times V)}$,
\[
  \P_{\mu}(\zeta\leq t_0)\geq  \prod_{i=1}^n \left(1-c_0\wedge \Big[2e^{\lambda_* t_0}\varphi(r_i, \upsilon_i)\Big]\right).
\]
Now, using the Lemma \ref{1stlem} we have, taking the limit in $n \in\mathbb{N}$,
we obtain
\begin{equation}
 \mathbf{1}_{\{\zeta<\infty\}} =   \limsup_{n\rightarrow \infty} \P_{X_n}(\zeta\leq t_0) \geq  \limsup_{n\rightarrow \infty} \prod_{i=1}^{N_n} \left(1-c_0\wedge
    \Big[2e^{\lambda_0
      t_0}\varphi(r_i(n), \upsilon_i(n))\Big]\right),
      \label{takelogs}
\end{equation}
where we have used the notation from \eqref{randommeasure}.
Since $c_0<1$, taking logarithms in \eqref{takelogs} and using $\log x\leq  x-1$, we deduce that, on 
$\{\zeta=\infty\}$, again taking limits on $\N$,
\[
 \lim_{n\to\infty} \sum_{i=1}^{N_n} \left(c_0\wedge\Big[2e^{\lambda_0 t_0}\varphi(r_i(n), \upsilon_i(n))\Big]\right)= \infty
\]
and hence the statement of the lemma follows. 
\end{proof}

Let us now return to the proof of Theorem \ref{zeroset}. 
First we consider the setting that $\lambda_*\leq 0$. Noting that, up to a normalising constant  $W_t  =
{\rm e}^{-\lambda_* t}\langle\varphi,X_t\rangle\geq \langle\varphi,X_t\rangle$, as $W$ is almost surely convergent,  the conclusion of Lemma \ref{2ndlem} forces us to deduce that $\{\zeta<\infty\}$ almost surely in order to avoid a contradiction. Hence, from \eqref{=}, we have that $\{\zeta<\infty\} =\{W_\infty = 0\}$ and both occur with probability one (irrespective of the starting configuration of $X$).
\smallskip

Next we consider the setting that $\lambda_*>0$.  Due to our assumptions and the boundedness of $\varphi$, we have uniformly, for all $r\in D$, $\upsilon \in D$ and  all times $t$ such that there is a discontinuity in $W$, 
$|W_t-W_{t-}|$ is uniformly bounded by some constant $M>0$,
$\P_{\delta_{(r,\upsilon)}}$ almost surely. 
Defining the stopping time
$T_1=\inf\{t\geq 0,W_t\geq 1\}$, using the fact that $W$ is a non-negative,  $L^1$, and hence uniformly integrable martingale, and using Doob's Optional Stopping Theorem, we deduce that, for all $r\in D$, $\upsilon \in V$,
\begin{align*}
  1&=\E_{\delta_{(r,\upsilon)}}[W_{T_1}]=\E_{\delta_{(r,\upsilon)}}[W_{ T_1}\mathbf{1}_{\{W_\infty>0\}}]\leq \frac{1}{\varphi(r,\upsilon)}\E_{\delta_{(r,\upsilon)}}((1+M)\mathbf{1}_{\{W_\infty>0\}}).
\end{align*}
It follows that, for all $r\in D$ and $\upsilon \in V$, $\P_{\delta_{(r,\upsilon)}}(W_\infty>0)\geq \varphi(r,\upsilon)/M$. Hence that
there exists $c_1>0$ such that
\[
  \P_{\delta_{(r,\upsilon)}}(W_\infty=0)\leq 1-c_1\varphi(r,\upsilon), \qquad r\in D, \upsilon\in V.
\]
Now, using the branching property, we obtain 
 for all $\textstyle{\mu=\sum_{i =1}^n \delta_{(r_i, \upsilon_i)}\in {\cal M}(D\times V)}$,
 \[
  \ln\,\P_{\mu}(W_\infty=0)\leq \sum_{i=1}^n
  \ln\left(1-c_1\varphi(x_i)\right)\leq -c_1\sum_{i=1}^n\varphi(x_i).
\]
Due to the conclusion of Lemma \ref{2ndlem} we have, on $\{\zeta = \infty\}$,
\[
  \limsup_{n\to\infty}\,\ln \P_{X_n}(W_\infty=0)\leq -c_1\lim_{n\rightarrow\infty}\langle\varphi, X_n\rangle= -\infty
\]
With the upper bound of any probability being unity, we can thus write 
\[
  \limsup_{n\to\infty} \P_{X_n}(W_\infty=0)\leq \mathbf{1}_{\{\zeta<\infty\}}.
\]

Markov's property now entails, for $r\in D$, $\upsilon\in V$,
\[
  \P_{\delta_{(r,\upsilon)}}(W_\infty=0)=\E_{\delta_{(r,\upsilon)}}\left[\P_{X_n}(W_\infty=0)\right],\quad\forall n\in\N,
\]
so that, by the Reverse Fatou's Lemma, 
\begin{align*}
  \P_{\delta_{(r,\upsilon)}}(W_\infty=0)&=\limsup_{n\to\infty}\E_{\delta_{(r,\upsilon)}}\left[\P_{X_n}(W_\infty=0)\right]\\
  &\leq \E_{\delta_{(r,\upsilon)}}\left[\limsup_{n\to\infty}\P_{X_n}(W_\infty=0)\right]\\
                           &\leq \E_{\delta_{(r,\upsilon)}}\left[\mathbf{1}_{\{\zeta<\infty\}}\right]\\
                           &=\P_{\delta_{(r,\upsilon)}}(\zeta<\infty).
\end{align*}
Together with \eqref{=}, this completes the proof of the theorem. \hfill$\square$

\section{Proof of Corollary \ref{L2}}
Doob's martingale inequality ensures that, for $\mu\in \mathcal{M}(D\times V)$
\[
\mathbb{E}_\mu[(\sup_{t\geq 0}W_t)^2] \leq \liminf_{s\to\infty} 4\mathbb{E}_\mu[(W_s)^2].
\]
Showing that the right-hand side above is finite is sufficient to obtain $L_2(\mathbb{P})$ convergence.
Note, however, that 
$
\mathbb{E}_\mu[(W_s)^2] = \mathbb{E}^\varphi_{\mu}[W_t]
$, $t\geq 0$, and hence, from \eqref{EIS}, the desired upper bound is proved.
 \hfill $\square$

\section*{Acknowledgements} This research was born out of a surprising connection that was made at the problem formulation ``Integrative Think Tank'' as part of the EPSRC Centre for Doctoral Training SAMBa in the summer of 2015. AEK and EH are indebted to Professor Paul Smith and Dr. Geoff Dobson from the ANSWERS modelling group at Wood for the extensive discussions at their offices in Dorchester as well as for giving  permission to use these images in Figure \ref{reactorcore}, which were constructed with Wood nuclear software ANSWERS.  We would also like to thank Alex Cox, Simon Harris and Minmin Wang for helpful discussions, as well as Jean Bertoin and Alex Watson for several interesting discussions on growth-fragmentation equations. Finally, we are extremely grateful to the assistance of an AE and an anonymous referee, which have helped us shape the presentation significantly.

\bibliography{references}{}
\bibliographystyle{plain}

\begin{table}[h]
\begin{tabular}{l l l}
&\multicolumn{1}{c}{\bf Glossary of some commonly used notation}&\\
&\multicolumn{1}{c}{\bf (Th. = Th., a. = above, b. = below)}&\\
&&\\
\hline
Notation & Description & Introduced\\
\hline
&&\\
$(\psi_t, t\geq 0)$ & Solution to mild NTE/NBP expectation 
semigroup& \eqref{semigroup}, \eqref{mild}\\
$D$ and $V$ & Physical and velocity domain & \S \ref{intro}\\
$\sigma_{\texttt{s}}$, $\sigma_{\texttt{f}}$ and $\sigma$ & Scatter, fission and total cross-sections 
&b. \eqref{bNTE}\\
$\pi_{\texttt s}$ and $\pi_{\texttt f}$ & Scatter and fission kernels & b. \eqref{bNTE}\\
$\bS$ and $\bF$ & Scatter and fission operators & \eqref{S}, \eqref{F}\\
$\mathcal{P}_{(r,\upsilon)}$  & Offspring law  of $X$ when parent at  $(r,\upsilon)\in D\times V$ & \eqref{Erv}\\
$((r_i,\upsilon_i), i = 1,\cdots, N)$& Position and number of offspring positions  of a family in $X$ & a. \eqref{PP}\\
$(X, \mathbb{P}_\mu)$ & NBP when issued from $\mu$ & 
\S \ref{PP}\\
$\U_t$& Linear advection semigroup & b. \eqref{branchingsemigroup}, \eqref{mild}\\
$G_{\texttt{f}}$&Branching generator of $(X, \mathbb{P}_\mu)$&\eqref{Gf}\\
$(u_t, t\geq 0)$ & Non-linear semigroup of $X$& \eqref{ut}\\
$\hat\U_t$& Non-linear advection semigroup & \eqref{ut}, \eqref{tildeU}\\
$n_\texttt{max}$ & Maximum number of neutrons in a fission event& b. \eqref{Erv}, (H4)\\
$\lambda_*$, $\varphi$ and $\tilde\varphi$& Leading eigenvalue, right- and left-eigenfunctions &Th. \ref{CVtheorem}\\
$(W_t, t\geq 0)$& Additive martingale &\eqref{mgdef}\\
$\zeta$& Extinction time &Th. \ref{Kesten}\\

&&\\
\hline
&&\\

$((R,\Upsilon), \mathbf{P})$& Many-to-one NRW  & Lemma \ref{NRWrep}\\
$\alpha$ and $\pi$ &Scatter rate and kernel for many-to-one NRW  & \eqref{alpha}, \eqref{aalone}\\
$\tau^D$ & First exit time of spatial component of $\alpha\pi$-NRW from $D$& Lemma \ref{NRWrep}\\
$((R,\Upsilon),  \mathbf{P}^{\dagger})$& Killed $\alpha\pi$-NRW&
\eqref{boldPdagger}\\
$\texttt{P}^\dagger$& Semigroup of killed $\alpha\pi$-NRW& \eqref{boldPdagger}\\

$\beta$ & Many-to-one potential &\eqref{betadef}, \eqref{phi}\\
$\texttt{k}$&Killing time of $\alpha\pi$-NRW& \eqref{kill}\\
$(J_k, k\geq 1)$ &Ordered jump times of killed $\alpha\pi$-NRW &b. \eqref{minor}\\

&&\\
\hline
&&\\
$(X,\mathbb{P}^\varphi_\mu)$& NBP after  change of measure with $W$ when issued from $\mu$ &\eqref{mgCOM}\\
$G^\varphi_{\texttt{f}}$&Branching generator of $(X,\mathbb{P}^\varphi_\mu)$&\eqref{Gphif}, \eqref{spinesg}\\
$({u}^\varphi_t,t\geq 0)$ &Non-linear semigroup of $(X, {\mathbb{P}}^\varphi_{\mu})$&
\eqref{uphi}\\
$({\psi}^\varphi_t,t\geq 0)$ &Linear semigroup of $(X,\mathbb{P}^\varphi_\mu)$&\eqref{psiphi}\\
$(X^\varphi , \tilde{\mathbb{P}}^\varphi_{\mu})$
& Dressed spine when issued from configuration $\mu$ &\eqref{dressedprocess}\\
$(\tilde{u}^\varphi_t,t\geq 0)$ &Non-linear semigroup of $(X^\varphi , \tilde{\mathbb{P}}^\varphi_{\mu})$&
\eqref{uttild}\\

$((R^\varphi, \Upsilon^\varphi), \tilde{\mathbf{P}}^\varphi )$&Marginal of $\tilde{\mathbb{P}}^\varphi $ giving law of spine NRW & Th. \ref{pathspine}, \eqref{NRWCOM}\\
$\alpha^\varphi$ and $\pi^\varphi$& Scatter rate and kernel of auxiliary NRW&\eqref{spineap}\\
$\mathcal{P}^\varphi_{(r,\upsilon)}$ & Scattering of velocities along the spine& \eqref{pointprocessCOM}\\
$\mathbf{P}^\varphi$&Law of $\alpha^\varphi\pi^\varphi$-NRW that agrees with $ \tilde{\mathbf{P}}^\varphi$ &a. \eqref{spineap}\\

&&\\
\hline
\end{tabular}
\label{table-notation}
\end{table}

\end{document}